\documentclass[11pt, reqno]{amsart}
\usepackage{textcmds} 
\usepackage{amsmath, amssymb, amsfonts, amstext, verbatim, amsthm, mathrsfs}
\usepackage{microtype}
\usepackage[all]{xy}
\usepackage[usenames]{color}
\usepackage{aliascnt}
\usepackage{enumitem}
\usepackage{xspace}
\usepackage{amsfonts}
\usepackage{amssymb}
\usepackage[centertags]{amsmath}
\usepackage{amsthm}
\usepackage{rotating}
\usepackage[margin=3.5cm]{geometry}
\usepackage{dsfont}
\usepackage{bm}
\usepackage{color}
\usepackage{subfigure}
\usepackage{amsmath}
\usepackage{array}
\usepackage[all]{xy}
\usepackage{euscript}
\usepackage[T1]{fontenc}
\usepackage{mathbbol}
\usepackage{nccmath}
\usepackage{marginnote}
\usepackage{stmaryrd}
\usepackage{youngtab}
\usepackage{tikz}
\usepackage{blkarray}
 \usepackage{booktabs}

\usepackage{graphics,graphicx}

\let\oldtocsection=\tocsection
\let\oldtocsubsection=\tocsubsection

\renewcommand{\tocsection}[2]{\hspace{0em}\oldtocsection{#1}{#2}}
\renewcommand{\tocsubsection}[2]{\hspace{1em}\oldtocsubsection{#1}{#2}}

\addtocontents{toc}{\protect\setcounter{tocdepth}{2}}

\usepackage[backref,colorlinks=true,linkcolor=black,citecolor=blue,urlcolor=blue,citebordercolor={0 0 1},urlbordercolor={0 0 1},linkbordercolor={0 0 1}]{hyperref} 

\tikzset{node distance=3cm, auto}

\makeatletter
\def\@secnumfont{\bfseries}
\def\section{\@startsection{section}{1}%
  \z@{.7\linespacing\@plus\linespacing}{.5\linespacing}%
  {\normalfont\Large\bfseries}}
\def\subsection{\@startsection{subsection}{2}%
  \z@{.5\linespacing\@plus.7\linespacing}{-.5em}%
  {\normalfont\large\bfseries}}
  \def\subsubsection{\@startsection{subsubsection}{3}%
  \z@{.5\linespacing\@plus.7\linespacing}{-.5em}%
  {\normalfont\bfseries}}
\makeatother

\theoremstyle{plain}
\newtheorem{thm}{Theorem}[section]
\newtheorem{lemma}[thm]{Lemma}
\newtheorem{prop}[thm]{Proposition}
\newtheorem{cor}[thm]{Corollary}
\newtheorem{question}[thm]{Question}
\newtheorem{problem}[thm]{Problem}

\newtheorem{observation}[thm]{Observation}

\newtheorem{claim}[thm]{Claim}
\newtheorem{addendum}[thm]{Addendum}
\newtheorem{example}[thm]{Example}
\newtheorem{remark}[thm]{Remark}

\newtheorem{definition}[thm]{Definition}

\theoremstyle{definition}

\numberwithin{figure}{section}
\numberwithin{table}{section}

\newcommand{\NI}{{\noindent}}

\newcommand{\la}{{\lambda}}

\newcommand{\Z}{\mathbb{Z}}

\newcommand{\F}{\mathbb{F}}
\newcommand{\R}{\mathbb{R}}
\newcommand{\Q}{\mathbb{Q}}
\newcommand{\C}{\mathbb{C}}
\newcommand{\CP}{\mathbb{CP}}
\newcommand{\D}{\mathbb{D}}
\newcommand{\K}{\mathbb{K}}

\renewcommand{\sc}{\op{SC}}
\newcommand{\eps}{\varepsilon}

\newcommand{\calA}{\mathcal{A}}
\newcommand{\calL}{\mathcal{L}}

\renewcommand{\bar}{\mathcal{B}}

\newcommand{\bdy}{\partial}

\newcommand{\Li}{\mathcal{L}_\infty}

\newcommand{\calM}{\mathcal{M}}

\newcommand{\wh}{\widehat}
\newcommand{\wt}{\widetilde}

\newcommand{\ovl}{\overline}
\newcommand{\op}[1]{{\operatorname{#1}}}
\newcommand{\e}{\eps}

\newcommand{\sh}{\op{SH}}

\newcommand{\chlin}{\op{C}\mathbb{H}_{\op{lin}}}
\newcommand{\cclin}{\op{CH}_{\op{lin}}}

\newcommand{\whcclin}{\wh{\op{CH}}_{\op{lin}}}

\newcommand{\cz}{{\op{CZ}}}
\newcommand{\ind}{\op{ind}}

\newcommand{\Op}{\mathcal{O}p}
\newcommand{\nil}{\varnothing}
\newcommand{\sss}{\vspace{2.5 mm}}
\renewcommand{\leqq}{\preceq}

\newcommand{\aug}{\e}

\newcommand{\triv}{\tau_0}
\newcommand{\m}{\mathfrak{m}}

\newcommand{\gw}{\op{GW}}
\renewcommand{\lll}{\Langle}
\newcommand{\rrr}{\Rangle}

\newcommand{\T}{\mathcal{T}}
\newcommand{\TT}{\T^{\bullet}}

\newcommand{\simp}{s}

\renewcommand{\bar}{\mathcal{B}}
\newcommand{\id}{\mathbb{1}}

\newcommand{\shookrightarrow}{\overset{S}\hookrightarrow}
\newcommand{\es}{\overset{L}\hookrightarrow}
\newcommand{\ws}{\overset{W}\hookrightarrow}

\newcommand{\calS}{\mathcal{S}}

\renewcommand{\a}{\mathfrak{a}}
\renewcommand{\b}{\mathfrak{b}}
\newcommand{\mr}{\mathring}
\newcommand{\CC}{\mathcal{C}}
\newcommand{\crit}{\op{crit}}
\newcommand{\ww}{\mathbb{w}}
\newcommand{\formal}{F}
\newcommand{\flex}{\op{Flex}}
\newcommand{\wein}{\mathfrak{W}}
\newcommand{\G}{\mathbb{G}}
\renewcommand{\F}{\mathbb{F}}
\newcommand{\I}{\mathbb{I}}
\newcommand{\calJ}{\mathcal{J}}
\newcommand{\Int}{\op{Int}\,}
\newcommand{\wfuk}{\mathcal{W}}
\newcommand{\vk}{\varkappa}

\makeatletter
\newcommand{\dashover}[2][\mathop]{#1{\mathpalette\df@over{{\dashfill}{#2}}}}
\newcommand{\fillover}[2][\mathop]{#1{\mathpalette\df@over{{\solidfill}{#2}}}}
\newcommand{\df@over}[2]{\df@@over#1#2}
\newcommand\df@@over[3]{%
  \vbox{
    \offinterlineskip
    \ialign{##\cr
      #2{#1}\cr
      \noalign{\kern1pt}
      $\m@th#1#3$\cr
    }
  }%
}
\newcommand{\dashfill}[1]{%
  \kern-.5pt
  \xleaders\hbox{\kern.5pt\vrule height.4pt width \dash@width{#1}\kern.5pt}\hfill
  \kern-.5pt
}
\newcommand{\dash@width}[1]{%
  \ifx#1\displaystyle
    2pt
  \else
    \ifx#1\textstyle
      1.5pt
    \else
      \ifx#1\scriptstyle
        1.25pt
      \else
        \ifx#1\scriptscriptstyle
          1pt
        \fi
      \fi
    \fi
  \fi
}
\newcommand{\solidfill}[1]{\leaders\hrule\hfill}
\makeatother

\begin{document}

\title{On the embedding complexity of Liouville manifolds}

\begin{abstract}
We define a family  of symplectic invariants which obstruct exact symplectic embeddings
between Liouville manifolds, using the general formalism of linearized contact homology and its $\Li$ structure.
As our primary application, we investigate embeddings between normal crossing divisor complements in complex projective space, giving a complete characterization in many cases. 
Our main embedding results are deduced explicitly from pseudoholomorphic curves, without appealing to Hamiltonian or virtual perturbations.
\end{abstract}

\author{Sheel Ganatra and Kyler Siegel}

\date{\today}

\maketitle

\tableofcontents

\section{Introduction}

\subsection{Context and motivation}\label{subsec:context}

A {\em Liouville domain} is a compact manifold-with-boundary equipped with a primitive one-form $\la$ such that $\omega := d\la$ is symplectic and $\la$ restricts to a positive contact form along the boundary. 
Liouville domains form a nice class of open symplectic manifolds which naturally arise in many geometric contexts, including:
\begin{itemize}
\item unit cotangent disk bundles of closed Riemannian manifolds 
\item sufficiently large compact pieces of smooth complex affine varieties
\item regular sublevel sets of Stein manifolds.
\end{itemize}
The principal goal of this paper is to develop tools to understand when one Liouville domain is ``larger'' or ``more complicated'' than another.
Specifically, given two Liouville domains $(X,\la)$ and $(X',\la')$ of the same dimension, we seek to understand when there is a {\em Liouville embedding} $X \es X'$.
This consists of a smooth embedding $\iota: X \hookrightarrow X'$ such that $\iota^*\la'$ agrees with $\la$ up to some positive scaling factor and the addition of an exact one-form (see \S\ref{subsubsec:embeddings}). The existence of a Liouville embedding $X \es X'$ is a qualitative notion, depending only on $X$ and $X'$ up to Liouville homotopy.
Equivalently, by attaching an infinite cylindrical end, any Liouville domain $(X,\la)$ can be completed to a (finite type) {\em Liouville manifold} $(\wh{X},\wh{\la})$ (e.g. the completion of a cotangent disk bundle is the full cotangent bundle), and the existence of a Liouville embedding $X \es X'$ is equivalent to having a smooth (but not necessarily proper) embedding $\iota: \wh{X} \hookrightarrow \wh{X'}$ such that $\iota^*(\wh{\la'}) - \wh{\la}$ is exact.

One reason for interest in Liouville embeddings comes from their connection to exact (compact) Lagrangian submanifolds. An {\em exact Lagrangian submanifold} of $(X,\la$) is a half-dimensional submanifold $L \subset X$ equipped with a function $f$ such that $\la|_{L} = df$ (so in particular $\omega|_L \equiv 0$).
The study of exact Lagrangian embeddings has played a prominent role in the symplectic topology literature, going back to Gromov's theorem \cite{gromov1985pseudo} that there are no closed exact Lagrangians in $\C^n$, and Arnold's ``nearby Lagrangian conjecture'' stating that there is a unique closed exact Lagrangian in the cotangent bundle of a closed manifold up to Hamiltonian isotopy.
By a version of the Weinstein neighborhood theorem, a given $L$ admits an exact Lagangian embedding into $(X,\la)$ if and only if there is a Liouville embedding $D^*_\eps L \es X$,
where $D^*_\eps L$ denotes the cotangent $\eps$-disk bundle of $L$ for some Riemannian metric and $\eps > 0$ sufficiently small.
As it turns out, most known examples of Liouville domains (such as $D^*_\eps L$) are a fortiori {\em Weinstein domains}, meaning they carry Morse functions suitably compatible with the Liouville structure.
A Weinstein domain deformation retracts onto its skeleton, which is an isotropic (but possibly singular)
closed subset; in the case of $D^*_\eps L$ with its canonical Liouville structure, the resulting skeleton is $L$ itself.
Hence, we intuitively view a more general Weinstein domain as the cotangent disk bundle of its singular skeleton, and Liouville embeddings of Weinstein domains as singular generalizations of exact Lagrangian embeddings.  

Liouville embeddings also constitute a primary class of morphisms under which functoriality holds
for the most widely studied symplectic invariants of Liouville domains, for instance symplectic cohomology, wrapped Fukaya categories and various other invariants built from the theory of pseudoholomorphic curves (see e.g. \cite{Seidel_biased_view,abouzaid2010open}).
Given a Liouville domain $X$ and a chosen ground ring $\K$, its symplectic cohomology $\sh(X)$ is, among other things, a unital $\K$-algebra whose isomorphism type depends only on $\wh{X}$ up to symplectomorphism.\footnote{In more detail, any two finite type Liouville manifolds which are symplectomorphic are also exact symplectomorphic by \cite[Lem. 11.2]{cieliebak2012stein}, and hence their symplectic cohomologies are isomorphic by a consequence of Viterbo functoriality as in \cite[\S7b]{Seidel_biased_view}.}
A Liouville embedding $X \es X'$ induces a transfer map $\sh(X') \rightarrow \sh(X)$ of unital $\K$-algebras.
We have also $\sh(B^{2n}) = 0$ and $\sh(D^*Q) \cong H_*(\calL Q)$,\footnote{More precisely, this isomorphism always holds over $\K=\Z/2$, and it holds for more general $\K$ if $Q$ is Spin, whereas the general case necessitates using a twisted coefficient system.} where $\calL Q$ denotes the free loop space of $Q$.
Combining these properties gives an elegant proof of Gromov's theorem as follows. Given a hypothetical exact Lagrangian $L \subset \C^n$, the transfer map $\sh(\C^n) \rightarrow \sh(D^*_\eps L)$ is necessarily the zero map. This forces the unit in $\sh(D^*_\eps L)$ to be zero, and hence $H_*(\calL L) =0$, but this is never the case.

The argument in the preceding paragraph can be formalized into the following simple but surprisingly powerful observation:
\begin{observation}\label{obs:vanishingSH}
Given a Liouville embedding $X \es X'$, if $\sh(X') = 0$, then we must also have $\sh(X) = 0$.
\end{observation}
\NI For example, every Weinstein domain is diffeomorphic to one which is {\em flexible} \cite[\S11.8]{cieliebak2012stein}, with the ball and more generally subcritical Weinstein domains arising as special cases.
Since flexible Weinstein domains have vanishing symplectic cohomology (see \cite[\S3.3]{murphy2018subflexible}), we immediately extend Gromov's theorem to find that there are no exact Lagrangians in any flexible Weinstein domain.

\sss

However, Observation~\ref{obs:vanishingSH} is rather insufficient in situations where $\sh(X)$ and $\sh(X')$ are both vanishing or both nonvanishing, especially since the transfer map is generally neither injective nor surjective.
For instance, let $X_k^{2n}$ denote the complement of a small neighborhood of $k$ generic hyperplanes in $\CP^n$. As we explain in \S\ref{subsec:div_compl}, $X_k^{2n}$ has a canonical Weinstein structure.
Concretely, we can ask:
\begin{problem}\label{prob:hyperplane_comp}
For which $k,k' \in \Z_{\geq 1}$ is there a Liouville embedding $X_k^{2n} \es X_{k'}^{2n}$?
\end{problem}
\NI We make a few preliminary observations:
\begin{itemize}
\item For $k \leq n$, $X_k^{2n}$ is subcritical, in fact its symplectic completion is $(\C^*)^{k-1}\times \C^{n-k+1}$, and hence we have $\sh(X_k^{2n}) = 0$. From this it is straightforward to produce a Liouville (in fact Weinstein) embedding of $X_{k'}^{2n}$ into $X_k^{2n}$ whenever $k,k' \leq n$.
\item There is a spectral sequence \cite{ganatra2020symplectic, mcleanslides} for computing the symplectic cohomology of any ample simple normal crossing divisor complement from the ordinary cohomology of various combinatorial strata defined by the normal crossings configuration. When $k \geq n+1$, this spectral sequence degenerates (see \cite[Thm 1.4 and Example 5.1]{ganatra2020symplectic}) and in particular $\sh(X_k^{2n}) \neq 0$ for any coefficient field $\K$. Therefore, by Observation~\ref{obs:vanishingSH} there is no Liouville embedding $X_{k'}^{2n} \es X_{k}^{2n}$ when $k' \geq n+1$ and $k \leq n$.

\item There is a Weinstein embedding $X_k^{2n} \ws X_{k'}^{2n}$ whenever $k < k'$ (see \S\ref{subsubsec:constructions}).
There is also a symplectic (and in particular smooth) embedding from $X_k^{2n}$ into $X_{k'}^{2n}$ for $k > k'$, given by adding back in some of the hyperplanes.
\end{itemize}
We are left to wonder about Liouville embeddings of $X_{k'}^{2n}$ into $X_k^{2n}$ in the case $k' > k \geq n+1$, and more generally:
\begin{question}\label{question:complexity}
Is there a natural notion of ``complexity'' of Liouville domains, such that more complicated Liouville domains cannot Liouville embed into less complicated ones?
\end{question}

We point out right away that there are already several interesting partial answers to Question~\ref{question:complexity} appearing in the literature, though these approaches are not sufficient (to our knowledge) to solve Problem~\ref{prob:hyperplane_comp} (see the discussion in \S\ref{subsec:obs_higher}).
For one, Abouzaid--Seidel \cite{abouzaid2010altering} introduced a ``homological recombination'' construction which modifies a given Weinstein domain so as to kill its symplectic cohomology when the characteristic of $\K$ belongs to a chosen set of primes, and otherwise leaving $\sh$ intact. 
As a corollary, by appealing to Observation~\ref{obs:vanishingSH}, we can find e.g. an infinite sequence of Weinstein domains $W_1,W_2,W_3,\dots$, all diffeomorphic to $B^6$, such that $W_i$ does not Liouville embed into $W_j$ unless $i \leq j$. According to \cite{lazsyl}, we can also arrange that $W_i$ does admit a Weinstein embedding into $W_j$ for $i < j$.

In a different direction, which lies closer to the heart of this paper, we have the notion of {\em dilation} \cite{seidel2012symplectic}, and its generalizations and cousins (see e.g. \cite{zhao2016periodic,Zhou_semidilation,li2019exact}).
The basic observation here is that symplectic cohomology $\sh$ has an $S^1$-equivariant analogue $\sh_{S^1}$, which enjoys a more refined version of 
Observation~\ref{obs:vanishingSH} (see \S\ref{subsec:obs_cyls})
For example, \cite{Zhou_semidilation} constructs Brieskorn varieties having $k$-dilations but not $(k-1)$-dilations for any fixed $k \in \Z_{\geq 0}$, and this translates into a non-existence result for Liouville embeddings.
Essentially the same structure is exploited by Gutt--Hutchings to construct symplectic capacities in \cite{Gutt-Hu}. In fact, although these capacities were developed as quantitative invariants, they become qualitative if one only remembers whether each capacity is finite or infinite. 
As we will explain, these notions are closely related to the linear version of the invariants we define in \S\ref{sec:invariants}, whereas our more general invariants parallel the higher symplectic capacities defined in \cite{HSC}.

\sss

Problem~\ref{prob:hyperplane_comp} naturally fits into a wider framework as follows.
Fix a positive integer $n$.
Let $\vec{d}  = (d_1,\dots,d_k) \in \Z^k_{\geq 1}$ denote a tuple of positive integers for some $k \in \Z_{\geq 1}$. 
We denote by $X^{2n}_{\vec{d}}$ the natural Weinstein domain given by the complement of a small neighborhood of $k$ smooth hypersurfaces of degrees $d_1,\dots,d_k$ in general position in $\CP^n$.
Notably, this depends only on $\vec{d}$ and is independent of all other auxiliary choices up to Weinstein deformation equivalence (see \S\ref{subsec:div_compl}).
Similarly, put $\vec{d'} = (d_1',\dots,d_{k'}') \in \Z^{k'}_{\geq 1}$ for some $k' \in \Z_{\geq 1}$, and denote the corresponding Weinstein domain by $X^{2n}_{\vec{d'}}$.
Note that with this notation we have $X_k^{2n} = X^{2n}_{(\underbrace{1,\dots,1}_k)}$.

\begin{problem}\label{prob:hyp}
For which tuples $\vec{d}$ and $\vec{d'}$ is there a symplectic / Liouville / Weinstein embedding of $X^{2n}_{\vec{d}}$ into $X^{2n}_{\vec{d'}}$?
\end{problem}

\subsection{Main results}

The following result is representative of the techniques of this paper:
\begin{thm}\label{thm:main_liouville}
Fix $n \in \Z_{\geq 1}$ and tuples $\vec{d} = (d_1,\dots,d_k) \in \Z_{\geq 1}^k$ and $\vec{d'} = (d_1',\dots,d_{k'}') \in \Z_{\geq 1}^{k'}$ with $\sum\limits_{i=1}^kd_i,\sum\limits_{i=1}^{k'}d_i' \geq n+1$.
Assume that we have $\sum\limits_{i=1}^{k'} d_i' < 2\sum\limits_{i=1}^{k}d_i - n -1$.
Then there is a Liouville embedding $X^{2n}_{\vec{d}} \es X^{2n}_{\vec{d'}}$ if and only if $\vec{d} \leqq \vec{d'}$.
\end{thm}
\NI We will deduce the obstructive part from a more general framework, a synopsis of which is given in \S\ref{subsubsec:obstructions}. The relevant embedding constructions, and in particular the definition of the combinatorial partial order ``$\leqq$'' are summarized in \S\ref{subsubsec:constructions}.

As an illustrative example, combining the above theorem with the observations from in the previous subsection solves Problem~\ref{prob:hyperplane_comp}:
\begin{cor}\label{cor:intro_X_k_emb}
For any $n \in \Z_{\geq 1}$ and $k' > k \geq n+1$, there is no Liouville embedding $X_{k'}^{2n} \es X_{k}^{2n}$.
\end{cor}

\begin{example}
Consider the case $n=1$ of Problem~\ref{prob:hyperplane_comp}, so that $X_k^{2}$ is the two-sphere minus $k$ open disks.
For $k < k'$, there is a Liouville embedding $X_k^2 \es X_{k'}^2$, given by iteratively attaching Weinstein $1$-handles.
By contrast, there is no Liouville embedding $\iota: X_{k'}^2 \es X_k^2$.
Indeed, given such an embedding, the complement $X_k^2 \setminus \iota(X_{k'}^2)$ would necessarily have at least one component which is disjoint  from $\bdy X_{k}^2$, and this violates Stokes' theorem (c.f. Example~\ref{ex:surfaces}).
\end{example}

\subsubsection{Obstructions}\label{subsubsec:obstructions}

In \S\ref{sec:invariants}, we define for each $m \in \Z_{\geq 0}$ a symplectomorphism invariant $\G \lll \T^m p \rrr$ of Liouville domains which takes values in positive integers and is monotone with respect to Liouville embeddings, i.e.
\begin{align*}
X \es X' \;\;\;\Longrightarrow\;\;\; \G\lll \T^m p \rrr(X) \leq \G\lll \T^m p\rrr(X')
\end{align*}
Heuristically, $\G \lll \T^m p \rrr(X)$ corresponds to the least number of positive ends of a rigid rational curve in $X$ which passes through a generic point $p$ and is tangent to order $m$ to a generic local divisor at $p$.
Whereas this would generally depend on the choices involved, the more precise definition of $\G\lll \T^m p \rrr(X)$ is based on the $\Li$ structure on linearized contact homology $\chlin(X)$.
In \S\ref{sec:computationsI}, we compute this invariant for divisor complements in projective space:
\begin{thm}\label{thm:main_G_computation}
For $n \in \Z_{\geq 1}$ and $\vec{d} = (d_1,\dots,d_k) \in \Z^k_{\geq 1}$ with $\sum_{i=1}^k d_i \geq n+1$, we have 
\begin{align}
\G \lll \T^{n-1} p\rrr(X^{2n}_{\vec{d}}) = \sum_{i=1}^k d_i.  
\end{align}
\end{thm}
\NI As an immediate consequence:
\begin{cor}\label{cor:G_obstructions}
For $n \in \Z_{\geq 1}$, we have $X^{2n}_{(d_1,\dots,d_k)} \not\es X^{2n}_{(d_1',\dots,d_{k'}')}$ whenever $\sum_{i=1}^k d_i > \sum_{i=1}^{k'} d_i' \geq n+1$.
\end{cor}
\NI It is worth emphasizing that there is a symplectic (and in particular smooth) embedding from $X_{\vec{d}}^{2n}$ into $X_{\vec{d'}}^{2n}$ whenever $\vec{d'}$ is a subtuple of $\vec{d}$, essentially given by adding back in some of the divisor components (see \S\ref{subsec:div_compl}).
For instance, there is a symplectic embedding $X_k^{2n} \shookrightarrow X_{k'}^{2n}$ for any $k,k' \in \Z_{\geq 1}$. In particular, the obstructions provided by Corollary~\ref{cor:G_obstructions} are a purely exact symplectic phenomenon.

\begin{remark}[on virtual perturbations]\label{rmk:virtual}
We point out that the general construction of linearized contact homology, with its $\Li$ structure and full functoriality package, requires a virtual perturbation framework to achieve transversality for configurations involving multiply covered curves. 
The polyfold theory of Hofer--Wysocki--Zehnder is widely expected to provide such a framework (see e.g. \cite{hofer2017polyfold}).
There are several other candidate such frameworks in development and in various stages of completeness - see e.g. \cite{fish2018lectures, pardon2019contact,HuN, bao2015semi, ishikawa2018construction} and the references therein. 
Our high level discussion of symplectic invariants in \S\ref{sec:invariants} and their computation in \S\ref{sec:computationsI} relies only general properties of SFT as outlined in \cite{EGH2000} and not on any particularities of the chosen perturbation framework.
Subsequently, in \S\ref{sec:computationsII} we give a detailed discussion of transversality for the curves relevant to our main applications, and proceed to give direct proofs based on classical transversality techniques.
\end{remark}

\sss

Although Corollary~\ref{cor:G_obstructions} rules out many Liouville embeddings between hypersurface complements, it turns out to be rather far from optimal in general. Indeed, there are many cases in Problem~\ref{prob:hyp} with $\sum_{i=1}^k d_i \leq \sum_{i=1}^{k'}d_i'$ for which a stronger obstruction is necessary. In \S\ref{sec:computationsII}, we refine the proof of Theorem~\ref{thm:main_G_computation} by analyzing the outcome of neck stretching in more detail, arriving at our main combinatorial obstruction.
\begin{thm}\label{thm:main_combinatorial}
Fix $n \in \Z_{\geq 1}$, and consider tuples $\vec{d} = (d_1,\dots,d_k) \in \Z_{\geq 1}^k$ and $\vec{d'} = (d_1',\dots,d_{k'}') \in \Z_{\geq 1}^{k'}$ with $\sum_{i=1}^k d_i, \sum_{i=1}^{k'}d_i' \geq n+1$.
Given a Liouville embedding $X_{\vec{d}}^{2n} \es X_{\vec{d'}}^{2n}$, we must have:
\begin{itemize}
\item positive integers $l,q \in \Z_{\geq 1}$ with $\sum_{i=1}^k d_i \leq l \leq \sum_{i=1}^{k'}d_i'$ and $q(\sum_{i=1}^kd_i - n - 1) \leq l -n - 1$
\item tuples $\vec{x}_1,\dots,\vec{x}_l \in \Z_{\geq 0}^k \setminus \{\vec{0}\}$, each having at most $n$ nonzero components, such that $\sum_{i=1}^l \vec{x}_i = q\vec{d}$
\item tuples $\vec{y}_1,\dots,\vec{y}_l \in \Z_{\geq 0}^{k'} \setminus \{\vec{0}\}$, 
such that $\sum_{i=1}^l \vec{y}_i = \vec{d'}$
\item a group homomorphism $\Phi: \Z^k/(\vec{d}) \rightarrow \Z^{k'}/(\vec{d'})$ such that $\Phi(\vec{x}_i\text{ mod }(\vec{d})) = \vec{y}_i\text{ mod }(\vec{d'})$ for $i = 1,\dots,l$.
\end{itemize}
\end{thm}

This theorem is proved using moduli spaces of genus zero punctured pseudoholomorphic curves with several positive ends.
The proof of Theorem~\ref{thm:main_liouville} in \S\ref{subsec:completing} will then be extracted from the combinatorics of Theorem~\ref{thm:main_combinatorial}, together with the constructions described in Theorem~\ref{thm:embeddings} below.
In fact, we conjecture that the main degree assumption $\sum\limits_{i=1}^{k'} d_i' < 2\sum\limits_{i=1}^{k}d_i - n -1$ in Theorem~\ref{thm:main_liouville} can be removed, but it is not immediately clear whether this can be deduced from the combinatorics of Theorem~\ref{thm:main_combinatorial}.

At present, we illustrate the utility of Theorem~\ref{thm:main_combinatorial} with some examples which go beyond Corollary~\ref{cor:G_obstructions}.
Note that these examples also hold without the assumption $\sum\limits_{i=1}^{k'} d_i' < 2\sum\limits_{i=1}^{k}d_i - n -1$.
\begin{example}\label{ex:hyperplane1}
In the context of Theorem~\ref{thm:main_combinatorial}, consider a Liouville embedding $X_{\vec{d}}^{2n} \es X_{\vec{d'}}^{2n}$ where the target $X_{\vec{d'}}^{2n} = X_{k'}^{2n}$ is a hyperplane complement.
Then the source is also a hyperplane complement, i.e. we must have $X_{\vec{d}}^{2n} = X_k^{2n}$ for some $k \leq k'$.

Indeed, in the context of Theorem~\ref{thm:main_combinatorial}, note that the rank of the image of $\Phi$ in $\Z^{k'}/(\vec{d'}) \cong \Z^{k'-1}$ must be at most $l-1$.
Since the rank of the domain $\Z^k/(\vec{d})$ of $\Phi$ is $k-1$, this is only possible if we have $k \geq l$. We therefore have $k \leq \sum_{i=1}^{k}d_i \leq l \leq k$, which implies that $d_1 = \dots = d_k = 1$.
\end{example}

\begin{example}\label{ex:domain_single_comp}
In the context of Theorem~\ref{thm:main_combinatorial}, consider a Liouville embedding $X_{\vec{d}}^{2n} \es X_{\vec{d'}}^{2n}$ where the source $X_{\vec{d}}^{2n} = X_{(d_1)}^{2n}$ is the complement of a single divisor component.
We claim that $d_1$ must divide $\gcd(\vec{d'})$.
Conversely, by Theorem~\ref{thm:embeddings} below, if $d_1$ divides $\gcd(\vec{d'})$ then there is a Weinstein embedding $X_{(d_1)}^{2n} \ws X_{\vec{d'}}^{2n}$. 
We conclude that $X^{2n}_{(d_1)} \es X^{2n}_{\vec{d'}}$ if and only if $d_1 | \gcd(\vec{d'})$.

To justify the claim, note that we must have $l \geq d_1$.
Moreover, since the domain of $\Phi$ is $\Z / (d_1)$, for each $i = 1,\dots,l$ we must have $d_1 \vec{y_i} = 0 \in \Z^{k'}/(\vec{d'})$, i.e. $d_1 \vec{y_i} = a_i\vec{d'}$ for some $a_i \in \Z_{\geq 1}$.
We then have
\begin{align*}
\vec{d'} = \sum_{i=1}^l \vec{y}_i = \sum_{i=1}^l \tfrac{a_i}{d_1}\vec{d'} \geq \tfrac{l}{d_1} \vec{d'} \geq \vec{d'},
\end{align*}
so all of these inequalities are equalities and we have $a_1 = \dots = a_l = 1$. 
It follows that $d_1$ divides each component of $\vec{d'}$, and hence it divides $\gcd(\vec{d'})$.
\end{example}

\subsubsection{Constructions}\label{subsubsec:constructions}

For $k \in \Z_{\geq 1}$, let $\calS_k := \Z_{\geq 1}^k/\Sigma_k$ denote the set of unordered $k$-tuples of positive integers.
Here $\Sigma_k$ denotes the symmetric group on $k$ letters with its natural action on $\Z^k_{\geq 1}$.
We will often represent the equivalence of a $k$-tuple by its unique representative $(d_1,\dots,d_k)$ such that $d_1 \geq \dots \geq d_k$.
Put $\calS := \cup_{k=1}^\infty \calS_k$.
We define a partial order on $\calS$ as follows:
\begin{definition}
For $\vec{d},\vec{d'} \in \calS$, we put $\vec{d} \leqq \vec{d'}$ if $\vec{d'}$ can be obtained from $\vec{d}$ by a sequence
of the following moves:
\begin{enumerate}
\item (combination) delete two entries $d_i,d_j$ for some $1 \leq i < j \leq k$ and add the new entry $d_i + d_j$
\item (duplication) add a new entry $d_{k+1}$ with $d_{k+1} = d_i$ for some $1 \leq i \leq k$.
\end{enumerate}
\end{definition}
\NI For example, we have $(3,2,2) \leqq (7,2)$ thanks to the following sequence of moves: $$(3,2,2) \rightsquigarrow (3,2,2,2) \rightsquigarrow (5,2,2) \rightsquigarrow (7,2).$$
By contrast, we have $(3,2,2) \not\leqq (10,1)$, since there is no way to acquire the entry $1$ by a sequence of the above moves.
\begin{thm}\label{thm:embeddings}
Fix $n \in \Z_{\geq 1}$. For $\vec{d},\vec{d'} \in \calS$ such that $\vec{d} \leqq \vec{d'}$ there is a Weinstein
embedding of $X^{2n}_{\vec{d}}$ into $X^{2n}_{\vec{d'}}$.
\end{thm}
\begin{remark}
Our proof of Theorem~\ref{thm:embeddings} takes inspiration from \cite{khoa}, which gives a more precise description of the resulting Weinstein cobordism in the case that the divisor has no triple intersection points.
In fact, Theorem~\ref{thm:embeddings} may already be known to experts, but we nevertheless include the proof for completeness.
\end{remark}

\sss

Note that Theorem~\ref{thm:main_combinatorial} obstructs Liouville embeddings, hence a fortiori Weinstein embeddings. Since the constructions provided by Theorem~\ref{thm:embeddings} are Weinstein embeddings,
we can also reformulate most of the preceding results in the Weinstein category. For example, the analogue of Theorem~\ref{thm:main_liouville} is:
\begin{cor}\label{cor:main_weinstein}
Fix $n \in \Z_{\geq 1}$ and tuples $\vec{d} = (d_1,\dots,d_k) \in \Z_{\geq 1}^k$ and $\vec{d'} = (d_1',\dots,d_{k'}') \in \Z_{\geq 1}^{k'}$ with $\sum_{i=1}^kd_i,\sum_{i=1}^{k'}d_i' \geq n+1$.
Assume that we have $\sum_{i=1}^{k'} d_i' < 2\sum_{i=1}^{k}d_i - n -1$.
Then there is a Weinstein embedding $X^{2n}_{\vec{d}} \ws X^{2n}_{\vec{d'}}$ if and only if $\vec{d} \leqq \vec{d'}$.
\end{cor}
\NI We note, however, that Weinstein embeddings have more restricted topology compared to Liouville embeddings. Namely, the complementary cobordism must admit a Morse function with all critical points having index at most half the dimension (see e.g. \S\ref{subsec:geom_prel}). Consequently, many of the obstructions involved in Corollary~\ref{cor:main_weinstein} follow simply from singular homology considerations (c.f. Remark~\ref{rmk:Liouville_vs_Weinstein_flex}).

\sss

As for the symplectic category, there is quite a bit more flexibility. 
For example, as mentioned above there are symplectic embeddings $X^{2n}_k \shookrightarrow X^{2n}_{k'}$ for any $k,k' \in \Z_{\geq 1}$. At the same time, in some cases symplectic embeddings are automatically Liouville embeddings due to first cohomology considerations, and hence Theorem~\ref{thm:main_liouville} applies.
\begin{cor}\label{cor:main_symplectic}
Fix $n \in \Z_{\geq 2}$ and tuples $\vec{d} = (d_1,\dots,d_k) \in \Z_{\geq 1}^k$ and $\vec{d'} = (d_1',\dots,d_{k'}') \in \Z_{\geq 1}^{k'}$ with $\sum_{i=1}^kd_i,\sum_{i=1}^{k'}d_i' \geq n+1$.
If there is a symplectic embedding $X^{2n}_{\vec{d}} \shookrightarrow X^{2n}_{\vec{d'}}$, then we must have that $\gcd(\vec{d})$ divides $\gcd(\vec{d'})$. 
Moreover, if we assume that $\gcd(\vec{d})$ is an entry of $\vec{d}$, then there is a symplectic embedding $X^{2n}_{\vec{d}} \shookrightarrow X^{2n}_{\vec{d'}}$ if and only if $\gcd(\vec{d})$ divides $\gcd(\vec{d'})$.
\end{cor}
\begin{proof}
Suppose that we have a symplectic embedding $X^{2n}_{\vec{d}} \shookrightarrow X^{2n}_{\vec{d'}}$. Put $g := \gcd(\vec{d})$.
By Theorem~\ref{thm:embeddings} there is a Weinstein embedding $X^{2n}_{(g)} \ws X^{2n}_{\vec{d}}$, and hence
by concatenating we get a symplectic embedding $X^{2n}_{(g)} \shookrightarrow X_{\vec{d'}}^{2n}$.
Since $H^1(X^{2n}_{(g)};\R) = 0$, this is automatically a Liouville embedding.
By Example~\ref{ex:domain_single_comp}, such a Liouville embedding exists only if $g$ divides $\gcd(\vec{d'})$.

If we assume that $g$ divides $\gcd(\vec{d'})$ and also that $g$ is an entry of $\vec{d}$, then we have a symplectic embedding $X^{2n}_{\vec{d}} \shookrightarrow X^{2n}_{(g)}$ given by adding back divisor components, and we can concatenate this with a Weinstein embedding $X^{2n}_{(g)} \ws X^{2n}_{\vec{d'}}$ to get a symplectic embedding
$X_{\vec{d}}^{2n} \shookrightarrow X_{\vec{d'}}^{2n}$.
\end{proof}

\sss

The rest of this paper is structured roughly as follows. 
In \S\ref{sec:setting_stage} we discuss the necessary background on divisor complements and pseudoholomorphic curves, meanwhile setting the notation for the rest of the paper.
In \S\ref{sec:invariants} we introduce our main symplectic embedding obstructions $\G\lll\T^mp\rrr$, which arise as simplifications of a more general family of symplectic invariants $\I^{\leq l}$.
In \S\ref{sec:computationsI} we begin the discussion of the relevant SFT moduli spaces and we prove Theorem~\ref{thm:main_G_computation}, assuming virtual perturbations.
In \S\ref{sec:computationsII}, we analyze the moduli spaces in more detail and prove Theorem~\ref{thm:main_combinatorial}.
In \S\ref{sec:constructions}, we produce Weinstein embeddings, and also discuss flexibility constructions which place our main results into broader context.
Finally, in \S\ref{sec:concl} we give a (highly nonexhaustive) list of open problems and future directions.

\subsection*{Acknowledgements}
S.G. would like to thank Daniel Pomerleano for the suggestion that the holomorphic lines in $\CP^n$ ought to be able to (through a numerical invariant proposed by Seidel \cite{seidelslides} extracted from $\Li$ structures in equivariant symplectic cohomology) exhibit the increasing Liouville complexity of hyperplane complements in $\CP^n$. S.G. was supported by NSF grant DMS--1907635.
K.S. thanks Oleg Lazarev and Mark McLean for helpful discussions and comments. K.S. we partly supported by NSF grant DMS-2105578. We also thank the anonymous referees for many helpful suggestions.

\begin{addendum}
After the first draft of this paper was completed, the authors learned of the concurrent paper \cite{moreno2020landscape} by A. Moreno and Z. Zhou, whose techniques and results are closely related to the present work. 
In \cite{moreno2020landscape}, the authors define and exploit algebraic structures on rational symplectic field theory in order to obstruct exact cobordisms between contact manifolds, implemented using Pardon's framework \cite{pardon2019contact}. In particular, their techniques recover our Corollary~\ref{cor:intro_X_k_emb} (see \cite[Thm. G]{moreno2020landscape}), and they moreover compute their invariants for a variety of other geometrically natural examples.
\end{addendum}

\section{Setting the stage}\label{sec:setting_stage}

\subsection{Geometric preliminaries}\label{subsec:geom_prel}

\subsubsection{Contact manifolds and symplectizations}

Recall that a {\em contact form} on a closed odd-dimensional manifold $Y$ is a maximally nondegenerate one-form $\alpha$, i.e. $\alpha \wedge \underbrace{d\alpha \wedge \dots \wedge d\alpha}_{n-1}$ is everywhere nonvanishing, where $\dim Y = 2n-1$.
If the orientation induced by this volume form agrees with a preferred orientation on $Y$ then we say that $\alpha$ is a {\em positive} contact from.
A {\em contact manifold} is a pair $Y$ equipped with a hyperplane distribution of the form $\xi = \ker \alpha$ for some contact form $\alpha$.
In this paper we will typically work with {\em strict} contact manifolds, i.e. contact manifolds having a preferred one-form $\alpha$.
By slight abuse of notation, we will often refer to the strict contact manifold simply by $Y$ when $\alpha$ is implicitly understood, and a similar remark holds for Liouville domains and so on.

Given a strict contact manifold $(Y,\alpha)$, the {\em symplectization} is the symplectic manifold $\R \times Y$ equipped with the symplectic form $d(e^r\alpha)$ and preferred one-form $e^r\alpha$, where $r$ is the coordinate on $\R$.
We will sometimes also utilize the positive (resp. negative) half-symplectization, given by restricting to $\R_{\geq 0} \times Y$ (resp. $\R_{\leq 0} \times Y$).

The {\em Reeb vector field} $R_\alpha$ is the unique vector field on $Y$ such that $\alpha(R_\alpha) \equiv 1$ and $d\alpha(R_\alpha,-) \equiv 0$. By a ($T$-periodic) {\em Reeb orbit} we mean a loop $\gamma: [0,T] \rightarrow Y$ with $\gamma(0) = \gamma(T)$ for some $T \in \R_{>0}$, such that $\dot{\gamma}(t) = R_\alpha(\gamma(t))$ for all $t \in [0,T]$.
Here $T$ is called the {\em period} or {\em action} of $\gamma$, denoted by $\calA(\gamma)$. 
Note that equivalently we have $\calA(\gamma) = \int_\gamma \alpha$.
A Reeb orbit $\gamma$ is {\em nondegenerate} if the map $\xi|_{\gamma(0)} \rightarrow \xi|_{\gamma(0)}$ induced by linearized the time-$T$ Reeb flow does not have $1$ as an eigenvalue, and the contact form $\alpha$ is nondegenerate if all of its Reeb orbits are.

\subsubsection{Flavors of open symplectic manifolds}

Recall that a {\em Liouville domain} is a pair $(X,\la)$, where $X$ is an even-dimensional compact manifold with boundary and $\la$ is a one-form such that $d\la$ is symplectic and $\la$ restricts to a positive\footnote{More precisely, we orient $X$ by the volume form $\wedge^n \omega$ and we orient $\bdy X$ via the boundary orientation.} contact form on $\bdy X$.
This last condition is equivalent to the
Liouville vector field $V_\la$, characterized by $d\la(V_\la,-) = \la$, being outwardly transverse along $\bdy X$.

There is a closely related notion of {\em Liouville manifold}, which is a pair $(X,\la)$, where $X$ is a noncompact manifold and $\la$ is a one-form such that $d\la$ is symplectic and the flow of the Liouville vector field $V_\la$ is complete. 
If we can moreover find a compact subdomain $D \subset X$ with smooth boundary such that $V_\la$ is outwardly transverse along $\bdy D$ and $\la$ is nonvanishing on $X \setminus D$, then $(X,\la)$ is said to be of ``finite type''.
In this case, the restriction $(D,\la|_D)$ defines a Liouville domain. Conversely, if $(X,\la)$ is a Liouville domain, its symplectic completion $(\wh{X},\wh{\la})$ is the finite type Liouville manifold given by attaching the positive half-symplectization of $(\bdy X, \la|_{\bdy X})$ to $(X,\la)$. 

Given two Liouville domains $(X,\la)$ and $(X,\la')$ on the same manifold $X$, we say that they are {\em Liouville homotopic} if there is a smooth one-parameter family of Liouville forms $\la_t$, $t \in [0,1]$, with $\la_0 = \la$ and $\la_1 = \la'$. Two Liouville domains $(X,\la)$ and $(X',\la')$ are {\em Liouville deformation equivalent} if there exists a diffeomorphism $F: X \rightarrow X'$ such that $(X,\la)$ and $(X,F^*\la')$ are Liouville homotopic.
These induce equivalent notions of Liouville homotopy and Liouville deformation equivalence between the corresponding symplectic completions.
By a version of Moser's Stability Theorem (see \cite[Prop. 11.8]{cieliebak2012stein}), if two Liouville domains are Liouville homotopic, then their symplectic completions are symplectomorphic.
Moreover, by \cite[Lem. 11.2]{cieliebak2012stein}, if two Liouville manifolds $(X,\la)$ and $(X,\la')$ are symplectomorphic, then we can find a diffeormorphism $G: X \rightarrow X'$ such that $G^*\la' - \la$ is an exact one-form.

In particular, the process of symplectic completion sets up a one-to-one correspondence between Liouville domains up to Liouville homotopy and finite type Liouville manifolds up to Liouville homotopy.
In the sequel we will mostly phrase results in terms of Liouville domains for convenience. 
A similar remark will apply for Weinstein domains / manifolds and Stein domains / manifolds.

\sss

A {\em Weinstein domain} is a triple $(X,\la,\phi)$, where $(X,\la)$ is a Liouville domain and $\phi: X \rightarrow \R$ is a generalized Morse function which is gradient-like for the Liouville vector field $V_\la$ and constant along $\bdy X$. 
Similarly, a {\em Weinstein manifold} is a triple $(X,\la,\phi)$, where $(X,\la)$ is a Liouville manifold and $\phi: X \rightarrow \R$ is an {\em exhausting} (i.e. proper and bounded from below) Morse function such that $V_\la$ is gradient-like for $\phi$.
The Weinstein manifold $(X,\la,\phi)$ is {\em finite type} if and only if $\phi$ has finitely many critical points, which implies that $(X,\la)$ is a finite type Liouville manifold.
A standard computation shows that each critical point of $\phi$ has index at most half the dimension of $X$, and this puts strong restrictions on the homotopy type of $X$.
Conversely, any manifold $X$ of dimension at least six which admits a nondegenerate two-form and an exhausting Morse function with critical points of index at most half the ambient dimension is diffeomorphic to a Weinstein manifold (see \cite[Thm. 13.2]{cieliebak2012stein}).

Two Weinstein domains $(X,\la,\phi)$ and $(X,\la',\phi')$ are {\em Weinstein homotopic} if there exists a smooth family of Weinstein domains $(X,\la_t,\phi_t)$ for $t \in [0,1]$, with $(\la_0,\phi_0) = (\la,\phi)$ and $(\la_1,\phi_1) = (\la',\phi')$.
Note that a generic one-parameter family of functions will have isolated {\em birth-death type degenerations}, which is why we require the functions $\phi_t$ to be only generalized Morse (see \cite[\S9.1]{cieliebak2012stein}).
Similarly, two Weinstein domains $(X,\la,\phi)$ and $(X',\la',\phi')$ are {\em Weinstein deformation equivalent} if there is a diffeomorphism $F: X \rightarrow X'$ such that $(X,\la,\phi)$ and $(X,F^*\la',F^*\phi')$ are Weinstein homotopic.

\sss

A {\em Stein manifold} is a (necessarily noncompact) complex manifold admitting a proper biholomorphic embedding into affine space $\C^N$ for some $N \in \Z_{\geq 1}$.
There are several other common equivalent definitions - see e.g. \cite[\S 5.3]{cieliebak2012stein} for more details.
It turns out that one can always find an exhausting strictly plurisubharmonic function $\phi: X \rightarrow \R$, which we can assume is Morse after a small perturbation.
Here strict plurisubharmonicity of $\phi$ is equivalent to $-dd^\C \phi$ being a K\"ahler form, where $d^\C\phi$ denotes the one-form $d\phi \circ J$, with $J$ the (integrable) almost complex structure.
The Stein manifold $(X,\la,\phi)$ is {\em finite type} if $\phi$ has finitely many critical points.
The definitions of {\em Stein homotopy} and {\em Stein deformation equivalence} mirror the Weinstein case.

Given a Stein manifold $(X,J)$ and an exhausting strictly plurisubharmonic function $\phi: X \rightarrow \R$ , we produce a Weinstein manifold $\wein(X,J) := (X,\la := -d^\C \phi, \psi \circ \phi)$, where $\psi: \R \rightarrow \R$ is a suitable diffeomorphism (this is needed to make the vector field $V_\la$ complete - see \cite[\S2.1]{cieliebak2012stein}).
Moreover, up to Weinstein homotopy this Weinstein manifold depends only on the Stein manifold $(X,J)$ up to Stein homotopy. 
In fact, by a deep result from \cite{cieliebak2012stein}, this association sets up to one-to-one correspondence between Stein manifolds up to Stein homotopy and Weinstein manifolds up to Weinstein homotopy.
As a consequence, for the qualitative embedding problems considered in this paper, it makes no essential difference if we work in the Stein or Weinstein category.

\sss

The above definitions also naturally generalize to the notions of {\em Liouville cobordism}, {\em Weinstein cobordism}, and {\em Stein cobordism}. For example, a Liouville cobordism (a.k.a. {\em exact cobordism}) is a pair $(X,\la)$ where $X$ is a compact manifold with boundary such that the Liouville vector field is inwardly transverse along some components of $\bdy X$ (the {\em negative boundary} $\bdy^- X$) and outwardly transverse along the components of $\bdy X$ (the {\em positive boundary} $\bdy^+ X$). 
Given a Liouville cobordism $X$, we pass to its symplectic completion by attaching the positive half-symplectization of $\bdy^+X$ to its positive end and the negative half-symplectization of $\bdy^-X$ to its negative end.
Similarly, a Weinstein cobordism is a triple $(X,\la,\phi)$, where $(X,\la)$ is a Liouville cobordism and $\phi$ is a Morse function which is constant along $\bdy^-X$ and $\bdy^+X$, such that $V_\la$ is gradient-like for $\phi$.

\sss

As we recall in \S\ref{subsec:div_compl}, smooth complex affine algebraic varieties are Stein manifolds, canonically up to Stein homotopy. 
In summary, we have the following hierarchy for exact symplectic manifolds
\begin{align*}
\text{affine} \Rightarrow \text{Stein} \Leftrightarrow \text{Weinstein} \Rightarrow \text{Liouville}.
\end{align*}

\begin{remark}
Although pseudoholomorphic curves are best behaved in exact symplectic manifolds, for many purposes it suffices to have exactness only near the boundary.
If we relax the definition of a Liouville cobordism by only requiring the one-form $\la$ to be defined near the boundary, we arrive at the notion of a {\em (strong) symplectic cobordism}.
\end{remark}

\subsubsection{Embeddings}\label{subsubsec:embeddings}
Recall (see e.g. \cite{Seidel_biased_view}) that a {\em Liouville embedding} from one Liouville domain $(X,\la)$ into another Liouville domain $(X',\la')$ of the same dimension is a smooth embedding $\iota: X \hookrightarrow X'$ such that $\iota^*(\la') = e^{\rho} \la + df$ for some constant $\rho \in \R$ and some smooth function $f: X \rightarrow \R$. 
As a shorthand, we put $(X,\la) \es (X',\la')$ or simply $X \es X'$ if such a Liouville embedding exists.
Note that in the case $\rho = 0$, this says that $\iota$ is an {\em exact symplectic embedding}, i.e. $\iota^*(\la') - \la$ is an exact one-form, and if moreover $f \equiv 0$, then 
$\iota$ is a {\em strict exact symplectic embedding}, i.e. it satisfies $\iota^*\la' = \la$.
Also, given Liouville domains $(X,\la)$ and $(X',\la')$, we will say that a smooth embedding $\iota: \wh{X} \hookrightarrow \wh{X}'$ is {\em exact symplectic} if we have $\iota^*\wh{\la}' = \wh{\la} + df$ for a smooth function $f: \wh{X} \rightarrow \R$.

The following lemma combines a few standard observations about Liouville embeddings.
\begin{lemma}\label{lem:liouville_emb_lemmas}\hspace{1cm}
\begin{enumerate}[label=(\alph*)]
\item Suppose that $(X,\la_t)_{t \in [0,1]}$ is a Liouville homotopy of Liouville domains. Then there is a diffeomorphism $h: \wh{X} \rightarrow \wh{X}$ such that $h^*\wh{\la}_1 = \wh{\la}_0 + df$ for some smooth function $f: \wh{X} \rightarrow \R$.
\item Let $(X,\la)$ and $(X,\la')$ be Liouville domains, and suppose there is a Liouville embedding of $(X,\la)$ into $(X',\la')$. Then there is a Liouville homotopy $(X',\la'_t)_{t \in [0,1]}$ with $\la'_0 = \la'$ and a strict exact symplectic embedding of $(X,\la)$ into $(X',\la'_1)$. Moreover, we can assume that we have $\la'_t|_{\bdy X'} = e^{q(t)}\la'|_{\bdy X'}$ for some smooth function $q: [0,1] \rightarrow \R$ and all $t \in [0,1]$.
\item For Liouville domains $(X,\la)$ and $(X',\la')$, there is a Liouville embedding $(X,\la) \es (X',\la')$ if and only if there is a (not necessarily proper) exact symplectic embedding of $(\wh{X},\wh{\la})$ into $(\wh{X'},\wh{\la'})$.
\item Suppose that there is a Liouville embedding $(X,\la) \es (X',\la')$ of Liouville domains. Then the same is true after applying a Liouville homotopy to $(X,\la)$ or $(X',\la')$.
\end{enumerate}
\end{lemma}
\begin{proof}
Part (a) is \cite[Prop. 11.8]{cieliebak2012stein}, proved using Moser's trick.

For (b), suppose that $\iota: X \hookrightarrow X'$ is a smooth embedding with $\iota^*\la' = e^\rho\la + df$ for some constant $\rho \in \R$ and smooth function $f: X \rightarrow \R$. 
After post-composing $\iota$ with the Liouville flow of $X'$ for some negative time, we can assume that we have $\iota(X) \subset \Int X'$.
Let $\wt{f}: X' \rightarrow \R$ be a smooth function whose restriction to $\iota(X)$ agrees with $\iota_* f$ and which vanishes outside of a small neighborhood of $\iota(X)$.
Consider the Liouville one-form on $X'$ given by $\wt{\la'} := e^{-\rho}(\la' - d\wt{f})$.
Then we have $\iota^*(\wt{\la'}) = \la$, and the family $\la'_t := e^{-t\rho}(\la' - td\wt{f})$ defines a Liouville homotopy with $\la'_0 = \la'$ and $\la'_1 = \wt{\la'}$.

For (c), first suppose that $\iota: \wh{X} \hookrightarrow \wh{X'}$ is a smooth embedding satisfying $\iota^*\wh{\la'} = \wh{\la} + df$ for a smooth function $f: \wh{X} \rightarrow \R$.
For $t \in \R$, let $\phi_t: \wh{X'} \rightarrow \wh{X'}$ denote the time-$t$ flow of the Liouville vector field, so we have $\phi_t^*\wh{\la'} = e^t \wh{\la'}$.
Then for $t \ll 0$, the composite embedding $\phi_t \circ \iota|_{X}: X \rightarrow \wh{X'}$ has image in $X'$, and it pulls back $\wh{\la'}$ to $e^t\wh{\la} + d(e^tf)$, so it is a Liouville embedding.

Conversely, suppose that $\iota: X \hookrightarrow X'$ is a Liouville embedding.
By (a) and (b), we can assume that $\iota$ is a strict exact symplectic embedding, i.e. we have $\iota^*\la' = \la$.
We extend $\iota$ to a smooth embedding $\wh{\iota}: \wh{X} \rightarrow \wh{X'}$ by requiring $\wh{\iota}$ to intertwine the Liouville flows on $(\wh{X},\wh{\la})$ and $(\wh{X'},\wh{\la'})$.
We then have $\wh{\iota}^*\wh{\la'} = \wh{\la}$.

Finally, (d) follows immediately by combining (a) and (c).

\end{proof}

\begin{remark}A key feature of Liouville embeddings $X \es X'$ is that 
curves without positive ends in the complementary cobordism $X' \setminus X$ are ruled out by Stokes' theorem.
That is, for any admissible almost complex structure $J$ on the symplectic completion of $X' \setminus X$, there are no nontrivial punctured asymptotically cylindrical $J$-holomorphic curves without positive ends (see \S\ref{subsubsec:SFT_mod_sp} below). 
\end{remark}

Similarly, given Weinstein domains $(X,\la,\phi)$ and $(X',\la',\phi')$ of the same dimension, a {\em strict Weinstein embedding} consists of a smooth embedding $\iota: X \rightarrow X'$ such that $\iota(X)$ is a sublevel set of $\phi'$ and we have $\iota^*\la' = \la$ and $\phi' \circ \iota = \phi$.
In this case, $X' \setminus \iota(X)$ equipped with the restrictions of $\la'$ and $\phi'$ is a Weinstein cobordism with positive end $\bdy X'$ and negative end $\iota(\bdy X)$.
More generally, we say there is a {\em Weinstein embedding} of $(X,\la,\phi)$ into $(X',\la',\phi')$, denoted by $X \ws X'$, if there is a strict Weinstein embedding after applying Weinstein homotopies to $X$ and $X'$.

\begin{example}\label{ex:surfaces}

For $g,k \in \Z_{\geq 0}$, let $\Sigma_{g,k}$ denote a compact surface of genus $g$ with $k$ boundary components.
Then $\Sigma_{g,k}$ admits a unique Liouville structure up to Liouville deformation equivalence. Indeed, it is easy to produce such a structure by attaching Weinstein one-handles to the two-ball, and if $\la_0$ and $\la_1$ are one-forms on $\Sigma_{g,k}$ which induce the same orientation then they can be joined by the Liouville homotopy $\la_t := (1-t)\la_0 + t\la_1$, $t \in [0,1]$.

Moreover, if $X$ and $X'$ are two-dimensional Liouville domains with a Liouville embedding $X \es X'$, by Stokes' theorem each 
component of $X' \setminus X$ must contain at least one component of $\bdy X'$, and in fact this condition also suffices to provide the existence of a Liouville embedding $X \es X'$.
It follows that there is a Liouville embedding $\Sigma_{g,k} \es \Sigma_{g',k'}$ if and only if we have $g \leq g'$ and $k-k' \leq g'-g$.
\end{example}

\subsubsection{The Conley--Zehnder index}\label{subsubsec:CZ}
The {\em Conley--Zehnder index} plays an important role in the Fredholm index formula for punctured curves.
Let $\gamma$ be a Reeb orbit in a nondegenerate strict contact manifold $(Y,\alpha)$.
The contact distribution $\xi := \ker \alpha$ equipped with the restriction of $d\alpha$ is a symplectic vector bundle over $Y$. 
The Conley--Zehnder index \cite{conley1984morse} of $\gamma$ is defined with respect to a choice of {\em framing}, i.e. a trivialization (up to homotopy) $\tau$ of the pullback of the symplectic vector bundle $\xi$ by $\gamma$.
We denote this by $\cz_{\tau}(\gamma) \in \Z$.
Given another framing $\tau'$, we have
\begin{align}
\cz_{\tau}(\gamma) - \cz_{\tau'}(\gamma) = 2m(\tau',\tau),
\end{align}
where we use $\tau$ to view the framing $\tau'$ as a loop in $\op{Sp}(2n-2)$, and $m(\tau',\tau) \in {\pi_1(\op{Sp}(2n-2))} \cong \Z$ is its Maslov index (see e.g. \cite{robbin1993maslov}).

Suppose that $Y$ is the contact boundary of a Liouville domain $X$. Given a spanning disk for $\gamma$, i.e. a map $u: \D^2 \rightarrow X$ with $u|_{\bdy \D^2} = \gamma$, there is a unique (up to homotopy) trivialization of $\gamma^*TX$ which extends to a trivialization of $u^*TX$, and this induces a trivialization of $\gamma^*\xi$.
Let $\cz_u(\gamma)$ denote the corresponding Conley--Zehnder index with respect to this trivialization.
Given another such spanning disk $u'$, the difference in Conley--Zehner index is given by
\begin{align}
\cz_{u}(\gamma) - \cz_{u'}(\gamma) = \langle 2c_1(X),A\rangle,
\end{align}
where $A \in H_2(X)$ is the homology class of the sphere given by gluing $u$ to $u'$ with its opposite orientation.

\subsection{SFT moduli spaces}\label{subsec:sft_moduli_spaces}

\subsubsection{Admissible almost complex structures}

Let $(Y,\alpha)$ be a strict contact manifold, and let $\R \times Y$ be its symplectization, with $\R$-coordinate $r$.
An almost complex structure $J$ on $\R \times Y$ is {\em admissible} or {\em cylindrical} if it is invariant under $r$-translations, sends $\bdy _r$ to $R_\alpha$, and restricts to a $d\alpha$-compatible almost complex structure on the contact distibution $\xi = \ker \alpha$ on $\{0\} \times Y$.
Note that such an almost complex structure is compatible with the symplectic form $\omega = d(e^r\alpha)$ on $\R \times Y$, i.e. $\omega(J-,-)$ is symmetic and nondegenerate on each tangent space, but it also satisfies an additional $\R$-symmetry. 

Similarly, suppose that $X$ is a strong symplectic cobordism, with corresponding symplectic form $\omega$ and contact forms $\alpha^\pm$ on $\bdy^{\pm}X$, and let ${\wh{X} = \left(\R_{\leq 0} \times \bdy^-X\right) \cup X \cup \left( \R_{\geq 0} \times \bdy^+X \right)}$ denote its symplectic completion.
An almost complex structure $J$ on $X$ is {\em admissible} if it is $\omega$-compatible on $X$, and it is cylindrical when restricted to the ends $\R_{\geq 0} \times \bdy^+X$ and $\R_{\leq 0} \times \bdy^-X$.

\subsubsection{SFT moduli spaces}\label{subsubsec:SFT_mod_sp}

The main analytical tool in this paper is the study of moduli spaces of punctured pseudoholomorphic curves \`a la symplectic field theory. 
We refer the reader to \cite{BEHWZ, abbas2014introduction} for more of the technical details, and we also recommend \cite{wendl2016lectures} for an excellent recent treatment. Since the setup here is quite similar to that of \cite[\S 3.2]{HSC}, we give here only a short summary to set our notation.

\sss

Let $(Y,\alpha)$ be a nondegenerate strict contact manifold, and let $J$ be an admissible almost complex structure on its symplectization $\R \times Y$.  
Suppose that we have two collections of Reeb orbits $\Gamma^+ = (\gamma_1^+,\dots,\gamma^+_{s^+})$ and $\Gamma^- = (\gamma^-_1,\dots,\gamma^-_{s^-})$ in $Y$, for some $s_+,s_- \in \Z_{\geq 0}$. 
We let $\calM^J_{Y}(\Gamma^+;\Gamma^-)$ denote the moduli space of $J$-holomorphic\footnote{We sometimes refer to $J$-holomorphic curves as  ``pseudoholomorphic curves'' or simply ``curves'' if the almost complex structure is unspecified or implicit. Similarly, we will also sometimes omit $J$ from our moduli space notation.} genus zero\footnote{All curves considered in this paper are genus zero and hence we will generally suppress the genus from the notation.} curves in $\R \times Y$, with $s^+$ punctures which are positively asymptotic to the Reeb orbits $\gamma^+_1,\dots,\gamma^+_{s^+}$, and $s^-$ punctures which are negatively asymptotic to the Reeb orbits $\gamma^-_1,\dots,\gamma^-_{s^-}$.
Note that such curves (called ``asymptotically cylindrical'' in \cite{wendl2016lectures}) are proper, and the conformal structure on the domain (as a sphere with $(s^+ + s^-)$-punctures) is unconstrained. 
The $\R$-invariance of $J$ induces a corresponding $\R$-action on $\calM^J_Y(\Gamma^+;\Gamma^-)$
which is free away from {\em trivial cylinders}, i.e. cylinders of the form $\R \times \gamma \subset \R \times Y$ with $\gamma$ a Reeb orbit in $Y$.
We denote the quotient by $\calM^J_Y(\Gamma^+;\Gamma^-) / \R$.

Given a curve $u \in \calM^J_Y(\Gamma^+;\Gamma^-)$, we define its {\em energy} by
\begin{align}
E(u) := \int_u d\alpha.
\end{align}
Note that this is {\em not} quite the same as the symplectic area $\int_u d(e^r\alpha)$, which is always infinite.
Nevertheless, we have $E(u) \geq 0$, with equality if and only if $u$ is a branched cover of a trivial cylinder.
By Stokes' theorem, we have 
\begin{align}
E(u) = \sum_{i=1}^{s^+}\calA_{\alpha}(\gamma_i^+) - \sum_{i=1}^{s^-}\calA_{\alpha}(\gamma_i^-).
\end{align}
In particular, $u$ must have at least one positive puncture. Also, in the case $s^+ = s^- = 1$, we have $E(u) = 0$ if and only if $\gamma_1^+ = \gamma_1^-$ and $u$ is a trivial cylinder.

\sss

Similarly, let $X$ be a strong symplectic cobordism, with nondegenerate contact forms $\alpha^\pm$ on $\bdy^\pm X$, and let $J$ be an admissible almost complex structure on its symplectic completion $\wh{X}$.
For a collection of Reeb orbits $\Gamma^+ = (\gamma_1^+,\dots,\gamma^+_{s^+})$ in $\bdy^+X$ and $\Gamma^- = (\gamma^-_1,\dots,\gamma^-_{s^-})$ in $\bdy^-X$, we let $\calM^J_X(\Gamma^+;\Gamma^-)$ denote the moduli space of $J$-holomorphic curves in $\wh{X}$ with $s^+$ punctures positively asymptotic to $\gamma^+_1,\dots,\gamma^+_{s^+}$ and $s^-$ punctures negatively asymptotic to $\gamma^-_1,\dots,\gamma^-_{s^-}$.
By slight abuse of notation, we will often suppress the completion process from the discussion and refer to elements of $\calM^J_X(\Gamma^+;\Gamma^-)$ simply as ``curves in $X$''.

In the case that $X$ is a {\em symplectic filling}, i.e. $\bdy^-X = \nil$, note that curves in $X$ cannot have negative ends, and we denote the moduli space with positive asymptotics $\Gamma^+ = (\gamma^+_1,\dots,\gamma^+_{s^+})$ by $\calM_X^J(\gamma^+_1,\dots,\gamma^+_{s^+})$ without risk of confusion.
Similarly, if $X$ is a {\em symplectic cap}, i.e. $\bdy^+X = \nil$, we denote the moduli space of curves in $X$ with negative asymptotics $\Gamma^- = (\gamma^-_1,\dots,\gamma^-_{s^-})$ by $\calM_X^J(\gamma^-_1,\dots,\gamma^-_{s^-})$.

We define the energy of a curve $u \in \calM_X^J(\Gamma^+;\Gamma^-)$ by
\begin{align}
E(u) :=  \int_u \check{\omega},
\end{align}
giving by integrating over $u$ the piecewise smooth two-form 
\begin{align}
\check{\omega} := (d\alpha^+)|_{\R_{\geq 0} \times \bdy^+X} + \omega|_X + (d\alpha^-)|_{\R_{\leq 0} \times \bdy^-X}.
\end{align}
Note that for $J$ admissible and $u \in \calM_X^J(\Gamma^+;\Gamma^-)$ we have $E(u) \geq 0$, with $E(u) = 0$ if and only if $u$ is a constant map.
If $X$ is furthermore a Liouville cobordism, we have by Stokes' theorem
\begin{align}
E(u) = \sum_{i=1}^{s^+}\calA_{\alpha^+}(\gamma_i^+) - \sum_{j=1}^{s^-}\calA_{\alpha^-}(\gamma_j^-).
\end{align}
In particular, $u$ must have at least one positive end.

\sss

Let $H_2(Y; \Gamma^+ \cup \Gamma^-)$ denote the set of $2$-chains in $Y$ with boundary $\sum_{i=1}^{s^+} \gamma_i^+ - \sum_{j=1}^{s^-}\gamma_j^-$, modulo boundaries of $3$-chains (c.f. \cite[\S3.1]{Hlect}).
This forms a torsor over $H_2(Y)$.
A curve $u \in \calM^J_Y(\Gamma^+;\Gamma^-)$ has a well-defined homology class
$[u] \in H_2(Y;\Gamma^+ \cup \Gamma^-)$, and for a given class $A \in H_2(Y;\Gamma^+ \cup \Gamma^-)$ we have the subspace
$\calM^J_{Y,A}(\Gamma^+;\Gamma^-) \subset \calM^J_Y(\Gamma^+;\Gamma^-)$ consisting of all curves lying in the class $A$.
Similarly, we denote by $H_2(X;\Gamma^+ \cup \Gamma^-)$ the set of $2$-chains in $X$ with boundary $\sum_{i=1}^{s^+} \gamma_i^+ - \sum_{j=1}^{s^-}\gamma_j^-$, modulo boundaries of $3$-chains,
and for a homology class $A \in H_2(X;\Gamma^+ \cup \Gamma^-)$ we have the corresponding subspace
$\calM^J_{X,A}(\Gamma^+;\Gamma^-) \subset \calM^J_X(\Gamma^+;\Gamma^-)$.
The energy of a pseudoholomorphic curve $u$ in $X$ is determined by its homology class $[u] \in H_2(X;\Gamma^+ \cup \Gamma^-)$.

\sss

We will also sometimes need to consider parametrized moduli spaces of pseudoholomorphic curves. For example, let $\{J_t\}_{t \in [0,1]}$ be a (smooth) one-parameter family of admissible almost complex structures on the symplectization $\R \times Y$.
We denote by $\calM_{Y}^{\{J_t\}}(\Gamma^+;\Gamma^-)$ the parametrized moduli space consisting of all pairs $(t,u)$, with $t \in [0,1]$ and $u \in \calM_Y^{J_t}(\Gamma^+;\Gamma^-)$.
Similarly, if $\{J_t\}_{t \in [0,1]}$ is a one-parameter family of admissible almost complex structures on the strong symplectic cobordism $X$, we denote the corresponding parametrized moduli space by $\calM_X^{\{J_t\}}(\Gamma^+;\Gamma^-)$.

\subsubsection{SFT compactness and neck stretching}\label{subsubsec:neck_stretching}

The SFT compactness theorem, which comes in several variants, is the counterpart for punctured pseudoholomorphic curves of Gromov's compactness theorem for closed curves.
It provides natural compactifications of each of the above moduli spaces. Roughly, in addition to the nodal degenerations which appear in the closed curve case, punctured curves can degenerate into multilevel {\em pseudoholomorphic buildings}.
For example, a typical element of the compactification of $\calM^J_{Y,A}(\Gamma^+;\Gamma^-)/\R$ consists of some number $l \geq 1$ of levels in the symplectization $\R \times Y$. Each level consists of one or more $J$-holomorphic curve components\footnote{Here by {\em component} we mean irreducible component, i.e. a curve whose domain is smooth and connected. For example, a cylinder with an attached sphere bubble consists of two components. We will use the term ``curve component'' when we wish to emphasize that there is a single component, as opposed to a nodal curve or building. By contrast, we will use the term ``configuration'' when we wish the emphasize the possibility of several components or levels.} in $\R \times Y$, such that the Reeb orbit asymptotics of adjacent levels are matched, and the total domain after gluing along paired punctures is a sphere with $s^+ + s^-$ punctures.
Moreover, the total homology class of the configuration is $A \in H_2(Y;\Gamma^+\cup\Gamma^-)$, the positive asymptotics of the top level are given by $\Gamma^+$, and the negative asymptotics of the bottom level are given by $\Gamma^-$.
Each curve component is defined up to biholomorphic reparametrization, and each level is defined up to translation in the $\R$ direction.
In addition to disallowing constant closed components with two or fewer special points, the SFT stability condition also disallows symplectization levels consisting only of trivial cylinders.

Similarly, a typical element of the compactification $\ovl{\calM}_{X,A}^J(\Gamma^+;\Gamma^-)$ of 
$\calM_{X,A}^J(\Gamma^+;\Gamma^-)$ consists of a pseudoholomorphic building with some number (possibly zero) of levels in the symplectization $\R \times \bdy^+X$, a single level in $X$, and some number (possibly zero) of levels in the symplectization $\R \times \bdy^- X$, subject to the  same conditions as above. 
Here the curve components in the $X$ level are $J$-holomorphic, whereas the components in $\R \times \bdy^{\pm}X$ are $J^{\pm}$-holomorphic, where $J^{\pm}$ are the cylindrical almost complex structures naturally determined by restricting $J$.
Here each of the symplectization levels is again defined only up to $\R$-translation, whereas the level in $X$ is defined without any quotient.

\sss

For a parametrized moduli space such as $\calM_{X,A}^{\{J_t\}}(\Gamma^+;\Gamma^-)$, the SFT compactification $\ovl{\calM}_{X,A}^{\{J_t\}}(\Gamma^+;\Gamma^-)$ is defined as the union over pairs $(C,t)$ for $C \in \ovl{\calM}_{X,A}^{J_t}(\Gamma^+;\Gamma^-)$ and $t \in [0,1]$. 
A variation on this called {\em stretching the neck} constitutes a fundamental tool in symplectic field theory. Namely, $X$ be a strong symplectic cobordism (this includes the case that $X$ is closed) with symplectic form $\omega$, and let $Y \subset X$ be a separating codimension one closed submanifold which is {\em contact type}, i.e. there is a one-form $\la$ defined near $Y$ satisfying $d\la = \omega$, and such that the Liouville vector field $V_\la$ is transverse to $Y$.
Following e.g. \cite[\S3.4]{BEHWZ} (see also \cite[Lem. 2.4]{cieliebak2018symplectic}), we can define a family of almost complex structures $\{J_t\}_{t \in [0,1)}$ on $\wh{X}$ which roughly has the effect of stretching out $(-\eps,\eps) \times Y$ to $(-R_t,R_t) \times Y$, with $\lim\limits_{t \rightarrow 1} R_t = \infty$, such that $J_t$ is cylindrical on $(-R_t,R_t) \times Y$.
More precisely let $J$ be an admissible almost complex structure on $X$ which is cylindrical on a small neighborhood $U$ of $Y$ which is identified with $(-\delta,\delta) \times Y$ for some $\delta > 0$ under the flow of $V_\la$, with $J$ invariant under translations in the first factor. 
Let $J_{\R \times Y}$ denote the induced cylindrical almost complex structure on the full symplectization $\R \times Y$.
Let $F_t: (-R_t,R_t) \rightarrow (-\delta,\delta)$ be a family of increasing diffeomorphisms for $t \in [0,1)$, such that $F_t$ has slope $1$ near $-R_t$ and $R_t$.
We then set $J_t$ to be $(F_t \times \id)_*(J_{\R \times Y}|_{(-R_t,R_t)\times Y})$ on $U$ and $J|_{X \setminus U}$ on $X \setminus U$. We assume that $R_0 = \eps$ and $F_0$ is the identity function, so that we have $J_0 = J$.

Although $\lim\limits_{t \rightarrow 1}J_t$ is not a well-defined almost complex structure on $\wh{X}$, we nevertheless have a compactified moduli space
$\ovl{\calM}^{\{J_t\}}_{X,A}(\Gamma^+;\Gamma^-)$. 
This has a well-defined projection to $[0,1]$, where the fiber over $1$ corresponds to pseudoholomorphic buildings in the {\em broken symplectic cobordism} $X^- \circledcirc X^+$, 
where $X^-$ and $X^+$ correspond to the bottom and top components of $X \setminus Y$.
More precisely, a typical element of the fiber over $1$ in $\ovl{\calM}^{\{J_t\}}_{X,A}(\Gamma^+;\Gamma^-)$ is a pseudoholomorphic building with some number (possibly zero) of levels in the symplectization $\R \times \bdy^+ X$ (this is vacuous if $\bdy^+X = \nil$), a single level in $X^+$, some number (possibly zero) of levels in the symplectization $\R \times Y$, a single level in $X^-$, and some number (possibly zero) of levels in the symplectization $\R \times \bdy^-X$ (this is vacuous if $\bdy^-X = \nil$). This configuration is subject to similar matching and stability conditions to the above.

\subsubsection{Regularity for simple curves}\label{subsubsec:regularity}

Ideally one would like to say for example that $\calM_{X,A}^J(\Gamma^+;\Gamma^-)$ is a smooth manifold and that $\ovl{\calM}_{X,A}^J(\Gamma^+;\Gamma^-)$ is a smooth compactification, at least for a generically\footnote{Following standard usage, we will say that subset of admissible almost complex structures is {\em generic} if it is comeager, i.e. it contains a countable intersection of open dense subsets (c.f. the Baire category theorem).} chosen admissible almost complex structure $J$.
Indeed, the Cauchy--Riemann equation defining a curve $u \in \calM_{X}^J(\Gamma^+;\Gamma^-)$ is a Fredholm problem of index
\begin{align}
\ind(u) = (n-3)(2- s^+ - s^-) + \sum_{i=1}^{s^+} \cz_\tau(\gamma_i^+) - \sum_{j=1}^{s^-} \cz_\tau(\gamma_j^-) + 2c_1^\tau(u). 
\end{align}
Here as usual we put $\dim X = 2n$, $\tau$ corresponds to a choice of framing of each of the involved Reeb orbits, and $c_1^\tau(u)$ denotes the first Chern number of $u$ relative to this choice of framings, i.e. the signed count of zeros of a section of $u^*TX$ which is constant with respect to the given framings.
Then $\ind(u)$ gives the expected (or ``virtual'') dimension of $\calM_{X}^J(\Gamma^+;\Gamma^-)$ near the curve $u$. 

If $u$ is {\em regular}, i.e. its linearized Cauchy--Riemann operator is surjective, then by a version of the implicit function theorem it can be shown that $\calM_{X,A}^J(\Gamma^+;\Gamma^-)$ is indeed a smooth manifold near $u$. Unfortunately, we cannot in general arrange that all elements $u \in \calM_{X,A}^J(\Gamma^+;\Gamma^-)$ are regular for generic $J$, due to the existence of multiple covers, which frequently appear with higher than expected dimension (e.g. they appear despite having negative index). 
Defining SFT in generality therefore necessitates the use of virtual perturbations (c.f.  Remark~\ref{rmk:virtual}).

Nevertheless, by standard techniques we can fortunately achieve regularity for {\em simple} curves. 
Namely, by \cite[Thm. 6.19]{wendl2016lectures}, any nonconstant asymptotically cylindrical $J$-holomorphic curve can be factored into a degree $\kappa$ holomorphic map between punctured Riemann surfaces, followed by a $J$-holomorphic curve which is an embedding apart from finitely many critical points and self-intersection points. We call $\kappa$ the {\em covering multiplicity} of $u$, and $u$ is simple if and only if we have $\kappa = 1$.
We denote the subspace of simple curves by $\calM^{J,\simp}_{X,A}(\Gamma^+;\Gamma^-) \subset \calM^J_{X,A}(\Gamma^+;\Gamma^-)$.

Any simple curve $u$ is {\em somewhere injective}, i.e. there is a point $z$ in its domain such that $du|_z \neq 0$ and $u^{-1}(u(z)) = z$ (see e.g. \cite[\S2.5]{JHOL}). 
A standard argument shows that, for any neighborhood $U$ of $u(z)$ and a generic perturbation $\wt{J}$ of $J$ supported in $U$, any $\wt{J}$-holomorphic curve in $X$ with a somewhere injective point mapping to $U$ is regular.
By leveraging this idea with some care, one can show that every simple curve is regular for generic $J$ (see \cite[\S 7.1]{wendl2016lectures}), and hence 
$\calM^{J,\simp}_{X}(\Gamma^+;\Gamma^-)$ is a smooth oriented\footnote{For a discussion of how to assign orientations SFT moduli spaces see e.g. \cite[\S 11]{wendl2016lectures}. Strictly speaking this only applies when all asymptotic Reeb orbits are good (see e.g. \cite[Def. 11.6]{wendl2016lectures}), which will be the case for all examples considered in this paper.}
manifold of dimension $\ind(u)$.

Similarly, simple curves in a symplectization are regular for generic $J$. In the case of a symplectization, some care is needed due to the $\R$-symmetry (c.f. the discussion in \cite[\S 8]{wendl2016lectures}). Assuming there are no trivial cylinders in the moduli space $\calM^{J,\simp}_Y(\Gamma^+;\Gamma^-)$, the quotient $\calM^{J,\simp}_Y(\Gamma^+;\Gamma^-)/\R$ is a smooth oriented manifold of dimension $\ind(u) -1$.
In particular, since this is necessarily nonnegative, nontrivial simple curves in a symplectization must appear with index at least $1$ for generic admissible $J$. 

In the case of a parametrized moduli space $\calM^{\{J_t\}}_{X}(\Gamma^+;\Gamma^-)$, we also have that simple curves lying over $t \in (0,1)$ are regular, provided that the homotopy $\{J_t\}$ is generic.
Note that regularity of $(t,u) \in \calM^{\{J_t\}}_{X}(\Gamma^+;\Gamma^-)$ does {\em not} imply regularity of $u$ as a $J_t$-holomorphic curve. 
Rather, if $(t,u)$ is regular for some $t \in (0,1)$, 
then $\calM^{\{J_t\}}_{X}(\Gamma^+;\Gamma^-)$ is a smooth oriented manifold of dimension $\ind(u) + 1$
near $(t,u)$, where $\ind(u)$ denotes the Fredholm index of $u$ as a $J_t$-holomorphic curve.
In particular, for $(t,u) \in \calM^{\{J_t\},\simp}_{X}(\Gamma^+;\Gamma^-)$ we must have $\ind(u) \geq -1$.

\sss

We will also need to know something about the structure of the SFT compactifications of one-dimensional moduli spaces, which we expect to be compact one-dimensional manifolds with boundary. Provided that all relevant curves are simple, this is indeed the case, with the necessary charts near the boundary provided by the procedure of gluing along cylindrical ends. For example, consider a generic homotopy $\{J_t\}_{t \in [0,1]}$, and assume that each of the curve components appearing in the compactification $\ovl{\calM}^{\{J_t\}}_{X}(\Gamma^+;\Gamma^-)$ is simple.
Then $\ovl{\calM}^{\{J_t\}}_{X}(\Gamma^+;\Gamma^-)$ is an oriented one-dimensional manifold whose boundary contains $\calM^{J_1}_X(\Gamma^+;\Gamma^-)$ and 
$\calM^{J_1}_X(\Gamma^+;\Gamma^-)$ (with its opposite orientation).
We defer the reader to \cite[\S 10.2.4]{wendl2016lectures} and the references therein for a more detailed discussion.

\subsubsection{Formal curves and anchors}

As a convenient device for bookkeeping, we will make use of the notion of {\em formal curves}. 
Namely, in a strong symplectic cobordism $X$, a formal curve $C$ consists of a nodal punctured surface $\Sigma$, with each puncture designated as either positive or negative, together with, for each irreducible component of $\Sigma$, the following data:
\begin{itemize}
\item
a collection of Reeb orbits $\Gamma^+= (\gamma^+_1,\dots,\gamma_{s^+}^+)$ in $\bdy^+X^+$ corresponding to the positive punctures of $\Sigma$
\item a collection of Reeb orbits $\Gamma^- = (\gamma^-_1,\dots,\gamma^-_{s^-})$ in $\bdy^-X$ corresponding to the negative punctures of $\Sigma$
\item a homology class $A_{\Sigma} \in H_2(X; \Gamma^+ \cup \Gamma^-)$.
\end{itemize}
Formal curves in a symplectization $\R \times Y$ of a strict contact manifold $Y$ are defined similarly, except that both the positive and negative Reeb orbits lie in $Y$, with homology classes $A_{\Sigma} \in H_2(Y; \Gamma^+ \cup \Gamma^-)$.
We will also additionally allow $C$ to have extra marked points decorated by ``formal'' local tangency constraints of the form $\lll \T^{m}p\rrr$ for some $m \in \Z_{\geq 0}$.
The formal curves considered in this paper will typically be connected and without any nodes, and they will always have total genus (after resolving nodes) zero.

Note that a formal curve $C$ has a well-defined Fredholm index.
For instance, in the case that $C$ is connected and without nodes or additional constraints we put
\begin{align}
\ind(C) := (n-3)(2-s^+-s^-) + \sum_{i=1}^{s^+}\cz_\tau(\gamma_i^+) - \sum_{j=1}^{s_-} \cz_\tau(\gamma_j^-) + 2c_1^\tau(A).
\end{align}
Any honest\footnote{We will call a curve ``honest'' when we wish to emphasize that it is not formal.} curve $u \in \calM^J_{X,A}(\Gamma^+;\Gamma^-)$ can be viewed as a formal curve, but a formal curve need not have any pseudoholomorphic representative. 
Consider strong symplectic cobordisms $X^+$ and $X^-$ with common contact boundary $Y = \bdy^+X^- = \bdy^-X^+$, and let $X^- \circledcirc X^+$ denote the strong symplectic cobordism obtained by concatenating them along $Y$.
Given pseudoholomorphic curves $u^- \in \calM_{X^-}(\Gamma;\Gamma^-)$ and $u^+ \in \calM_{X^+}(\Gamma^+;\Gamma)$ with shared Reeb orbit asymptotics $\Gamma$,
we can {\em formally glue} along the orbits of $\Gamma$ in a natural way to obtain a formal curve $C$ in $X^- \circledcirc X^+$.  
Importantly, note that the index is additive under this operation, i.e. we have $\ind(C) = \ind(u^-) + \ind(u^+)$.

\sss

Another important device from symplectic field theory is that of {\em anchors}, which are used to correct naively defined structure maps. For example, if $X$ is a Liouville domain, the linearized contact homology $\chlin(X)$ can be viewed as the anchor-corrected version of cylindrical contact homology, the latter not typically being well-defined without additional assumptions (see e.g. the discussion in \cite[\S 1]{HuN}).  
Suppose that $X^-$ and $X^+$ are strong symplectic cobordisms with common contact boundary $Y = \bdy^+X^- = \bdy^-X^+$, and assume that we have $\bdy^-X^- = \nil$. A pseudoholomorphic curve in $X^+$, anchored in $X^-$, consists of a two-level pseudoholomorphic building, with top level in $X^+$ and bottom level in $X^-$, such that each positive end of a component in $X^-$ is paired with a negative end of a component in $X^+$, but we allow unpaired negative ends of components in $X^+$.
We will refer to the unpaired negative ends of components in $X^+$ as the negative ends of the anchored curve.
The index of an anchored curve is by definition the sum of the indices of each component, and the topological type is that of the punctured surface given by gluing together the components of the domain along paired punctures.
Intuitively, we view a curve in $X^+$ anchored in $X^-$ as a curve in $X^+$ with some extra ``tentacles'' extending into $X^-$, noting that the level in $X^+$ may consist of more than one component. 
Similarly, if $X$ is a Liouville domain, we define curves in the symplectization $\R \times \bdy X$, anchored in $X$, in essentially the same way, but now with top level in $\R \times \bdy X$.
We can also speak of anchored formal curves, defined similarly but with each level only a formal curve (i.e. we have two levels, one in $X^+$ and one in $X^-$, each consisting of one or more formal curve components).

\subsection{Local tangency constraints}

In order to probe higher dimensional moduli spaces of pseudoholomorphic curves in the Weinstein domains $X_{\vec{d}}^{2n}$, we will need to impose additional geometric constraints to cut down dimensions.
Although there are a number of possible geometric constraints we could impose, such as multiple point constraints or blowup constraints (see e.g. the discussion in \cite[\S 5]{HSC}), the most fruitful for us are local tangency constraints.
The basic idea, pioneered by Cieliebak--Mohnke \cite{CM2}, is to require curves to pass through a generically chosen point $p$ and to be tangent to specified order to a generically chosen germ of a divisor $D$ passing through $p$.
For example, if $M$ is a closed $2n$-dimensional symplectic manifold with $A \in H_2(M)$, the count of pseudoholomorphic curves in $M$ representing the class $A$ and with tangency order $m$ (i.e. contact order $m+1$) to $D$ at $p$ gives rise to a Gromov--Witten type invariant, denoted by $\gw_{M,A}\lll \T^m p\rrr \in \Q$, which is independent of all choices. 
Note that the local tangency constraint $\lll \T^m p\rrr$ cuts down the expected dimension by $2n + 2m - 2$.
These counts are defined in \cite{McDuffSiegel_counting} using classical transversality techniques for semipositive closed symplectic manifolds, in which case they are integer-valued, computable by an explicit algorithm at least in dimension four.
We can also incorporate local tangency constraints into moduli spaces of punctures curves, and we denote the analogues of the aforementioned moduli spaces by $\calM_{X,A}^J(\Gamma^+;\Gamma^-)\lll \T^m p \rrr$, $\calM_{Y,A}^J(\Gamma^+;\Gamma^-)\lll\T^mp\rrr$, and so on.

As explained in \cite[\S 4]{McDuffSiegel_counting}, one can equivalently replace the local tangency constraint with a skinny ellipsoidal constraint. Namely, after removing a small neighborhood of $p$ which is symplectomorphic to a sufficiently skinny $2n$-dimensional ellipsoid, curves satisfying the constraint $\lll \T^m p \rrr$ are substituted by curves with an additional negative puncture which is asymptotic to the $(m+1)$-fold cover of the smallest action Reeb orbit in the boundary of the skinny ellipsoid. This approach, while somewhat less geometrically natural, has the advantage of casting the constraint entirely within the standard framework of asymptotically cylindrical curves in strong symplectic cobordisms. For easy of exposition, we stick with the local tangency terminology and notation.

\subsection{Weinstein structure on a divisor complement}\label{subsec:div_compl}

In this subsection, we discuss the geometry and topology of complements of divisors in closed symplectic manifolds.
We first recall that there is a natural Weinstein structure on the complement of any ample simple normal crossing divisor in a smooth projective complex variety.
We then formulate Theorem~\ref{thm:div_compl1}, which gives a precise model for the Reeb dynamics.
Together with Proposition~\ref{prop:compl_cz}, this gives an explicit understanding of the actions, first homology classes, and Conley--Zehnder indices of the corresponding closed Reeb orbits.

Let $M^{2n}$ be a smooth complex projective variety,
and let $D \subset M$ be an ample divisor, i.e. $D = \sigma^{-1}(0)$ is the zero set of a holomorphic section $\sigma$ of an ample line bundle $\calL \rightarrow M$. 
We assume that $D$ is a simple normal crossing divisor, i.e. each irreducible component is smooth, and near each point of $D$ there are local holomorphic coordinates $z_1,\dots,z_n$ such that $D$ is cut out by the equation $z_1\dots z_k = 0$ for some $1 \leq k \leq n$.
Recall that ampleness of $\calL$ is equivalent to positivity, i.e. the existence of a Hermitian inner product $\langle -,-\rangle$ on $\calL$ such that curvature with respect to the Chern connection is a K\"ahler form.
Given a holomorphic section $\sigma$, this is equivalent to the function $\phi := -\log ||\sigma||$ being a strictly plurisubharmonic function on $X := M \setminus D$, where $||-||$ is the norm corresponding to $\langle-,-\rangle$.
In this case, $-dd^\C\phi$ extends to a K\"ahler form on $M$.

By \cite[Lem. 4.3]{Seidel_biased_view}, the critical points of $\phi$ form a compact subset of $X$.
In particular, after a small perturbation we can assume that $\phi$ is a Morse function. Then since $\phi$ is exhausting, $(X,J)$ is a Stein manifold, and the restriction to $\{\phi \leq C\}$ is a Stein domain for any $C > 0$ sufficiently large.
Note that the Liouville vector field dual to $\la := -d^{\C}\phi$ is not complete, though this can easily be rectified by postcomposing $\phi$ with a suitable function $\psi: \R \rightarrow \R$ (c.f. \cite[Prop. 2.11]{cieliebak2012stein}), after which $(M \setminus D,\la,\psi \circ \phi)$ becomes a Weinstein manifold.
As explained in \cite[\S 4a]{Seidel_biased_view}, it follows from Hironaka's resolution of singularities that any smooth complex affine variety can be presented in this way as the complement of an ample simple normal crossing divisor in a smooth projective variety, and moreover the resulting Stein manifold is independent of all choices (the compactifying divisor, the Hermitian metric, etc) up to Stein deformation equivalence.

\sss

Now let $D_1,\dots,D_k$ denote the irreducible components of an ample normal crossing divisor $D$ in a smooth complex projective variety $M$,
and consider a nonzero tuple $\vec{v} = (v_1,\dots,v_k) \in \Z_{\geq 0}^k \setminus \{\vec{0}\}$, which we suppose has exactly $r \geq 1$ nonzero components.
 For future reference, we introduce some additional notation:
\begin{itemize}
\item Let $D_{\vec{v}}$ denote the intersection of all those $D_i$ for which $v_i \neq 0$.
\item Let $\mr{D}_{\vec{v}} := D_{\vec{v}} \setminus \bigcup\limits_{i\;:\; v_i = 0}D_i$ denote the open stratum of $D_{\vec{v}}$.
\item Let $ND_{\vec{v}} \rightarrow D_{\vec{v}}$ denote the normal bundle to $D_{\vec{v}} \subset M$.
There is a natural reduction of the structure group to $U(1)^{\times r}$, and in particular we can locally identify the fibers with $\C^{\times r}$ in a manner which preserves the splitting.
We will also sometimes identify $ND_{\vec{v}}$ with a small neighborhood of $D_{\vec{v}}$ in $M$
\item Let $S_{\vec{v}}$ denote the $\mathbb{T}^{r}$ bundle over $D_{\vec{v}}$ given by $(ND_{\vec{v}} \setminus D) / \R_{> 0}^r$ and let $\mr{S}_{\vec{v}} \rightarrow \mr{D}_{\vec{v}}$ denote its restriction to $\mr{D}_{\vec{v}}$
\item Let $S_{\vec{v}}/S^1$ denote the $\mathbb{T}^{r-1}$-bundle over $D_{\vec{v}}$ given by quotienting $S_{\vec{v}}$ by the restriction of the natural free $\mathbb{T}^r$ action to the circle $\{ t\vec{v}\;:\; t \in \R\} \subset \R^r/\Z^r =  \mathbb{T}^r$,
 and let $\mr{S}_{\vec{v}}/S^1$ denote its restriction to $\mr{D}_{\vec{v}}$.
\item For $i = 1,\dots,k$, let $c_i$ denote a small disk in $M$ which intersects $D_i$ once transversally and negatively and is disjoint from the other divisor components,
and let $[\bdy c_i] \in H_1(X)$ denote the homology class of its boundary.
\end{itemize}

\sss

In order to discuss the action filtration on a divisor complement, we also recall the notion of wrapping numbers.
\begin{definition}\cite{mcleanslides}
Assume that $M$ is a smooth complex projective variety with a divisor $D = \sigma^{-1}(0)$, where $\sigma$ is a holomorphic section of an ample line bundle $\calL \rightarrow M$.
For $i = 1,\dots,k$, the $i$th {\em holomorphic wrapping number} is minus the vanishing order of $\sigma$ along $D_i$. 
\end{definition}
\NI There is also a purely symplectic analogue given as follows. Following  \cite{mclean2012growth, tehrani2018normal}, recall that a {\em symplectic simple normal crossing (SNC) divisor} $D$ in a symplectic manifold $(M,\omega)$ consists of a collection of transversely intersecting symplectic submanifolds $D_1,\dots,D_k \subset M$ such that each partial intersection $D_I := \bigcap\limits_{i \in I}D_i,\; I \subset \{1,\dots,k\}$, is a symplectic submanifold, and the ``symplectic orientation'' on $D_I$ agrees with the ``intersection orientation''.
We note that this last condition is equivalent to the existence of a compatible almost complex structure $J$ for $(M,\omega)$ which makes each $D_i$ $J$-holomorphic.
\begin{definition}\cite{mcleanslides}
Let $(M,\omega)$ be a symplectic manifold with a symplectic SNC divisor $D = D_1 \cup \dots \cup D_k$, and let $\la$ be a one-form on $M \setminus D$ with $d\la = \omega|_{M \setminus D}$.
Let $N$ be the closure of a small neighborhood of $D$ with smooth boundary which deformation retracts onto $D$, and let $\rho: N \rightarrow [0,1]$ be a function which is equal to $1$ near $D$ and vanishes near $\bdy N$.
Let $\wt{\omega}$ be the two-form on $N$ given by $\omega$ near $D$ and $d(\rho \la)$ away from $D$.
The {\em symplectic wrapping numbers} $\ww_1,\dots,\ww_k$ are the unique coefficients such that $-\sum_{i=1}^k \ww_i[D_i] \in H_{2n-2}(N;\R)$ is Poincar\'e--Lefschetz dual to $[\wt{\omega}] \in H^2(N,\bdy N;\R)$.
\end{definition}

\sss

The following theorem summarizes most of what we will need to know about the symplectic geometry of divisor complements:
\begin{thm}[{see \cite[\S5]{mclean2012growth} or \cite[Thm. 2.17]{tehrani2018normal}}]\label{thm:div_compl1}
Fix $\CC \in \R_{> 0}$ arbitrarily large and $\eps \in \R_{>0}$ arbitrarily small. 
Let $M$ be a smooth complex projective variety with an ample simple normal crossing divisor $D$, and let $(X,\la,\phi)$ denote the associated Weinstein domain corresponding to the divisor complement. 
For each $\vec{v} \in \Z_{\geq 0}^k \setminus \{\vec{0}\}$, pick an exhausting Morse function\footnote{Note that $\mr{S}_{\vec{v}}/S^1$ is typically noncompact, hence we must specify the behavior at infinity.
Since $f_{\vec{v}}$ is exhausting, after choosing a Riemannian metric which is Morse--Smale for $f_{\vec{v}}$, the resulting Morse cohomology is isomorphic to the ordinary cohomology of $\mr{S}_{\vec{v}}/S^1$. By Poincar\'e duality, this is isomorphic the Borel--Moore homology of $\mr{S}_{\vec{v}}/S^1$ after a degree shift.
}
 $f_{\vec{v}}: \mr{S}_{\vec{v}}/S^1 \rightarrow \R$.
After a Weinstein homotopy of $(X,\la,\phi)$ and a deformation of $D$ through symplectic SNC divisors, 
we can find a K\"ahler form $\omega$ on $M$ and an embedding $\iota: X \hookrightarrow M$ such that:
\begin{enumerate}
	\item[(1)] $N := \ovl{M \setminus \iota(X)}$ deformation retracts onto $D$
	\item[(2)] $\iota^*\omega = d\la$, and $\iota_*\la$ extends to a one-form $\wt{\la}$ on $M \setminus D$ such that $d\wt{\la} = \omega|_{M \setminus D}$
	\item[(3)] the symplectic wrapping numbers of $D$ coincide with the holomorphic wrapping numbers. 
\end{enumerate}
Moreover, the contact form $\alpha := \la|_{\bdy X}$ has nondegenerate Reeb dynamics, where:
\begin{enumerate}
	\item[(4)] the Reeb orbits of $(\bdy X,\alpha)$ of period less than $\CC$ are in one-to-one correspondence with the set of critical points $\crit(f_{\vec{v}})$ of $f_{\vec{v}}$ as $\vec{v}$ ranges over $\Z^k_{\geq 0}\setminus \{\vec{0}\}$ such that $-\sum_{i=1}^k v_i\ww_i \leq \CC$.
	\item[(5)] the Reeb orbit $\gamma_{\vec{v}}^A$ corresponding to the tuple $\vec{v} \in \Z^k_{\geq 0}\setminus \{\vec{0}\}$ and critical point $A \in \crit(f_{\vec{v}})$ lies in the homology class $\sum_{i=1}^k v_i[\bdy c_i] \in H_1(X)$
	\item[(6)] the action of $\gamma_{\vec{v}}^A$ is given by $-\sum_{i=1}^k v_i\ww_i$, up to a discrepancy of $\eps$.
\end{enumerate}
\end{thm}
\NI In the sequel, we will typically assume that the above theorem has already been applied to a given divisor complement, and by slight abuse of notation we view $\iota$ as an inclusion $X \subset M$ and denote $\wt{\la}$ again by $\la$.

\begin{remark}\hspace{1cm}
\begin{itemize}
\item Note that for any admissible almost complex structure $J$ on the symplectic completion of $N$, any $J$-holomorphic curve in $N$ must intersect $D$. Indeed, otherwise by (2) we can apply Stokes' theorem together with nonnegativity of energy to get a contradiction.
\item For most of the pseudoholomorphic curve arguments in this paper we have an a priori upper bound on the actions of Reeb orbits which could arise. This means we can simply take $\CC$ to be sufficiently large and safely ignore all Reeb orbits of action greater than $\CC$.
\end{itemize}
\end{remark}

\sss

We end this subsection by discussing the Conley--Zehnder indices of the Reeb orbits $\gamma_{\vec{v}}^A$ described in Theorem~\ref{thm:div_compl1}.
In the context of Theorem~\ref{thm:div_compl1}, each Reeb orbit $\gamma_{\vec{v}}^A$ bounds a small spanning disk $u$ in $N$ which has homological intersection $v_i$ with $[D_i]$ for $i = 1,\dots,k$.
\begin{lemma}\label{lem:cz_well_def}
Assume that the Poincar\'e dual to $c_1(TM)$ can be expressed in the form $h_1[D_1] + \dots + h_k[D_k] \in H_{2n-2}(M;\R)$ for some $h_1,\dots,h_k \in \R$. 
Let $\gamma$ be a Reeb orbit in $\bdy X$, and let $u,u': \D^2 \rightarrow M$ be two bounding disks such that $[u] \cdot [D_i] = [u'] \cdot [D_i]$ for $i = 1,\dots,k$.
We have $\cz_u(\gamma) = \cz_{u'}(\gamma)$.
\end{lemma}
\begin{proof}
Let $S$ denote the sphere in $N = M \setminus X$ obtained by gluing $u$ to $u'$ with its opposite orientation, and let $[S] \in H_2(N)$ denote the corresponding homology class. 
Note that the homological intersection number $[S] \cdot [D_i]$ vanishes for $i = 1,\dots k$.
Then by \S\ref{subsubsec:CZ} we have
\begin{align}
\cz_{u}(\gamma) - \cz_{u'}(\gamma) = 2 c_1(M) \cdot [S]
= \sum_{i=1}^k h_i [D_i] \cdot[S] = 0.
\end{align}
\end{proof}
By default, we will compute the Conley--Zehnder index of the Reeb orbit $\gamma_{\vec{v}}^A$ via a small spanning disk $u$ as in Lemma~\ref{lem:cz_well_def} which satisfies $[u]\cdot [D_i] = v_i$ for $i = 1,\dots,k$.
We denote this trivialization of $TM$ along the Reeb orbits $\gamma_{\vec{v}}^A$ by $\triv$, and we denote the corresponding Conley--Zehnder index by $\cz_{\triv}$.
\begin{prop}\label{prop:compl_cz}\cite[\S2]{ganatra2020symplectic}
For each $\vec{v} \in \Z_{\geq 0}^k \setminus \{\vec{0}\}$ with $D_{\vec{v}} \neq \nil$, we have
\begin{align}
\cz_{\triv}(\gamma_{\vec{v}}^A) = n-1 - |A| - 2\sum_{i=1}^k v_i.
\end{align}
\end{prop}
\NI Here $|A|$ denotes the Morse index of $A \in \crit(f_{\vec{v}})$. Putting 
\begin{align}
\delta = \delta(\gamma_{\vec{v}}^A) := n-1 - |A|,  
\end{align}
 we have alternatively
$\cz_{\triv}(\gamma_{\vec{v}}^A) = \delta - 2\vec{v}\cdot \vec{1}$ with $\delta \leq n-1$.
Here we use the shorthand $\vec{1} := \underbrace{(1,\dots,1)}_{k}$.

\begin{remark}\label{rmk:gradings_antican}
Consider the case that $D$ is anticanonical, with irreducible components $D_1,\dots,D_k$, and
let $\kappa$ be a meromorphic section of the canonical bundle of $M$ which is nonvanishing away from $D$. For $i = 1,\dots,k$, let $a_i$ denote the order of vanishing (possibly negative) of $\kappa$ along $D_i$.
In this case we can alternatively compute Conley--Zehnder indices with respect to the holomorphic volume form $\kappa|_{X}$, and we have
\begin{align}
\cz_{\kappa}(\gamma^A_{\vec{v}}) = n - 1 - |A| - 2\sum_{i=1}^k v_i(a_i+1).
\end{align}

\end{remark}

\begin{example}As a simple example, let us specialize to the case of a four-dimensional hyperplane complement $X_k^4$.
For $k \in \Z_{\geq 2}$, let $\Sigma_k$ denote the two-sphere with $k$ punctures.
Then for any nonzero $\vec{v} = (v_1,\dots,v_k) \in \Z_{\geq 0}^k$ we have:
\begin{itemize}
\item when there is exactly one nonzero component of $\vec{v}$, $\mr{S}_{\vec{v}}$ is diffeomorphic to
$\Sigma_{k-1} \times S^1$ and $\mr{S}_{\vec{v}}/S^1$ is diffeomorphic to $\Sigma_{k-1}$
\item when there are exactly two nonzero components of $\vec{v}$, $\mr{S}_{\vec{v}}$ is diffeomorphic to $\mathbb{T}^2$ and $\mr{S}_{\vec{v}}/S^1$ is diffeomorphic to $S^1$
\item if three or more components of $\vec{v}$ are nonzero, then $D_{\vec{v}} = \nil$.
\end{itemize}

Now choose exhausting Morse functions $f_{\vec{v}}: \mr{S}_{\vec{v}}/S^1 \rightarrow \R$ as in Theorem~\ref{thm:div_compl1}, which in this example we can assume are perfect, so that the critical points give rise to a distinguished basis for $H^*(\mr{S}_{\vec{v}}/S^1)$.
The Reeb orbits of $\bdy X_k^4$ (of period less than $\CC$) are then given explicitly as follows:
\begin{enumerate}
\item For each $\vec{v}$ with exactly one nonzero component, we have the Reeb orbit $\gamma_{\vec{v}}^A$ with $\cz_{\triv}(\gamma_{\vec{v}}^A) = 1 - \sum_{i=1}^k v_i$, corresponding to the unique basis element of $H^0(\Sigma_{k-1})$.
\item For each $\vec{v}$ with exactly one nonzero component,
we have the Reeb orbits $\gamma^A_{\vec{v}}$ with $\cz_{\triv}(\gamma^A_{\vec{v}}) =  - \sum_{i=1}^k v_i$, corresponding to the $k-2$ basis elements of $H^1(\Sigma_{k-1})$
\item For each $\vec{v}$ with exactly two nonzero components,
we have the Reeb orbit $\gamma^A_{\vec{v}}$ with $\cz_{\triv}(\gamma^A_{\vec{v}}) = 1 - \sum_{i=1}^k v_i$, corresponding to the unique basis element of $H^0(S^1)$
\item For each $\vec{v}$ with exactly two nonzero components,
we have the Reeb orbit $\gamma^A_{\vec{v}}$ with $\cz_{\triv}(\gamma^A_{\vec{v}}) = - \sum_{i=1}^k v_i$, corresponding to the 
unique basis element of $H^1(S^1)$.
\end{enumerate}
\end{example}

\section{A family of invariants}\label{sec:invariants}

In this section we describe our general family of symplectic invariants which obstruct Liouville embeddings. 
Firstly, in \S\ref{subsec:obs_cyls} we elaborate on the $S^1$-equivariant analogue of Observation~\ref{obs:vanishingSH} and the resulting Liouville embedding obstruction $\F$, which has appeared in the literature in various forms. 
In \S\ref{subsec:obs_higher}, we vastly generalize this by incorporating $\Li$ structures, and we encode this data as the invariant $\I^{\leq l}$.
 Finally, in \S\ref{subsec:obs_simp} we introduce the simplified invariant $\G\lll\T^m p\rrr$, whose computation is more tractable and will suffice for our main applications.

\subsection{Obstructions from cylinders}\label{subsec:obs_cyls}

We seek to generalize the binary phenomenon of vanishing symplectic cohomology as in Observation~\ref{obs:vanishingSH}.
For simplicity, we work throughout over $\K = \Q$.
Let $X$ be a Liouville domain, and let $e \in \sh(X)$ denote the unit in its symplectic cohomology ring with its pair of pants product. We will also use $e$ to denote the unit in the ordinary cohomology ring $H^*(X)$.
Let 
$$ \dots \rightarrow H^*(X) \rightarrow \sh^*(X) \rightarrow \sh^*_+(X) \xrightarrow{\delta} H^{*+1}(X) \rightarrow \dots$$
denote the long exact sequence coming from splitting off the generators of low action.
Note that symplectic cohomology is generally graded only by $\Z/2$, although this can upgraded to a $\Z$ grading if $c_1(X) = 0$, and by default we grade Hamiltonian orbits by $n - \cz$ (as usual, $n$ denotes half the real dimension of $X$).
Since $\sh(X)$ is a unital $\K$-algebra, it vanishes if and only if we have $e = 0 \in \sh(X)$, or equivalently if $e \in H^0(X)$ lies in the image of $\delta$.

Passing to $S^1$-equivariant symplectic cohomology produces more information as follows. We refer the reader to e.g. \cite{Bourgeois-Oancea_equivariant,Seidel_biased_view,ganatra2019cyclic, Gutt-Hu} for more on the technical setup and structural properties of $\sh_{S^1}$.
By slight abuse of notation, let $e$ also denote the image of $e \in \sh(X)$ under the ``erase'' map $\sh^*(X) \rightarrow \sh^*_{S^1}(X)$. It should be emphasized that $\sh_{S^1}(X)$ does not have a product, although it does have a Lie bracket of degree $-2$, which corresponds to the ``string bracket'' in the case that $X$ is cotangent bundle.
In particular, the vanishing of $e \in \sh_{S^1}(X)$ does not necessarily imply that $\sh_{S^1}(X)$ itself vanishes.

We put $\K[u^{-1}]$ as a shorthand for the $\K[u]$-module $\K[u,u^{-1}] / (u\K[u])$, where $u$ has degree $2$.
From the algebraic point of view (c.f. \cite[\S 2]{ganatra2019cyclic}), $\sc(X)$ is endowed with the structure of an $S^1$-complex, i.e. we have a sequence of operations $\delta^i: \sc^*(X) \rightarrow \sc^{*+1-2i}$ for $i \in \Z_{\geq 0}$, with $\delta^0$ the differential and $\delta^1$ descending to the BV operator, such that we have $\sum\limits_{i+j = k}\delta^i \circ \delta^j = 0$ for all $k \in \Z_{\geq 0}$.
Then $\sh_{S^1}(X)$ is the homology of the positive cyclic chain complex $(\sc_{S^1}(X),\bdy_{S^1})$ with
 $\sc_{S^1}(X) := \sc(X) \otimes \K[u^{-1}]$ and $\bdy_{S^1} := \sum\limits_{i=0}^{\infty} u^i \delta^i$.

Similar to the nonequivariant case, we have the connecting map 
$$\delta_{S^1}: \sh^*_{S^1,+}(X) \rightarrow H^{*+1}_{S^1}(X),$$
which is a map of $\K[u]$-modules. Here $H^{*}_{S^1}(X)$ is canonically identified with $H^{*}(X) \otimes \K[u^{-1}]$.
Let $P_0: H_{S^1}(X) \approx H^*(X) \otimes \K[u^{-1}]\rightarrow \K[u^{-1}]\langle e \rangle$ be induced by the map $H^*(X) \rightarrow H^0(X) = \K\langle e\rangle$ projecting to degree zero (we assume that $X$ is connected, so that $H^0(X)$ is one-dimensional and generated by $e$).
We also use $e$ to denote the image of the unit under the natural map $H^*(X) \rightarrow H^*_{S^1}(X)$.

\begin{definition}
Let $\F(X) \in \Z_{\geq 0} \cup \{\infty\}$ be the smallest $k$ such that $u^{-k}e$ does not lie in the image of $P_0 \circ \delta_{S^1}: \sh^*_{S^1,+}(X) \rightarrow \K[u^{-1}]\langle e \rangle$. If no such $k$ exists, we put $\F(X) = \infty$.
\end{definition}
\NI Note that if $u^{-k}e$ lies in the image of $P_0 \circ \delta_{S^1}$, then by $\K[u]$-linearity so do the elements $u^{-k+1}e,\dots,u^{-1}e,e$.
Using Viterbo functoriality and standard invariance properties
for symplectic cohomology, we have:
\begin{prop}\label{prop:F}
For a Liouville domain $X$, $\F(X)$ is independent of all choices and invariant under Liouville deformation equivalences.
Given a Liouville embedding $X \es X'$ of Liouville domains $X,X'$, we have $\F(X) \geq \F(X')$.
\end{prop}

Proposition~\ref{prop:F} is proved in \cite{Gutt-Hu} with mostly quantitative applications in mind.
Indeed, the authors define a sequence of symplectic capacities $c_1^{\op{GH}}(X) \leq c_2^{\op{GH}}(X) \leq c_3^{\op{GH}}(X) \leq \dots$ valued in $\R_{> 0} \cup \{\infty\}$, and $\F(X)$ is equivalent to the number of such capacities which are finite. For example, if $X$ is a star-shaped domain in $\C^n$, then all of these capacities are finite and hence we have $\F(X) = \infty$.

A closely related notion of {\em higher dilation} is introduced in \cite{zhao2016periodic}, based on the original definition of {\em dilation} from \cite{seidel2012symplectic}.
Recall that a Liouville domain $X$ admits a dilation if $\Delta(x) = e \in \sh^0(X)$ for some $x \in \sh^{1}(X)$, where $\Delta: \sh^*(X) \rightarrow \sh^{*-1}(X)$ denotes the BV operator. It admits a higher dilation if and only if we have $\F(X) = \infty$.
The higher dilation concept is refined in \cite{Zhou_semidilation} by declaring that $X$ has a {\em $k$-dilation} if $e \in \sh_{S^1}(X)$ is killed on the $(k+1)$st page of the spectral sequence induced by the $u$-adic filtration.
Then $\F(X) > 0$ if and only if $X$ admits a $k$-dilation for some $k$ (or equivalently $X$ admits a cyclic dilation in the sense of \cite{li2019exactCY}).

\begin{remark}
The definition of $\F(X)$ is also formally similar to the notion of {\em algebraic torsion} of contact manifolds introduced in \cite{latschev2011algebraic}, which provides a hierarchy of symplectic fillability obstructions. Whereas the former involves only genus zero curves and applies to Liouville domains, the latter is based on higher genus symplectic field theory and applies to contact manifolds which cannot be strongly filled.
\end{remark}

\subsection{Incorporating curves with several positive ends}\label{subsec:obs_higher}

The invariant $\F(X)$ is unfortunately not strong enough to tackle Problem~\ref{prob:hyperplane_comp} or the more general Problem~\ref{prob:hyp}.
Indeed, recall that $\sh_{S^1,+}^*(X)$ has a natural grading by $H_1(X)$, corresponding to the homology classes of generator loops. Moreover, the map $\delta_{S^1}: \sh_{S^1,+}^*(X) \rightarrow H^{*+1}(X)$ is compatible with this grading, which means that it is supported on the graded piece of the trivial class in $H_1(X)$. 
However, for $\vec{d} = (d_1,\dots,d_k)$ with $k \geq n+1$, none of the Reeb orbits in $\bdy X_{\vec{d}}^{2n}$ with our preferred contact form are contractible in $X_{\vec{d}}^{2n}$ (see \S\ref{subsec:hypersurf_compl}).
It then follows that $\delta_{S^1}$ is trivial, and hence:
\begin{lemma}
For $\vec{d} = (d_1,\dots,d_k) \in \Z_{\geq 1}^k$ with $k \geq n+1$, we have $\F(X_{\vec{d}}^{2n}) = 0$.
\end{lemma}
\NI In principle one could imagine extracting more information using the product or Lie bracket on $\sh$ or the Lie bracket on $\sh_{S^1}$, which are based on pseudoholomorphic curves with not one but two positive ends.
However, for $X_{\vec{d}}^{2n}$ with $k \geq 2n + 1$ these operations are purely ``topological'' by first homology considerations (see Theorem~\ref{thm:div_compl1}(5) and \eqref{eq:hom_of_X_vec_d} in \S\ref{sec:computationsI}), i.e. they vanish unless one of the inputs comes from $H^*(X)$, and hence it seems unlikely that they carry any nontrivial embedding obstructions. Note that the ring structure on $\sh(X_{\vec{d}}^{2n})$ in this case is isomorphic to the log cohomology ring from \cite[\S3]{ganatra2020symplectic}.

\sss

To go further, we can consider the chain-level $\Li$ structure on $\sc_{S^1}(X)$, which encodes certain counts of genus zero pseudoholomorphic curves with an arbitrary number of positive ends and one negative end.
By upgrading the map ${P_0 \circ \delta_{S^1}: \sc_{S^1,+}(X) \rightarrow \K[u^{-1}]}$ to an $\Li$ homomorphism and appealing to the bar construction framework of \cite{HSC}, we obtain a large family of symplectic invariants which behave well with respect to Liouville embeddings.
Unfortunately, a complete description of this $\Li$ algebra has not yet appeared in the literature, and its relationship to the direct geometric approach of \S\ref{sec:computationsII} is somewhat opaque.

Rather, in this paper, following \cite{HSC} we replace $\sc_{S^1,+}(X)$ with its SFT counterpart $\cclin(X)$, 
and we replace the connecting map $\delta_{S^1}$ with a map counting curves with local tangency constraints.
Here $\cclin(X)$ denotes {\em linearized contact chains}, the chain complex computing linearized contact homology $\chlin(X)$.\footnote{Adopting the notational convention of \cite{bourgeois2012effect}, $\cclin$ by default denotes the chain level object, and we use the boldface $\chlin$ to denote its homology.}
We refer the reader to \cite{EGH2000} for a structural description of linearized contact homology and to e.g. \cite{fish2018lectures, pardon2019contact,HuN, bao2015semi, ishikawa2018construction} for some technical approaches to its construction.\footnote{At the time of writing the homotopy relations needed to establish suitable invariance properties of linearized contact homology have not yet been adequately addressed in many of these approaches (see e.g. the discussion of cylindrical contact homology in \cite[\S1.8]{pardon2019contact}).}
An isomorphism between positive $S^1$-equivariant symplectic cohomology and linearized contact homology (assuming $\K = \Q$) is described in \cite{Bourgeois-Oancea_equivariant}.
In the following discussion, we defer to \cite[\S3]{HSC} for a more detailed discussion of the $\Li$ structure on $\cclin$ and \cite[\S5]{HSC} for augmentations defined by local tangency constraints.

As a $\K$-module, $\cclin(X)$ is freely spanned by the good\footnote{Recall that a Reeb orbit is {\em good} if it is a cover of another Reeb orbit whose Conley--Zehnder index has the opposite parity. All of the Reeb orbits appearing in the main examples in this paper are good.} Reeb orbits of $\bdy X$.
Using the $n - \cz$ grading convention, the $\Li$ operations $\ell^1,\ell^2,\ell^3,\dots$ are such that $\ell^k: \odot^k \cclin(X) \rightarrow \cclin(X)$ has degree $4-3k$.
Here $\ell^k$ counts (possibly virtually perturbed) index one asymptotically cylindrical pseudoholomorphic curves in the symplectization $\R \times \bdy X$ modulo target translations, anchored in $X$, with $k$ positive punctures and one negative puncture.
Strictly speaking this makes $\cclin(X)$ into a {\em shifted} $\Li$ algebra, and following \cite{chscI} it is convenient to 
instead grade by $n - \cz - 3$, so that each operation $\ell^k: \odot^k\cclin(X) \rightarrow \cclin(X)$ has degree $+1$ (here $\odot^k$ denotes the $k$-fold graded symmetric tensor product over $\K$).

The {\em bar complex} $\bar\cclin(X)$ is by definition the chain complex given by the reduced symmetric tensor algebra $\ovl{S}\cclin(X) = \bigoplus\limits_{k=1}^\infty \odot^k \cclin(X)$, equipped with the degree $+1$ bar differential $\wh{\ell}: \ovl{S}\cclin(X) \rightarrow \ovl{S}\cclin(X)$. 
For $l \in \Z_{\geq 1}$, let $\bar^{\leq l}\cclin(X) \subset \bar\cclin(X)$ denote the subcomplex spanned by elements of tensor word length at most $l$. 
We also put $\bar^{\leq \infty}\cclin(X) := \bar\cclin(X)$.

Let $\aug\lll \TT \rrr: \cclin(X) \rightarrow \K[t]$ denote the $\Li$ homomorphism given by counting curves in $X$ with local tangency constraints. 
Here $\K[t]$ is viewed as an abelian $\Li$ algebra, i.e. all operations vanish identically, graded such that $t^k$ has degree $-4-2k$.
More precisely, this $\Li$ homomorphism consists of terms $\aug^k\lll \TT p \rrr: \odot^k \cclin(X) \rightarrow \K[t]$ for $k \in \Z_{\geq 1}$.
For Reeb orbits $\gamma_1,\dots,\gamma_k$, the structure coefficient 
 $\langle \aug^k \lll \TT p\rrr(\gamma_1,\dots,\gamma_k),t^m\rangle$ 
 counts $k$-punctured spheres in the (symplectic completion of) $X$ 
with positive Reeb orbit asymptotics $\gamma_1,\dots,\gamma_k$, and having tangency order $m$ (i.e. contact order $m+1$) to a generic local divisor $D$ at a point $p \in X$.
Let $\wh{\aug}\lll \TT p\rrr: \bar \cclin(X) \rightarrow \bar\K[t]$ denote the induced map on bar complexes, which has degree zero with our conventions. Note that $\bar\K[t]$ is simply $\ovl{S}\K[t]$ equipped with trivial bar differential, and we identify its homology $H\bar\K[t]$ with $\ovl{S}\K[t]$.

\begin{definition}For $l \in \Z_{\geq 1} \cup \{\infty\}$, let $\I(X) \subset \ovl{S}\K[t]$ denote the image of the homology level map $H \wh{\aug}\lll \TT p\rrr: H\bar \cclin(X) \rightarrow \ovl{S}\K[t]$. 
More refinedly, let $\I^{\leq l}(X)$ denote the image of the same homology level map after restricting to $H \bar^{\leq l}\cclin(X)$.
\end{definition}

As in \cite{HSC}, by the functoriality package for $\cclin(X)$ and $\aug\lll \TT\rrr$ we have: 
\begin{thm}\label{thm:I}
For a Liouville domain $X$ and $l \in \Z_{\geq 1} \cup \{\infty\}$, $\I^{\leq l}(X)$ is independent of all choices and invariant under Liouville deformation equivalences.
Given a Liouville embedding $X \es X'$ of Liouville domains $X,X'$, we have $\I^{\leq l}(X') \subset \I^{\leq l}(X)$.
\end{thm}
\NI Note that $l=1$ corresponds to the case of curves with only one positive end as in $\F(X)$.
At the other extreme, $\I(X) = \I^{\leq \infty}(X)$ corresponds to the case of no restrictions on the number of positive ends.

\begin{remark}
It is also natural to consider analogous invariants defined by replacing the local tangency constraint by some other geometric constraint. For example, we can consider curves with a fixed number of generic point constraints.
As explained in \cite[\S5]{HSC}, this necessitates the more elaborate formalism of rational symplectic field theory.
\end{remark}

\subsection{The simplified invariant $\G\lll \T^m p\rrr$}\label{subsec:obs_simp}

The invariant $\I^{\leq l}$ provides strong obstructions to the existence of Liouville embeddings $X \es X'$ between Liouville domains $X,X'$. However, its full computation requires a rather strong understanding of both the chain-level $\Li$ algebra $\cclin(X)$ and the $\Li$ homomorphism $\aug\lll \TT p\rrr: \cclin(X) \rightarrow \K[t]$, which is quite challenging to achieve for all but the simplest examples.
Also, the map $\wh{\aug}\lll \TT p\rrr$ has a geometric interpretation as counting curves with several components, but this is somewhat unintuitive.

In order to have a more easily interpretable invariant, we now introduce a simplified invariant whose computation is typically more tractable and which involves only irreducible curves. 
To achieve this, let $\pi_k: \ovl{S}\K[t] \rightarrow \odot^k \K[t]$ denote the projection to the subspace spanned by elements of word length $k$.
In particular, we have $\pi_1: \ovl{S}\K[t] \rightarrow \K[t]$.

As a warmup, consider the condition that $1 \in \K[t]$ lies in $\pi_1(\I^{\leq l}(X))$. Heuristically, to first approximation this means there is a rigid curve in $X$ which passes through a generic point constraint and has at most $l$ positive ends.
However, to make this more accurate we need to keep in mind that (a) in addition the asymptotic orbits must define a cycle with respect to the bar complex differential, (b) this cycle could be a linear combination of several elementary tensors, and (c) the relevant curves are possibly anchored and virtually perturbed, and they are counted algebraically with signs.

More generally, we can replace the point constraint $\lll p \rrr$, which corresponds to $1 \in \K[t]$, with a local tangency constraint $\lll \T^mp\rrr$, which corresponds to $t^m \in \K[t]$.
\begin{definition}
Let $\G\lll \T^m p \rrr(X) \in \Z_{\geq 1}\cup \{\infty\}$ denote the smallest $l$ such that $\pi_1(\I^{\leq l}(X))$
has nontrivial image under the projection map $\pi_{t^m}: \K[t] \rightarrow \K\langle t^m \rangle$.\footnote{If $c_1(X) = 0$, by grading considerations this is equivalent to saying that $t^m$ lies in $\pi_1(\I^{\leq l}(X))$.}
If no such $l$ exists, we put $\G\lll \T^m p \rrr(X) = \infty$.
\end{definition}
\NI Heuristically, to first approximation $\G\lll \T^m p\rrr(X)$ records the smallest number of positive ends of a rigid curve in $X$ satisfying a $\lll\T^m p\rrr$ constraint.
The following is immediately extracted from Theorem~\ref{thm:I}:
\begin{thm}
For a Liouville domain $X$ and $m \in \Z_{\geq 0}$, $\G\lll \T^m p \rrr(X)$ is independent of all choices and invariant under Liouville deformation equivalences. Given a Liouville embedding $X \es X'$ of Liouville domains $X,X'$, we have $\G(X) \leq \G(X')$.
\end{thm}

\begin{remark}
A closely related invariant based on an $\Li$ structure on $\sc_{S^1}$ and defined in terms of the $u$-adic spectral sequence is mentioned in \cite{seidelslides}.
\end{remark}

\section{Computations for hypersurface complements I: SFT version}\label{sec:computationsI}

The goal of this section is to prove Theorem~\ref{thm:main_G_computation}. 
We begin with some generalities on projective hypersurface complements in \S\ref{subsec:hypersurf_compl}.
The proof then proceeds by establishing a lower bound in \S\ref{subsec:lower_bound} and an upper bound in \S\ref{subsec:upper_bound}.
Finally, in \S\ref{subsec:CL}, we describe a more general framework for producing Maurer--Cartan elements.

Note that in this section we assume SFT transversality via virtual perturbations, although the arguments are independent of any specific perturbation scheme. In the next section we upgrade this to the stronger Theorem~\ref{thm:main_combinatorial} and also remove this assumption.

\subsection{Geometry of hypersurface complements in projective space}\label{subsec:hypersurf_compl}

We now specialize the discussion from \S\ref{subsec:div_compl} to the complement of a collection of generic hypersurfaces in projective space.
Here by {\em generic} we mean that each hypersurface is smooth, and the collection defines a simple normal crossing divisor.
Given a tuple $\vec{d} \in \Z_{\geq 1}^k$ with $k \in \Z_{\geq 1}$, 
we consider the corresponding Weinstein domain $X_{\vec{d}}^{2n} = \CP^n \setminus \Op(D)$,
where $D_1,\dots,D_k$ is a generic collection of (smooth) hypersurfaces in $\CP^n$ of degrees $d_1,\dots,d_k$ respectively, and we put $D := D_1 \cup \dots \cup D_k$.

After a Weinstein homotopy, we arrange that $X_{\vec{d}}^{2n}$ has geometry as in Theorem~\ref{thm:div_compl1}. In particular, for each $\vec{v} \in \Z^k_{\geq 0} \setminus \{\vec{0}\}$ we choose a 
Morse function $f_{\vec{v}}: \mr{S}_{\vec{v}}/S^1 \rightarrow \R$,
which we further assume has a unique minimum.
Let $\vec{e}_i$ denote the $i$th standard basis vector, i.e. $\vec{e}_i := (\underbrace{0,\dots,0}_{i-1},1,\underbrace{0,\dots,0}_{k-i})$.
Then there is a unique Reeb orbit (of action less than $\CC$) of $\bdy X^{2n}_{\vec{d}}$ of the form $\gamma_{\vec{e}_i}^A$ with $\cz_{\triv}(\gamma_{\vec{e}_i}^A) = n-3$, corresponding to the unique minimum $A$ of $f_{\vec{e}_i}: \mr{S}_{\vec{e}_i}/S^1 \rightarrow \R$. For future reference, we denote these orbits by $\beta_1,\dots,\beta_k$. Heuristically, these correspond to the fundamental classes of the open divisor strata $\mr{D}_{\vec{e}_1},\dots,\mr{D}_{\vec{e}_k}$.

\sss

Let $[\bdy c_i] = \beta_i$ for $i=1,\dots,k$ denote the homology classes of small loops surrounding the hypersurfaces $D_1,\dots,D_k$ as in \S\ref{subsec:div_compl}. 
Then $H_1(X^{2n}_{\vec{d}})$ is $(k-1)$-dimensional, with
\begin{align}\label{eq:hom_of_X_vec_d}
H_1(X^{2n}_{\vec{d}}) = \Z\langle [\bdy c_1],\dots,[\bdy c_k]\rangle / (d_1[\bdy c_1]+\dots + d_k[\bdy c_k]) \cong \Z^k/(\vec{d})
\end{align}
(see e.g. \cite[Prop. 2.3]{libgober2007lectures}).

\sss

For $k,n \in \Z_{\geq 1}$ and a tuple $\vec{d} = (d_1,\dots,d_k) \in \Z_{\geq 1}^k$, put $\gcd(\vec{d}) := \gcd(d_1,\dots,d_k)$.
As explained in \S\ref{subsec:div_compl}, we have a preferred trivialization $\triv$ of the symplectic vector bundle $TX_{\vec{d}}^{2n}$ over each Reeb orbit $\gamma_{\vec{v}}^A$ in $\bdy X_{\vec{d}}^{2n}$. 
Observe that a symplectic embedding $X_{\vec{d}}^{2n} \shookrightarrow X_{\vec{d'}}^{2n}$ pulls back $c_1(X_{\vec{d'}}^{2n})$ to $c_1(X_{\vec{d}}^{2n})$. In particular, if $X_{\vec{d'}}^{2n}$ is Calabi--Yau, i.e. $c_1(X_{\vec{d'}}^{2n}) = 0$, then the same must also be true of $X_{\vec{d}}^{2n}$.
The Calabi--Yau condition for $X_{\vec{d}}^{2n}$ is equivalent to the existence of
$a_1,\dots,a_k \in \Z$ such that $\sum_{i=1}^k a_id_i = -n-1$.
More generally, we have:
\begin{lemma}\label{lem:c_1}
For $i \in \Z$, we have $ic_1(X_{\vec{d}}^{2n}) = 0 \in H_2(X_{\vec{d}}^{2n})$ if and only if $i(n+1)$ is divisible by $\gcd(\vec{d})$.
\end{lemma}
\NI 
As a consequence, we obtain the following purely formal counterpart to Theorem~\ref{thm:main_combinatorial}.
In the following, a codimension zero smooth embedding is an {\em almost symplectic embedding} if it preserves the homotopy class of the symplectic form as a nondegenerate two-form (or, equivalently, it preserves the homotopy class of a compatible almost complex structure).
Put
\begin{align}
\formal_n(\vec{d}) := \frac{\gcd(\vec{d})}{\gcd(\gcd(\vec{d}),n+1)}.
\end{align}
\begin{cor}\label{cor:formal_obstr}
Suppose there is an almost symplectic embedding of $X_{\vec{d}}^{2n}$ into $X_{\vec{d'}}^{2n}$.
Then we must have that $\formal_n(\vec{d})$ divides $\formal_n(\vec{d'})$. 
\end{cor}

\begin{proof}
Let $\iota$ be an almost symplectic embedding $X_{\vec{d}}^{2n} \hookrightarrow X_{\vec{d'}}^{2n}$. Observe that $F_n(\vec{d})$ is the smallest positive $i$ such that $\gcd(\vec{d})$ divides $i(n+1)$, or equivalently such that $ic_1(X_{\vec{d}}^{2n}) = 0$. 
Note that $\iota$ preserves the homotopy class of compatible almost complex structures, and hence first Chern classes.
Then we have $$F_n(\vec{d'})c_1(X_{\vec{d}}^{2n}) =\iota^*( F_n(\vec{d'})c_1(X_{\vec{d'}}^{2n})) = 0,$$
and hence $F_n(\vec{d'})$ is a multiple of $F_n(\vec{d})$. 
\end{proof}
\begin{proof}[Proof of Lemma~\ref{lem:c_1}]
Put $X := X_{\vec{d}}^{2n}$, and consider the long exact sequence
\begin{align*}
\dots \rightarrow H^2(\CP^n,X) \rightarrow H^2(\CP^n) \rightarrow H^2(X) \rightarrow \dots,
\end{align*}
which using excision and Poincare--Lefschetz duality we can rewrite as
\begin{align*}
\dots \rightarrow H_{2n-2}(D) \rightarrow H^2(\CP^n) \rightarrow H^2(X) \rightarrow \dots .
\end{align*}
Let $[H] \in H_{2n-2}(\CP^n)$ denote the hyperplane class.
Using the identifications $H_{2n-2}(D) = \Z\langle [D_1],\dots,[D_k]\rangle$ and $H^2(\CP^n) = \Z[H]^{\vee}$,
observe that the image of an element $x \in H_{2n-2}(D)$ is $(x \cdot [H]) [H]^{\vee}$.
In particular, the element $[D_i] \in H_{2n-2}(D)$ gets mapped to $d_i[H]^\vee$ for $i  = 1,\dots,k$,
and hence the image of this map is $\Z\langle \gcd(\vec{d})[H]^\vee\rangle$.
Since $c_1(X) \in H^2(X)$ is the image of $(n+1)[H]^{\vee} \in H^2(\CP^n)$, this vanishes if and only if $\gcd(\vec{d})$ divides $n+1$.
More generally, for $i \in \Z$, $i c_1(X)$ vanishes if and only if $\gcd(\vec{d})$ divides $i(n+1)$.
\end{proof}
\begin{remark}
It is interesting to compare Corollary~\ref{cor:formal_obstr} with Example~\ref{ex:domain_single_comp}.
Namely, consider a Liouville embedding $X^{2n}_{(d_1)} \es X^{2n}_{\vec{d'}}$.
Under the assumption $\gcd(d_1,n+1) = 1$, we have $F_n((d_1)) = d_1$, and hence Corollary~\ref{cor:formal_obstr} implies that $d_1$ must divide $F_n(\vec{d'})$, and hence also $\gcd(\vec{d'})$.
By contrast, in the case that $d_1$ is a divisor of $n+1$, we have $F_n((d_1)) = 1$, so Corollary~\ref{cor:formal_obstr} is vacuous, whereas Theorem~\ref{thm:main_combinatorial} implies that $d_1 | \gcd(\vec{d'})$ still holds.
\end{remark}

\sss

Fix $\vec{d} = (d_1,\dots,d_k) \in \Z_{\geq 1}^k$ for some $k \in \Z_{\geq 1}$. Recall that the loci of Reeb orbits for $\bdy X_{\vec{d}}^{2n}$ give rise to the spectral sequence in \cite{ganatra2020symplectic, mcleanslides}, which computes the symplectic cohomology of $X_{\vec{d}}^{2n}$ and whose first page is described in terms of the ordinary cohomology of the torus bundles $\mr{S}_{\vec{v}}$ over the open divisor strata $\mr{D}_{\vec{v}}$.
A straightforward consequence using compatibility with the grading by $H_1(X_{\vec{d}}^{2n})$ is that the spectral sequence degenerates at the first page if we have $k \geq n+1$. 
Combining this with \cite[Cor. 1.2]{ganatra2020symplectic}:
\begin{prop}
We have $\sh(X_{\vec{d}}^{2n}) \neq 0$ for any coefficient ring $\K$, provided that either $\sum_{i=1}^k d_i \geq n+1$ or $d_i \geq 2$ for some $i \in \{1,\dots,k\}$.
\end{prop}
\NI Note that $X_{n+1}^{2n}$ is Weinstein deformation equivalent to $D^*\mathbb{T}^n$, and in particular has nonvanishing symplectic cohomology. For $\sum_{i=1}^k d_i \geq n+1$, we have $X_{n+1}^{2n} \ws X_{\vec{d}}^{2n}$ by Theorem~\ref{thm:embeddings}, and hence $\sh(X_{\vec{d}}^{2n}) \neq 0$ by Observation~\ref{obs:vanishingSH}.

\begin{remark}
Note that for $\sum_{i=1}^k d_i < n+1$ with $d_i \geq 2$ for some $i \in \{1,\dots,k\}$, $X_{\vec{d}}^{2n}$ is not subcritical or flexible.
It would be interesting to see whether the assumption $\sum_{i=1}^{k'} d_i' \geq n+1$ in Corollary~\ref{cor:G_obstructions} could be weakened.  
\end{remark}

\subsection{Lower bound}\label{subsec:lower_bound}

In this subsection we prove the following lemma, which is based on index and first homology considerations:
\begin{lemma}\label{lem:lb_sft}
For $\vec{d} = (d_1,\dots,d_k) \in \Z_{\geq 1}^k$ with $\sum_{i=1}^k d_i \geq n+1$, we have $\G\lll \T^{n-1}p\rrr(X^{2n}_{\vec{d}}) \geq \sum_{i=1}^k d_i$.
\end{lemma}

\begin{proof}
Let $\aug^k \lll \TT p \rrr: \odot^k\cclin(X_{\vec{d}}^{2n}) \rightarrow \K[t]$ for $k \in \Z_{\geq 1}$ denote the maps constituting the $\Li$ homomorphism $\aug \lll \TT p \rrr: \cclin(X_{\vec{d}}^{2n}) \rightarrow \K[t]$, i.e. the coefficient of $t^m$ in $\aug^k \lll \TT p \rrr$ counts curves with $k$ positive ends and a $\lll \T^m p\rrr$ local tangency constraint. 
Let $\aug\lll \T^m p\rrr: \odot^k\cclin(X) \rightarrow \K$ denote the $\Li$ augmentation whose constituent maps $\aug^k\lll \T^m p\rrr$ for $k \in \Z_{\geq 1}$ are given by post composing $\aug^k\lll \TT p \rrr$ with the projection to the $t^m$ component of $\K[t]$.

According to the definition of $\G\lll \T^{n-1}p\rrr(X^{2n}_{\vec{d}})$, it suffices to show that we have $\aug^l\lll \T^{n-1} p\rrr(\gamma_1,\dots,\gamma_l) = 0$ for any $l < \sum_{i=1}^k d_k$ and Reeb orbits $\gamma_1,\dots,\gamma_l$ in $\bdy X_{\vec{d}}^{2n}$.
Since $\aug^l\lll \T^{n-1}p\rrr$ counts index zero curves, this follows immediately from the next lemma. 
\end{proof}

\begin{lemma}\label{lem:too_few_ends}
Assume $\sum_{i=1}^k d_i \geq n+1$. Let $u$ be a formal curve in $X_{\vec{d}}^{2n}$ with $l < \sum_{i=1}^kd_i$ positive ends and satisfying the constraint $\lll \T^{n-1}p\rrr$.
Then we have $\ind(u) < 0$.
\end{lemma}
\begin{proof}
For $u$ a formal curve in $X_{\vec{d}}^{2n}$ as in the statement of the lemma with $l$ positive ends, which we take to be of the form $\gamma_{\vec{v}_1}^{A_1},\dots, \gamma_{\vec{v}_l}^{A_l}$ for some vectors $\vec{v}_1,\dots,\vec{v}_l \in \Z_{\geq 0}^k$ and critical points $A_i \in \crit(f_{\vec{v}_i})$ for $i = 1,\dots,l$.
Note that we must have $\sum_{i=1}^l [\gamma_{\vec{v}_i}^{A_i}] = 0 \in H_1(X_{\vec{d}}^{2n})$, i.e. 
$\sum_{i=1}^l \vec{v}_i = q\vec{d}$
for some $q \in \Z_{\geq 1}$.

Let $\triv$ denote the framing of Reeb orbits in $\bdy X_{\vec{d}}^{2n}$ which extends over small spanning disks as in \S\ref{subsec:div_compl}, and let $c_1^{\triv}(u)$ denote the relative first Chern number of $u$ with respect to $\triv$.
By capping off each asymptotic Reeb orbit of $u$ with its small spanning disk, we obtain a formal two-sphere $S$ in $\CP^n$ of degree $q$. Since the trivialization $\triv$ extends over each of the small bounding disks, we have
$$c_1^{\triv}(u) = c_1(S) = q(n+1).$$

By the discussion in \S\ref{subsec:div_compl}, for each $1 \leq i \leq l$ we have $\cz_{\triv}(\gamma_{\vec{v_i}}^{A_i}) = \delta_i - 2\vec{v}\cdot \vec{1}$ for some $\delta_i \leq n-1$.
Noting that the constraint $\lll \T^{n-1} p\rrr$ is codimension $4n-4$, we then have
\begin{align*}
\ind(u) &= (n-3)(2-l) + \sum_{i=1}^l \cz(\gamma_{\vec{v}_i}^{A_i}) + 2c_1^{\triv}(u) - (4n-4)\\
&= (n-3)(2-l) + \sum_{i=1}^l\left( \delta_i - 2\vec{v}_i \cdot\vec{1}\right) +2q(n+1)- 4n + 4\\
&\leq (n-3)(2-l) + (n-1)l - 2\left(\sum_{i=1}^l \vec{v}_i \right)\cdot\vec{1} + 2q(n+1)- 4n + 4\\
&= (n-3)(2-l) + (n-1)l - 2q\left(\vec{d} \cdot \vec{1} - n - 1\right) - 4n + 4\\
&\leq (n-3)(2-l) + (n-1)l - 2\left(\sum_{i=1}^kd_i - n - 1\right) - 4n+4\\
&= 2l - 2\sum_{i=1}^k d_i\\
&< 0.
\end{align*}
\end{proof}

\subsection{Upper bound}\label{subsec:upper_bound} We prove the following lemma, which together with Lemma~\ref{lem:lb_sft} proves Theorem~\ref{thm:main_G_computation}:

\begin{lemma}\label{lem:ub}
For any $\vec{d} = (d_1,\dots,d_k) \in \Z_{\geq 1}^k$, we have $\G\lll \T^{n-1}p\rrr(X^{2n}_{\vec{d}}) \leq \sum_{i=1}^k d_i$.
\end{lemma}
\NI The basic idea for getting an upper bound is as follows. Starting with the moduli space $\calM_{\CP^n,[L]}\lll \T^{n-1} p\rrr$ of degree one curves in $\CP^n$ satisfying a $\lll \T^{n-1}p\rrr$ local tangency constraint, we stretch the neck along the boundary of $X_{\vec{d}}^{2n}$, keeping the constraint in the interior (a similar approach appears in a slightly different context in \cite{CM2,tonk}).
Building on the discussion in \S\ref{subsec:hypersurf_compl}, we have strong control over the geometry of the complement $\CP^n \setminus X_{\vec{d}}^{2n}$, which is a small neighborhood of the divisor $D$.
We show that the outcome must be a nonzero count of curves in $X_{\vec{d}}^{2n}$ with precisely $\sum_{i=1}^k d_i$ positive ends, and moreover the corresponding collection of Reeb orbits must give rise to a cycle in $\bar\cclin(X_{\vec{d}}^{2n})$.

\sss

Let $N_{\vec{d}}^{2n}$ denote the closure of $\CP^n \setminus X_{\vec{d}}^{2n}$.
We begin with a lemma characterizing certain relative homology classes in $H_2(N_{\vec{d}}^{2n},\bdy N_{\vec{d}}^{2n})$.
\begin{lemma}\label{lem:rel_hom_es}
Consider the exact sequence
\begin{align*}
H_2(N_{\vec{d}}^{2n}) \xrightarrow[]{\;\;f\;\;} H_2(N_{\vec{d}}^{2n},\bdy N_{\vec{d}}^{2n}) \xrightarrow[]{\;\;\delta\;\;}  H_1(\bdy N_{\vec{d}}^{2n}),
\end{align*}
and suppose that $A \in H_2(N_{\vec{d}}^{2n},\bdy N_{\vec{d}}^{2n})$ lies in the kernel of $\delta$.
Then there is some $q \in \Z$ such that we have $A \cdot [D_i] = qd_i$ for each $i = 1,\dots,k$.
\end{lemma}
\begin{proof}
By exactness, we have $A = f(B)$ for some element $B \in H_2(N_{\vec{d}}^{2n})$.
Let $\iota: H_2(N_{\vec{d}}^{2n}) \rightarrow H_2(\CP^n)$ denote the map induced by the inclusion $N_{\vec{d}}^{2n} \subset \CP^n$.
Observe that we have $f(B) \cdot [D_i] = \iota(B) \cdot [D_i]$ for each $1 \leq i \leq k$. 
Since $H_2(\CP^n)$ is generated by the homology class $[L]$ of a line, we have $\iota(B) = q[L]$ for some $q \in \Z$, and hence
$A \cdot [D_i] = q[L]\cdot [D_i] = qd_i$ for each $1 \leq i \leq k$.
\end{proof}

\sss

By our transversality assumptions, 
each limiting configuration under the aforementioned neck stretching procedure must be a two-level building, with
\begin{itemize}
\item
top level in $N_{\vec{d}}^{2n}$ consisting of a collection of index zero curves, each with at least one negative end
\item
bottom level in $X_{\vec{d}}^{2n}$ consisting of a collection of index zero curves, each with at least one positive end,
\end{itemize}
such that the total configuration represents a sphere in class $[L] \in H_2(\CP^n)$.
Note that the bottom level has one ``main'' component $u$ which inherits the $\lll \T^{n-1}p\rrr$ constraint.
By grouping together components sharing a paired asymptotic end, excepting the positive ends of the main component, we can view this as a building with:
\begin{itemize}
\item
top level in $N_{\vec{d}}^{2n}$ consisting of a collection of index zero planes $C_1,\dots,C_l$, anchored in $X^{2n}_{\vec{d}}$
\item bottom level in $\CP^n$ consisting of just the main component $u$.
\end{itemize}
Note that since the anchored planes are composed of pseudoholomorphic curves, they have nonnegative energy, and by positivity of intersection\footnote{Strictly speaking positivity of {\em local} intersection points would require a more delicate discussion of virtual techniques. Here we only use need the fact that the total homological intersection number is positive, which is manifestly perturbation invariant.} the homological intersection number $[C_i] \cdot [D_j]$ is nonnegative for each $i \in \{1,\dots,l\}$ and $j \in \{1,\dots,k\}$.
Put $\vec{s}_i = ([C_i] \cdot [D_1] ,\dots, [C_i]\cdot [D_k]) \in \Z_{\geq 0}^k$.
Then, by Lemma~\ref{lem:formal_caps} below, each of the anchored planes $C_i$ is negatively asymptotic to $\beta_j$ for some $j \in \{1,\dots,k\}$ (recall that we defined the Reeb orbit $\beta_j$ in \S\ref{subsec:hypersurf_compl}), and we have $\vec{s}_i = e_j$.
Since the total configuration represents the line class in $H_2(\CP^n)$, we must have $\sum_{i=1}^l \vec{s}_i = \vec{d}$.
It follows that we have $l = \sum_{i=1}^k d_i$, and up to reordering the positive ends of $u$ are 
$\underbrace{\beta_1,\dots,\beta_1}_{d_1},\dots,\underbrace{\beta_k,\dots,\beta_k}_{d_k}$.
For brevity, in the sequel we denote this list by $\beta_1^{\times d_1},\dots,\beta_k^{\times d_k}$.

\begin{lemma}\label{lem:formal_caps}
Let $C$ be a formal plane in $N_{\vec{d}}^{2n}$, anchored in $X^{2n}_{\vec{d}}$, with negative asymptotic $\gamma_{\vec{v}}^A$. 
Assume that the homological intersection number of $C$ with each component of $D$ is nonnegative, and put
$\vec{s} := ([C]\cdot [D_1],\dots,[C]\cdot [D_k]) \in \Z_{\geq 0}^k$.
Assume also that $C$ has nonnegative symplectic area.
Then we have $\ind(C) \geq 0$. Moreover, if $\ind(C) = 0$, for some $j \in \{1,\dots,k\}$ we must have $\vec{v} = \vec{s} = \vec{e}_j$ and $|A|=0$, and hence $\gamma_{\vec{v}}^A = \beta_j$.
\end{lemma}
\begin{proof}
We will assume $n \geq 2$ (the case $n=1$ can be checked directly and is left to the reader).
Note that due to the anchors $C$ make involve components in $X^{2n}_{\vec{d}}$, and hence does not a priori define a homology class in $H_2(N_{\vec{d}}^{2n},\bdy N_{\vec{d}}^{2n})$. 
However, since $H_1(X^{2n}_{\vec{d}},\bdy X^{2n}_{\vec{d}}) = 0$, for each component of $C$ lying in $X^{2n}_{\vec{d}}$ there is a formal curve in $\bdy X^{2n}_{\vec{d}}$ with the same topological type, positive asymptotics, and energy, and after making these replacements we get a homology class in $H_2(N_{\vec{d}}^{2n},\bdy N_{\vec{d}}^{2n})$ which we will denote by $[C']$.

For $i = 1,\dots,k$, let $c_i$ be a small disk intersecting $D_i$ once tranversely and negatively and disjoint from the other components of $D$ as in \S\ref{subsec:div_compl}, and let $[c_i] \in H_2(N_{\vec{d}}^{2n},\bdy N_{\vec{d}}^{2n})$ denote its relative homology class.
Since $\sum_{i=1}^k v_i[c_i] - [C']$ lies in the kernel of the connecting map $H_2(N_{\vec{d}}^{2n},\bdy N_{\vec{d}}^{2n}) \rightarrow H_1(\bdy N_{\vec{d}}^{2n})$, by Lemma~\ref{lem:rel_hom_es} we have $\vec{s} = \vec{v} + q\vec{d}$ for some $q \in \Z$. 

Let $\rho: N_{\vec{d}}^{2n} \rightarrow [0,1]$ be a function which is $0$ near $\bdy N_{\vec{d}}^{2n}$ and $1$ outside of a small neighborhood $U$ of $\bdy N_{\vec{d}}^{2n}$, 
and let $\wt{\omega}$ be the two-form on $N_{\vec{d}}^{2n}$ given by $\rho \la$ on $U$ and $\omega$ on $N_{\vec{d}}^{2n} \setminus U$. 
Let $\vec{\ww} = (\ww_1,\dots,\ww_k)$ denote the corresponding symplectic wrapping numbers as in \S\ref{subsec:div_compl}, i.e. $-\sum_{i=1}^k \ww_i[D_i] \in H_{2n-2}(N_{\vec{d}}^{2n};\R)$ is Poincar\'e--Lefschetz dual to $[\wt{\omega}] \in H^2(N_{\vec{d}}^{2n},\bdy N_{\vec{d}}^{2n};\R)$.
By Theorem~\ref{thm:div_compl1}, these agree with the holomorphic wrapping numbers, so we have $\ww_i = -d_i < 0$ for $i = 1,\dots,k$.
We then have:
\begin{align*}
	0 \leq \int_{[C']} \omega 
	&= \int_{[C']} \wt{\omega} + \int_{[C']}(\omega - \wt{\omega})\\
	&= \left(-\sum_{i=1}^k \ww_i[D_i]\right) \cdot [C] - \calA(\gamma_{\vec{v}}^A)\\
 &\leq -\vec{\ww} \cdot \vec{s} + \vec{\ww} \cdot \vec{v} + \eps\\
 &= -q\vec{\ww}\cdot \vec{d} + \eps,
\end{align*}
where in the second line we have used the definition of symplectic wrapping numbers and in the third line we have used part (6) of Theorem~\ref{thm:div_compl1}.
Since each component of $\vec{d}$ is positive, each component of $\ww$ is nonpositive, and $\eps$ is arbitrarily small, we must have $q \geq 0$.

Observe that we have $c_1^{\tau}(C) = q(n+1)$, 
since we can glue $C$ to the small spanning disk for $\gamma_{\vec{v}}^A$ with its opposite orientation to get a formal sphere of degree $q$.
We then have
\begin{align*}
\ind(C) &= n-3 + 2c_1^{\triv}(C) - \cz_{\triv}(\gamma_{\vec{v}}^A)\\
&= n-3 + 2q(n+1) - (\delta(\gamma_{\vec{v}}^A) - 2\vec{v} \cdot \vec{1})\\
&\geq -2 + 2q(n+1) + 2\vec{v} \cdot \vec{1},
\end{align*}
with equality only if $\delta(\gamma_{\vec{v}}^A) = n-1$. 
We recall here that $\vec{v} \cdot \vec{1}$ is a shorthand for $\sum_{i=1}^k v_i$, and $\delta$ is defined in the discussion following Proposition~\ref{prop:compl_cz}.
Moreover, $-2 + 2\vec{v} \cdot \vec{1}$ and $2q(n+1)$ are both nonnegative, with equality only if 
$\vec{v}\cdot \vec{1} = 1$ and $q = 0$, in which case we must have $\vec{v} = \vec{e}_j$ for some $j \in \{1,\dots,k\}$.
Since $\beta_j$ is the unique Reeb orbit of the form $\gamma^A_{\vec{e}_j}$ satisfying $\delta(\gamma^A_{\vec{e}_j}) = n-1$, the lemma follows.

 \end{proof}

Since neck stretching produces a cobordism of moduli spaces, the above discussion shows that the count of curves in $X_{\vec{v}}^{2n}$ with positive ends $\beta_1^{\times d_1},\dots,\beta_k^{\times d_k}$ and satisfying the constraint $\lll \T^{n-1}p\rrr$ is nonzero.
Therefore, to complete the proof of Lemma~\ref{lem:ub}, it suffices to show that $(\odot^{d_1}\beta_1) \odot \dots \odot (\odot^{d_k}\beta_k)$ is closed under the bar differential:

\begin{lemma}\label{lem:bar_cycle}
The element $(\odot^{d_1}\beta_1) \odot \dots \odot (\odot^{d_k}\beta_k) \in \bar\cclin(X^{2n}_{\vec{d}})$ is a cycle with respect to the bar differential.
\end{lemma}
\begin{proof}
Recall that the $\Li$ operations on $\cclin(X_{\vec{d}}^{2n})$ count index one curves in the symplectization $\R \times \bdy X_{\vec{d}}^{2n}$, anchored in $X_{\vec{d}}^{2n}$, with some number $l \geq 1$ of positive ends and one negative end. Let $C$ be such an anchored curve, with top ends $\beta_{i_1},\dots,\beta_{i_l}$ for some $i_1,\dots,i_l \in \{1,\dots,k\}$ such that $\sum_{j=1}^l \vec{e}_{i_j} \leq \vec{d}$,
and let $\gamma_{\vec{v}}^A$ be the bottom end.
First homology considerations give
$$[\gamma_{\vec{v}}^A] = [\beta_{i_1}] + \dots + [\beta_{i_l}] \in H_1(X_{\vec{d}}^{2n}),$$
and hence $\vec{v} = \vec{e}_{i_1} + \dots + \vec{e}_{i_l} + q\vec{d}$ for some $q \in \Z$.

By nonnegative of energy and Stokes' theorem, we  have
\begin{align*}
0 \leq E(C) &= \calA(\beta_{i_1}) + \dots + \calA(\beta_{i_l}) - \calA(\gamma_{\vec{v}}^A) \\&\leq -\ww \cdot (\vec{e}_{i_1} + \dots + \vec{e}_{i_l} - \vec{v}) + (l+1)\eps \\&= -\ww \cdot (-q\vec{d}) + (l+1)\eps.
\end{align*}
Since each component of $\ww$ is negative, each component of $\vec{d}$ is positive, and $\eps > 0$ is arbitrarily small,
we must have $q \leq 0$.
This implies that $q = 0$, since each component of $\vec{v}$ is nonnegative.
We then have
\begin{align*}
\ind(C) &= (n-3)(1-l) + \sum_{j=1}^l \cz_{\triv}(\beta_{i_j}) - \cz_{\triv}(\gamma_{\vec{v}}^A) + 2c_1^{\triv}(C)\\
&= (n-3)(1-l) + l(n-3) - (\delta(\gamma_{\vec{v}}^A) - 2 \vec{v} \cdot \vec{1})\\
&\geq (n-3)(1-l) + l(n-3) - (n-1) + 2l\\
&= 2l-2.
\end{align*}
Note that $2c_1^{\triv}(C) = 0$ since $q = 0$ (c.f. Lemma~\ref{lem:cz_well_def}).
This shows that $\ind(C) \geq 2$ unless $l = 1$.

The case $l=1$ corresponds to a cylinder $C$ in $\R \times \bdy X_{\vec{d}}^{2n}$, anchored in $X_{\vec{d}}^{2n}$, with top end $\beta_{i_1}$ and bottom end $\gamma_{\vec{v}}^A$, and $\ind(C) = 1$ means that we have
 $\cz(\gamma_{\vec{v}}^A) = \cz(\beta_{i_1}) - 1$.
Note that by action and first homology considerations as above we must have $\vec{v} = \vec{e}_{i_1}$.
Furthermore, we claim that $C$ cannot be anchored. 
Indeed, we have $E(C) \leq \eps$ (c.f. the proof of Lemma~\ref{lem:formal_caps}), whereas any curve in $X_{\vec{d}}^{2n}$ would have energy at least that of the minimal action of a Reeb orbit in $\bdy X_{\vec{d}}^{2n}$, which is in turn bounded from below by
$\min\limits_{1 \leq i \leq k}-\ww_i - \eps \geq 1-\eps.$

Although we cannot a priori rule out the existence of $C$ as an honest index one cylinder in $\R \times \bdy X_{\vec{d}}^{2n}$, it suffices to show that the count of such cylinders (modulo target translations) is vanishing for any fixed choice of negative asymptotic Reeb orbit $\gamma_{\vec{e}_{i_1}}^A$.
To see this, consider the compactified moduli space of index one pseudoholomorphic planes in $N_{\vec{d}}^{2n}$ with negative asymptotic $\gamma_{\vec{e}_{i_1}}^A$ and homological intersection number one with $D_i$. 
Its boundary consists of
two-level configurations, with:
\begin{itemize}
\item top level consisting of an index zero plane in $N_{\vec{d}}^{2n}$ with negative asymptotic $\beta_{i_1}$
\item bottom level consisting of an index one cylinder in the symplectization $\R \times \bdy N_{\vec{d}}^{2n}$, positively asymptotic to $\beta_{i_1}$ and negatively asymptotic to $\gamma_{\vec{e}_{i_1}}^A$.
\end{itemize}
We claim that the count of planes in the top level is nonzero.
Since the total count of boundary configurations is zero, it then follows that the count of cylinders in the bottom level is necessarily zero, as desired.

Finally, the preceding claim follows from the neck stretching procedure described at the beginning of this subsection.
Indeed, by energy considerations as above, 
each of the anchored planes $C_1,\dots,C_l$ is in fact unanchored.
Consequently, since neck stretching induces a cobordism of moduli spaces, the count of index zero planes in $N^{2n}_{\vec{d}}$ with negative asymptotic $\beta_i$ is necessarily nonzero for each  $1 \leq i \leq k$.

\end{proof}

\subsection{The Cieliebak--Latschev formalism}\label{subsec:CL}
In this subsection, which is logically independent from the rest of the paper, we provide a broader perspective on the upper bound in the previous subsection based on Maurer--Cartan theory. This approach, which builds on unpublished work of Cieliebak--Latschev and is discussed also in \cite[\S 4]{HSC} from a slightly different perspective, can be used to produce bar complex cycles in greater generality.
In particular, we prove:
\begin{thm}\label{thm:cl_ub}
Let $M^{2n}$ be a closed symplectic manifold of dimension $2n \geq 4$ and let $A \in H_2(M)$ be a homology class such that the count $\gw_{M,A}\lll \T^m p \rrr \in \Q$
is nonzero\footnote{We note that the invariant $\gw_{M,A}\lll \T^m p\rrr$ is well-defined via classical perturbation techniques if $M$ is semipositive by \cite[Prop. 2.2.2]{McDuffSiegel_counting}, whereas the definition for general $M$ necessitates virtual perturbations.} for some $m \in \Z_{\geq 0}$. 
Let $X^{2n}$ be a $2n$-dimensional Liouville domain admitting a symplectic embedding into $M$ such that the
induced map $H^{2n-2}(M) \rightarrow H^{2n-2}(\ovl{M \setminus X})$ is injective.
Then we have
$$ \G\lll \T^m p \rrr(X) < \infty.$$
\end{thm}

Let $X$ be a Liouville domain which is symplectically embedded into a closed symplectic manifold $M$.
As before, for simplicity we work over $\K = \Q$. 
Let $\cclin(X;\wt{\K})$ denote the $\Li$ algebra as described in \S\ref{subsec:obs_higher}, but with the following modifications:
\begin{itemize}
\item
as a $\K$-module, the generators of $\cclin(X;\wt{\K})$ are pairs $(\gamma,[\Sigma])$, where $\gamma$ is a good Reeb orbit in $\bdy X$ and $[\Sigma]$ is a $2$-chain $\Sigma$ in $M$ with boundary $\gamma$, modulo boundaries of $3$-chains in $M$
\item the differential and higher $\Li$ operations count the same
curves as before, and the 
bounding $2$-chain $\Sigma$ of the output is given by concatenating these with the bounding $2$-chains of the inputs.
\end{itemize}
Here we are assuming that $\bdy X$ has nondegenerate Reeb dynamics, which we can always achieve by a small perturbation.
Note that, for a given Reeb orbit $\gamma$, the set of possible choices of $[\Sigma]$ is a torsor over $H_2(M)$.

Each generator $(\gamma,[\Sigma])$ of $\cclin(X;\wt{\K})$ has a well-defined energy, given by the integrating the symplectic form $\omega$ of $M$ over $\Sigma$. This induces a decreasing $\R$-filtration on $\cclin(X;\wt{\K})$, and we denote the corresponding completed $\Li$ algebra by $\whcclin(X;\wt{\K})$.
The bar complex $\bar\cclin(X;\wt{\K})$ also inherits a decreasing filtration by energy (i.e. the energy of an elementary tensor is the sum of the energies of its components), and we denote the corresponding completed chain complex by $\wh{\bar}\cclin(X;\wt{\K})$.

The Cieliebak--Latschev formalism associates to the symplectic embedding $X \shookrightarrow M$ a Maurer--Cartan element
$$\m \in \whcclin(X;\wt{\K}),$$
given by the (possibly infinite) count of index zero planes in the symplectic cap $N:= \ovl{M \setminus X}$, anchored\footnote{More precisely, each such configuration is a two-level pseudoholomorphic building, with top level in $N$ and bottom level in $X$, such that the total configuration after formally gluing along each pair of Reeb orbits is a plane.} in $X$
(c.f. \cite[\S 4]{HSC}). 
Note that this sum is well-defined in $\whcclin(X;\wt{\K})$, since by SFT compactness there are only finitely many configurations with energy below any given value.
Since $\m$ lies in the positive part of the filtration on $\whcclin(X;\wt{\K})$, it has a well-defined exponential
\begin{align}
\exp(\m) = \sum_{l=1}^\infty \frac{1}{l!}\underbrace{\m \odot \dots \odot \m}_l \in \wh{\bar}\cclin(X;\wt{\K}), 
\end{align}
and the Maurer--Cartan equation for $\m$ is equivalent to the fact that $\exp(\m)$ is a cycle.

Given a pair $(\gamma,[\Sigma])$ as above, note that $[\Sigma]$ defines a well-defined element in $H_2(M,X) \cong H_2(N,\bdy N)$,
and this gives rise to a natural $H_2(N,\bdy N)$-grading on the $\Li$ algebra $\cclin(X;\wt{\K})$, and also on its completed bar complex $\wh{\bar}\cclin(X;\wt{\K})$.
Given a homology class $A \in H_2(M)$, let $\wt{A} \in H_2(N,\bdy N)$ denote its restriction to $N$ (i.e. we apply Poincar\'e--Lefschetz duality to the input and output of the restriction map $H^{2n-2}(M) \rightarrow H^{2n-2}(N)$), and let $\exp(\m)_{\wt{A}} \in \wh{\bar}\cclin(X;\wt{\K})$ denote the part of $\exp(\m)$ lying in the graded piece corresponding to $\wt{A}$. 
Then $\exp(\m)_{\wt{A}} \in \wh{\bar}\cclin(X;\wt{\K})$ is itself a cycle.

We claim that $\exp(\m)_{\wt{A}}$ in fact lifts to a cycle $x$ in the uncompleted bar complex $\bar\cclin(X;\wt{\K})$. Indeed, it suffices to show that there is a uniform upper bound on the energy of each summand of $\exp(\m)_{\wt{A}}$, since then the SFT compactness theorem implies that there are only finitely many such terms.
To justify the claim, consider a summand of $\exp(\m)_{\wt{A}}$, which we represent as a formal curve in $N$, anchored in $X$. We denote by $C$ the resulting formal curve in $N$ after throwing away any anchors in $X$. Then $C$ represents the homology class $\wt{A} \in H_2(N,\bdy N)$, and we denote its negative asymptotic Reeb orbits by $\gamma_1,\dots,\gamma_l$.
Let $\la$ be a primitive one-form for $\omega$ defined near $\bdy N$, and let $\rho: N \rightarrow [0,1]$ be a function which is $0$ near $\bdy N$ and $1$ outside of a small neighborhood $U$ of $\bdy N$.
Let $\wt{\omega}$ be the two-form on $N$ given by $d(\rho \la)$ on $U$ and $\omega$ on $N \setminus U$.
We have
\begin{align*}
\int_C \omega &= \int_C \wt{\omega} + \int_C (\omega - \wt{\omega})\\
&= \int_C \wt{\omega} + \int_{C \cap U} d([1-\rho]\la)\\
&= \int_C \wt{\omega} - \sum_{i=1}^l \calA(\gamma_i).
\end{align*}
Note that $\int_C\wt{\omega}$ depends only on the homology classes $[\wt{\omega}] \in H^2(N, \bdy N;\R)$ and $[C] = \wt{A} \in H_2(N,\bdy N;\R)$,
and hence we have $\int_C \omega \leq [\wt{\omega]} \cdot [\wt{A}]$ as desired.

Now let $\K[H_2(M)]$ denote the group ring of $H_2(M)$, and let $\wh{\K[H_2(M)]}$ denote its completion with respect to symplectic area.
Put $\gw_{M}\lll \T^m p \rrr := \sum\limits_{A \in H_2(M)} e^A \gw_{M,A}\lll \T^m p \rrr \in \wh{\K[H_2(M)]}$.
In general, neck stretching curves with a $\lll \T^m p \rrr$ constraint gives the relation in $\wh{\K[H_2(M)]}$ of the form
\begin{align}
\pi_1\circ \wh{\aug}\lll \T^m p \rrr(\exp(\m)) = \gw_{M}\lll \T^m p \rrr.
\end{align}
Here $\pi_1 \circ \wh{\aug}\lll \T^m p\rrr: \wh{\bar}\cclin(X) \rightarrow \K[\wh{H_2(M)}]$ is the induced map counting curves with a $\lll \T^m p\rrr$ local tangency constraint as in \S\ref{subsec:obs_simp}, except that we now concatenate these curves with the input curves to define homology classes in $H_2(M)$, and we pass to completions.
Since the restriction map $H_2(M) \rightarrow H_2(N,\bdy N)$ is injective by assumption, $A \in H_2(M)$ is the unique class which restricts to $\wt{A} \in H_2(N,\bdy N)$. Therefore by projecting the above relation to the graded piece corresponding to $\wt{A}$, we get
\begin{align}
\pi_1 \circ \wh{\aug}\lll \T^m p \rrr(\exp(\m)_{\wt{A}}) = \gw_{M,A}\lll \T^m p \rrr \,e^A.
\end{align}
Note that $\exp(\m)_{\wt{A}}$ is in fact a finite sum by the SFT compactness theorem, so it follows from the definition that we have
\begin{align*}
\G\lll \T^m p\rrr(X) < \infty,
\end{align*}
which completes the proof of Theorem~\ref{thm:cl_ub}.

\begin{example}
In the case of the natural inclusion $X_{\vec{d}}^{2n} \subset \CP^n$ for a tuple $\vec{d} = (d_1,\dots,d_k)$, we have
$H_{2n-2}(N_{\vec{d}}^{2n})\cong \K\langle [D_1],\dots,[D_k]\rangle$, where $D_1,\dots,D_k$ represent hypersurfaces of degrees $d_1,\dots,d_k$, and the induced map $H^{2n-2}(\CP^n) \rightarrow H^{2n-2}(N_{\vec{d}}^{2n})$ is injective, so Theorem~\ref{thm:cl_ub} applies.
In this case, by the argument in \S\ref{subsec:upper_bound}, for $A = [L]$ the line class in $H_2(\CP^n)$ we have that
$\exp(\m)_{\wt{A}}$ is a multiple of $(\odot^{d_1}\beta_1) \odot \dots \odot (\odot^{d_k}\beta_k)$
\end{example}

\section{Computations for hypersurface complements II: avoiding virtual perturbations}\label{sec:computationsII}

The main goal in this section is to prove Theorem~\ref{thm:main_combinatorial}. 
In \S\ref{subsec:some_lemmas}, we revisit the neck stretching argument from the previous subsection and analyze the possible degenerations in more detail without virtual perturbations. Subsequently, in \S\ref{subsec:completing} we assemble these ingredients and complete the proof.

\subsection{Some lemmas}\label{subsec:some_lemmas}

In this subsection we formulate various technical results about our moduli spaces of interest, providing the main ingredients for the proof in the next subsection.
Fix $\vec{d} = (d_1,\dots,d_k) \in \Z_{\geq 1}^k$ for some $k \in \Z_{\geq 1}$. 
Note that for this subsection we do not need to assume $\sum_{i=1}^k d_i \geq n+1$.

\begin{lemma}\label{lem:main_cpnt_nonneg_ind}
Let $J$ be a generic admissible almost complex structure on the symplectic completion of $X_{\vec{d}}^{2n}$, 
and let $u$ be a $J$-holomorphic curve in $\wh{X}_{\vec{d}}^{2n}$ satisfying the constraint $\lll \T^{n-1}p\rrr$.
Then we have $\ind(u) \geq 0$.
\end{lemma}
\begin{proof}

Let us take the positive Reeb orbit asymptotics of $u$ to be $\gamma_{\vec{v}_1}^{A_1},\dots,\gamma_{\vec{v}_l}^{A_l}$, and put 
\begin{align*}
\delta_i := \delta(\gamma_{\vec{v}_i}^{A_i}) = n-1-|A_i|
\end{align*}
for $i = 1,\dots,l$.
As before, the constraint $\lll \T^{n-1}p\rrr$ is codimension $4n-4$, and the index of $u$ is given by
\begin{align*}
\ind(u) &=  (n-3)(2-l) - (4n - 4) +\sum_{i=1}^l\cz(\gamma_{\vec{v}_i}^{A_i}) + 2c_1^{\triv}(u) \\
&= (n-3)(2-l) - (4n - 4) + \sum_{i=1}^l \delta_i - 2\left( \sum_{i=1}^l \vec{v}_i \right) \cdot \vec{1} + 2c_1^{\triv}(u).
\end{align*}
Since $\sum_{i=1}^l [\gamma_{\vec{v}_i}^{A_i}] = 0 \in H_1(X_{\vec{d}}^{2n})$, we must have $\sum_{i=1}^l \vec{v}_i = q\vec{d}$ for some $q \in \Z_{\geq 1}$. Then the sphere in $\CP^n$ obtained by capping off each end of $u$ by the corresponding small spanning disk as in \S\ref{subsec:div_compl} has degree $q$, and since $\triv$ extends to these disks we have $c_1^{\triv}(u) = q(n+1)$.

Now suppose that $u$ is a $\kappa$-fold branched cover of its underlying simple curve $\ovl{u}$,
i.e. $u$ is given by the precomposition of $\ovl{u}$ with an order $\kappa$ branched cover $\Phi$, which extends over the punctures to a holomorphic map $\CP^1 \rightarrow \CP^1$.
Let us denote the positive Reeb orbit asymptotics of $\ovl{u}$ by $\gamma_{\vec{w}_1}^{B_1},\dots,\gamma_{\vec{w}_{\ovl{l}}}^{B_{\ovl{l}}}$, and put $\ovl{\delta}_i := \delta(\gamma_{\vec{w}_i}^{B_i})$ for $i = 1,\dots,\ovl{l}$.
The point in the domain of $u$ satisfying the $\lll \T^{n-1}p\rrr$ constraint is mapped by $\Phi$ to a point in the domain of $\ovl{u}$ satisfying a constraint $\lll \T^{m-1}p\rrr$ for some $m \in \Z_{\geq 1}$. 
Taking into account this constraint, the index of $\ovl{u}$ is given by
\begin{align}
\ind(\ovl{u}) &= (n-3)(2-\ovl{l}) - (2n + 2m - 4) + \sum_{i=1}^{\ovl{l}} \ovl{\delta}_i - 2\left(\sum_{i=1}^{\ovl{l}}\vec{w}_i\right)\cdot \vec{1} + 2c_1^{\triv}(\ovl{u}).
\end{align}

We define the {\em branching order} of a branched cover of the Riemann sphere at a point in the domain to be the local degree of the map at that point minus one. 
Let $\a$ denote the branching order of $\Phi$ at the point satisfying the $\lll\T^{n-1}p\rrr$ constraint, and let $\b$ be the sum of the branching orders of $\Phi$ over all of its punctures.
Note that we must have 
$$\a \leq \kappa - 1,$$
and by the Riemann--Hurwitz formula we have 
$$\a + \b \leq 2\kappa - 2.$$
Also, since $u$ satisfies the constraint $\lll \T^{n-1}p\rrr$ and since the contact order to the local divisor gets multiplied by the local degree of the cover, we must have
$$ m(\a+1) \geq n.$$
Furthermore, looking at the punctures we have
$$ l = \kappa \ovl{l} -\b.$$
We also have:
\begin{claim}
 $\sum_{i=1}^l \delta_i \geq \kappa\sum_{i=1}^{\ovl{l}} \ovl{\delta}_i - (n-1)\b.  $
\end{claim}
\begin{proof}
Fix $i \in \{1,\dots,\ovl{l}\}$, and let $j_1,\dots,j_\vk$ be the indices of the punctures of $u$ which cover the $i$th puncture of $\ovl{u}$.
Let $\b_i$ denote the contribution to $\b$ coming from these punctures, i.e. we have $\b_i = \kappa - \vk$.
Since $\ovl{\delta}_i \leq n-1$, we have
\begin{align*}
\delta_{j_1} + \cdots + \delta_{j_\vk} = \vk \ovl{\delta}_i = (\kappa - \b_i) \ovl{\delta}_o \geq \kappa \ovl{\delta}_i - (n-1)\b_i.
\end{align*}
The claim follows after summing over $i$.
\end{proof}
Lastly, observe that we have
$$ \sum_{i=1}^l \vec{v}_i = \kappa \sum_{i=1}^{\ovl{l}}\vec{w}_i,$$
and hence $c_1^{\triv}(\ovl{u}) = q(n+1)/\kappa$.

We then have
\begin{align*}
\ind(u) &\geq (n-3)(2-\kappa\ovl{l}+\b) - (4n - 4) + \left(\kappa\sum_{i=1}^{\ovl{l}}\ovl{\delta}_i - (n-1)\b\right) - 2\kappa\left(\sum_{i=1}^{\ovl{l}}\vec{w}_i\right) \cdot \vec{1} + 2q(n+1),
\end{align*}
and therefore
\begin{align*}
\ind(u) - \kappa \ind(\ovl{u}) &\geq (n-3)(2-\kappa \ovl{l} + \b) - 4n + 4 - (n-1)\b -\kappa(n-3)(2-\ovl{l}) + \kappa(2n + 2m - 4)\\
&= -2n - 2 + 2\kappa - 2\b + 2m\kappa\\
&\geq -2n - 2 + 2\kappa - 2(2\kappa - 2 - \a) + 2m\kappa\\
&= -2n + 2 + 2\a + 2\kappa(m-1)\\
&\geq -2n + 2 + 2(n/m-1) +2\kappa(m-1)\\
&= (m-1)(2\kappa - 2n/m)\\
&\geq (m-1)(2\kappa - 2[a+1])\\
& \geq 0.
\end{align*}
Since $\ovl{u}$ is simple, it is regular for $J$ generic, and hence we have $\ind(\ovl{u}) \geq 0$.
\end{proof}

We now revisit the neck stretching procedure for the moduli space $\calM_{\CP^n,[L]}\lll \T^{n-1} p\rrr$ along the contact type hypersurface $\bdy X_{\vec{d}}^{2n} \subset \CP^n$ as in \S\ref{subsec:upper_bound}.
Consider a generic compatible almost complex structure $J$ on $\CP^n$ which is cylindrical near $\bdy X_{\vec{d}}^{2n}$. Let $J_X$ and $J_N$ denote the induced admissible almost complex structures on the symplectic completions of $X_{\vec{d}}^{2n}$ and $N_{\vec{d}}^{2n}$ respectively, given by restricting $J$ and then extending over the cylindrical ends. 
Similarly, let $J_{\R \times \bdy X}$ denote the resulting admissible almost complex structure on the symplectization $\R \times \bdy X_{\vec{d}}^{2n}$, given by restricting $J_X$ to $\bdy X_{\vec{d}}^{2n}$ and extending $\R$-invariantly.
We assume that $J_X,J_N,J_{\R \times \bdy X}$ are generic, so that simple curves are regular.
Now let $\{J_t\}_{t \in [0,1)}$ be the one-parameter family of almost complex structures on $\CP^n$ realizing the neck stretching as described in \S\ref{subsubsec:neck_stretching}.
\begin{lemma}\label{lem:cpn_ns_result}
Under the above neck stretching, each limiting configuration corresponding to $t = 1$ in the compactified moduli space $\ovl{\calM}_{\CP^n,[L]}^{\{J_t\}}\lll \T^{n-1} p\rrr$
is a two-level pseudoholomorphic building with
\begin{itemize}
\item top level consisting of $\sum_{i=1}^k d_i$ index zero regular $J_N$-holomorphic planes in $\CP^{n}$, with negative asymptotic ends $\beta_1^{\times d_1},\dots,\beta_k^{\times d_k}$ and lying in the homology classes $\underbrace{[c_1],\dots,[c_1]}_{d_1},\dots,\underbrace{[c_k],\dots,[c_k]}_{d_k} \in H_2(N_{\vec{d}}^{2n},\bdy N_{\vec{d}}^{2n})$ respectively 
\item bottom level consisting of a single index zero regular $J_X$-holomorphic (genus zero) curve in $X_{\vec{d}}^{2n}$ with positive ends $\beta_1^{\times d_1},\dots,\beta_k^{\times d_k}$ and satisfying the constraint $\lll \T^{n-1}p\rrr$. 
\end{itemize}
\end{lemma}
\begin{proof}
By the SFT compactness theorem, any limiting configuration consists of a multilevel pseudoholomorphic building, with
\begin{itemize}
\item top level $J_N$-holomorphic in $N_{\vec{d}}^{2n}$ 
\item some number (possibly zero) of levels $J_{\R \times \bdy X}$-holomorphic in the symplectization $\R \times \bdy X_{\vec{d}}^{2n}$ 
\item bottom level $J_X$-holomorphic in $X_{\vec{d}}^{2n}$ consisting of a ``main component'' $u$ satisfying the $\lll \T^{n-1}p\rrr$ constraint, along with some number (possibly zero) of additional unconstrained components. 
\end{itemize}
Note that in principle some of the components in a given level could be joined by nodes (each of which increases the expected codimension by two), but this is easily ruled out.
Namely, by formally gluing together all pairs of asymptotic ends, we obtain a possibly nodal formal sphere in $\CP^n$.
Since the total degree is one and each component has positive area and hence positive degree, this precludes any nodes.

Let us now formally glue together all pairs of ends except for those corresponding to the positive asymptotics of the main component $u$ to arrive at the following simplified picture:
\begin{itemize}
\item
top level consisting of some number $l \geq 1$ of formal planes in $N_{\vec{d}}^{2n}$, anchored in $X_{\vec{d}}^{2n}$
\item bottom level consisting of a single main pseudoholomorphic component $u$ in $X_{\vec{d}}^{2n}$ with $l$ positive ends and satisfying the $\lll \T^{n-1}p\rrr$ constraint.
\end{itemize}
By Lemma~\ref{lem:formal_caps}, each of the formal planes has nonnegative index.
Similarly, by Lemma~\ref{lem:main_cpnt_nonneg_ind}, the main component has nonnegative index.
Since the total configuration has index zero, it follows that the main component and each of the formal planes $C$ must have index zero. 
Lemma~\ref{lem:formal_caps} then implies that each of the components in the top level has negative end $\beta_j$ for some $j \in \{1,\dots,k\}$, and has homological intersection $\delta_{ij}$ with $[D_i]$ for each $i \in \{1,\dots,k\}$.

We wish to show that each of these formal planes $C$ corresponds to an honest pseudoholomorphic curve in $N_{\vec{d}}^{2n}$.
We can suppose that the configuration underlying $C$ consists of a multilevel pseudoholomorphic building with
\begin{itemize}
\item top level in $N_{\vec{d}}^{2n}$
\item some number (possibly zero) of levels in the symplectization $\R \times \bdy X_{\vec{d}}^{2n}$
\item bottom level (possibly empty) in $X_{\vec{d}}^{2n}$.
\end{itemize}
Since the symplectic form on $N_{\vec{d}}^{2n}$ is exact away from $D$ (c.f. part (2) of Theorem~\ref{thm:div_compl1}), by Stokes' theorem each component in the top level must intersect $D$ nontrivially. 
Since the homological intersection number of $[C]$ with $[D] = \sum_{i=1}^k [D_i]$ is $1$, it follows by positivity of intersection that there is exactly one component in the top level.
Also, as in the proof of Lemma~\ref{lem:bar_cycle}, energy considerations rule out any component in the bottom level $X_{\vec{d}}^{2n}$. Indeed, note that we have $E(C) \leq \eps$ (c.f. the proof of Lemma~\ref{lem:formal_caps}), whereas any such component would have at least one positive Reeb orbit asymptotic and hence energy at least $1-\eps$.
Similarly, there cannot be any components without any negative ends in a symplectization level.
It follows that the component in the top level must be a plane, and by similar energy considerations there are no symplectization levels (since any component with energy less than $\eps$ would necessarily be a trivial cylinder, violating the stability condition in the SFT compactness theorem), so $C$ is an honest plane in $N_{\vec{d}}^{2n}$. Evidently $C$ is simple since $\beta_j$ is a primitive Reeb orbit. 
Finally, since $J_X$ is generic, regularity of the component $u$ follows from the simple Lemma~\ref{lem:main_curves_are_simple} below.
\end{proof}

\begin{lemma}\label{lem:main_curves_are_simple}
Let $J_X$ be any admissible almost complex structure on the symplectic completion of $X_{\vec{d}}^{2n}$,
and consider $u \in \calM^{J_X}_{X_{\vec{d}}^{2n}}(\beta_1^{\times d_1},\dots,\beta_k^{\times d_k})\lll \T^{n-1} p\rrr$.
Then $u$ is simple, and hence regular if $J_X$ is generic.
\end{lemma}
\begin{proof}
Suppose by contradiction that $u$ is a $\kappa$-fold cover of its underlying simple curve $\ovl{u}$ for some $\kappa \geq 2$. 
Let $\gamma^{B_1}_{\vec{w}_1},\dots,\gamma^{B_{\ovl{l}}}_{\vec{w}_{\ovl{l}}}$ denote positive ends of $\ovl{u}$.
Since we have $\sum_{i=1}^{\ovl{l}}[\gamma^{B_i}_{\vec{w}_i}] = 0 \in H_1(X_{\vec{d}}^{2n})$, we must have $\sum_{i=1}^{\ovl{l}} \vec{w}_i = q\vec{d}$ for some $q \in \Z_{\geq 1}$. However, we then must have $\kappa q\vec{d} = \sum_{i=1}^k d_i\vec{e}_i = \vec{d}$, which is not possible unless $\kappa = 1$.
\end{proof}

The following proposition is roughly the geometric analogue of Lemma~\ref{lem:bar_cycle}.
Let $J_{\R \times \bdy X}$ be a fixed generic cylindrical almost complex structure on the symplectization $\R \times \bdy X_{\vec{d}}^{2n}$,
and let $\calJ_{X}$ denote the space of all admissible almost complex structures on $\wh{X}_{\vec{d}}^{2n}$ which agree with $J_{\R \times \bdy X}$ on a neighborhood of the cylindrical end.
\begin{prop}\label{prop:cnt_fin_and_well_def}
For generic $J_X \in \calJ_X$, 
the signed count 
of (genus zero) $J_X$-holomorphic curves in $X_{\vec{d}}^{2n}$ with positive asymptotics $\beta_1^{\times d_1},\dots,\beta_k^{\times d_k}$ and satisfying the constraint $\lll \T^{n-1}p\rrr$ is finite, nonzero, and independent of $J_X$.
\end{prop}
\begin{remark}
Since $J_{\R \times \bdy X}$ is arbitrary, it follows that the count in Proposition~\ref{prop:cnt_fin_and_well_def} is nonzero for any choice of 
generic admissible almost complex structure $J_X$ on $\wh{X}$, i.e. not necessarily with fixed behavior at infinity.
With a bit more work, it is also possible to show that this count is entirely independent of this choice, but since we will not explicitly need this we omit the proof for brevity.
\end{remark}

\begin{proof}[Proof of Proposition~\ref{prop:cnt_fin_and_well_def}]
As before, let $J$ be a generic compatible almost complex structure on $\CP^n$ which agrees with $J_{\R \times \bdy X}$ on $\Op(\bdy X_{\vec{d}}^{2n})$ and restricts to $J_X$ and $J_N$ on $X_{\vec{d}}^{2n}$ and $N_{\vec{d}}^{2n}$ respectively, and let $\{J_t\}_{t \in [0,1)}$ be a family of almost complex structures on $\CP^n$ with $J_0 = J$ which realizes the neck stretching along $\bdy X_{\vec{d}}^{2n}$.
According to \cite[Prop. 2.2.2]{McDuffSiegel_counting}, the count $\gw_{\CP^n,[L]}\lll \T^{n-1}p\rrr = \# \calM^{J,\simp}_{\CP^n,[L]}\lll \T^{n-1}p\rrr$ is independent of $J$ (provided that it is generic), and in fact by \cite[Prop. 3.4]{CM2} we have $\gw_{\CP^n,[L]}\lll \T^{n-1}p\rrr = (n-1)! \neq 0$.
Now consider the compactified moduli space $\ovl{\calM}_{\CP^n,[L]}^{\{J_t\}}\lll \T^{n-1}p\rrr$, and let $\pi$ denote its natural projection to $[0,1]$.
The fiber $\pi^{-1}(0)$ coincides with $\calM_{\CP^n,[L]}^{J,\simp}\lll\T^{n-1}p\rrr$, 
whereas according to Lemma~\ref{lem:cpn_ns_result} the fiber $\pi^{-1}(1)$ consists of two-level pseudoholomorphic buildings with regular components.
 In particular, by standard gluing along cylindrical ends (see e.g. \cite[Thm. 2.54]{pardon2019contact}), $\ovl{\calM}_{\CP^n,[L]}^{\{J_t\}}\lll \T^{n-1}p\rrr$ defines a one-dimensional oriented topological cobordism, at least after restricting the family to $[1-\delta,1)$ for $\delta > 0$ sufficiently small.
Using again \cite[Prop. 2.2.2]{McDuffSiegel_counting}, the counts $\# \pi^{-1}(0)$ and $\# \pi^{-1}(1-\delta)$ coincide, and hence by counting signed boundary points we obtain the relation
  \begin{align*}\label{eqn:neck_stretch}\tag{*}
(n-1)! = \#\calM^{J_X}_{X_{\vec{d}}^{2n}}(\beta_1^{\times d_1},\dots,\beta_k^{\times d_k})\lll \T^{n-1}p \rrr \cdot \prod_{i=1}^k \left(\#\calM^{J_N}_{N_{\vec{d}}^{2n},[c_i]}(\beta_i)\right)^{d_i}.
\end{align*}
In particular, we have $\#\calM^{J_X}_{X_{\vec{d}}^{2n}}(\beta_1^{\times d_1},\dots,\beta_k^{\times d_k})\lll \T^{n-1} p\rrr \neq 0$, as well as 
${\#\calM^{J_N}_{N_{\vec{d}}^{2n},[c_i]}(\beta_i) \neq 0}$ for $i = 1,\dots,k$, 
and all of these counts must be finite.

Finally, suppose that we have another generic admissible almost complex structure $J_X'$ which coincides with $J_X$ on a neighborhood of the cylindrical end of $\wh{X}_{\vec{d}}^{2n}$.
Let $J'$ be the compatible almost complex structure on $\CP^n$ which restricts to $J_{X'}$ and $J_N$ on $X_{\vec{d}}^{2n}$ and $N_{\vec{d}}^{2n}$ respectively.
Since we have also 
\begin{align}
\# \calM^{J',\simp}_{\CP^n,[L]}\lll \T^{n-1}p\rrr = \gw_{\CP^n,[L]}\lll \T^{n-1}p\rrr,
\end{align}
 by comparing \eqref{eqn:neck_stretch} with the analogous relation using $J'$ instead of $J$, we must have
\begin{align}
\#\calM^{J_X}_{X_{\vec{d}}^{2n}}(\beta_1^{\times d_1},\dots,\beta_k^{\times d_k})\lll \T^{n-1}p  \rrr = \#\calM^{J_X'}_{X_{\vec{d}}^{2n}}(\beta_1^{\times d_1},\dots,\beta_k^{\times d_k})\lll \T^{n-1} p\rrr.
\end{align}

\end{proof}

\subsection{Completing the obstructions proof}\label{subsec:completing}

We now complete the proof of Theorem~\ref{thm:main_combinatorial}.
Fix $n \in \Z_{\geq 1}$ and tuples of positive integers $\vec{d}=(d_1,\dots,d_k) \in \calS$ and $\vec{d'} = (d_1',\dots,d_{k'}') \in \calS$ with $\sum_{i=1}^k d_i, \sum_{i=1}^{k'}d_i' \geq n+1$, put $X := X_{\vec{d}}^{2n}$ and $X' := X_{\vec{d'}}^{2n}$, 
and consider a hypothetical Liouville embedding $\iota: X \es X'$. 
Let $\la$ and $\la'$ denote the preferred Liouville one-forms on $X$ and $X'$ respectively provided by Theorem~\ref{thm:div_compl1}. 
By Lemma~\ref{lem:liouville_emb_lemmas} (b), after applying a Liouville homotopy to $X'$ we can assume that we have $\iota^*\la' = \la$.
Note that since this homotopy only modifies the contact form on $\bdy X'$ by a positive scaling factor, the Reeb dynamics are unaffected (expect for possibly rescaling the periods of all Reeb orbits by a fixed constant), so all of our results about the geometry of $X'$ still apply.

Let $J'$ be a generic admissible almost complex structure on the symplectic completion of $X'$.  After deforming $J'$, we can assume that $\iota^*(J')$ is the restriction of a generic admissible almost complex structure $J$ on the symplectic completion of $X$. In particular, $J'$ is cylindrical near $Y := \iota(\bdy X)$.
Now let $\{J_t'\}_{t \in [0,1)} \in \calJ_{X'}$ be a corresponding neck stretching family of admissible complex structures on the symplectic completion of $X$, with $J_0' = J'$ and $J_t'$ limiting as $t \rightarrow 1$ to a broken almost complex structure on $X \circledcirc (X' \setminus \iota(X))$.
We consider the compactification $\ovl{\calM}^{\{J_t'\}}_{X'}\lll \T^{n-1}p\rrr(\beta_1^{\times d_1'},\dots,\beta_k^{\times d_k'})$ provided by the SFT compactness theorem, and let $\pi$ denote the natural projection to $[0,1]$.
\begin{lemma}
In the situation above, the fiber $\pi^{-1}(1)$ is nonempty.
\end{lemma}
\begin{proof}
Observe that the fiber $\pi^{-1}(t)$ is nonempty for any $t \in [0,1)$.
Indeed, if $\pi^{-1}(t)$ were empty, then in particular the moduli space $\calM_{X'}^{J_t}\lll \T^{n-1}p\rrr(\beta_1^{\times d_1'},\dots,\beta_k^{\times d_k'})$ would be empty, and hence trivially regular, contradicting Proposition~\ref{prop:cnt_fin_and_well_def}.
It follows then by compactness of $\ovl{\calM}^{\{J_t'\}}_{X'}\lll \T^{n-1}p\rrr(\beta_1^{\times d_1'},\dots,\beta_k^{\times d_k'})$ that $\pi^{-1}(1)$ is nonempty.
\end{proof}

We now consider a configuration in $\pi^{-1}(1)$ and use its existence to read off various consequences.  A priori, we have a pseudoholomorphic building with
\begin{itemize}
	\item some number (possibly zero) of levels in the symplectization $\R \times \bdy X'$
	\item a level in the cobordism $X' \setminus \iota(X)$
	\item some number (possibly zero) of levels in the symplectization $\R \times \bdy X$
	\item bottom level in the domain $X$ consisting of one ``main'' component inheriting the constraint $\lll \T^{n-1}p\rrr$, along with some number (possibly zero) of unconstrained components.
\end{itemize}
By formally gluing together all pairs of ends except for those corresponding to the positive ends of the main component in the bottom level, we arrive at the following simplified picture:
\begin{itemize}
	\item a level in the cobordism $X' \setminus \iota(X)$ consisting of some number $l \geq 1$ of formal components, anchored in $X$
	\item bottom level consisting of a main component $u$ with $l$ positive ends and satisfying the constraint $\lll \T^{n-1}p\rrr$.
\end{itemize}
Note that main component is pseudoholomorphic with respect to the generic admissible almost complex structure $J$ on $X$, and so by Lemma~\ref{lem:main_cpnt_nonneg_ind} we must have $\ind(u) \geq 0$. Therefore, by Lemma~\ref{lem:too_few_ends} we must have $l \geq \sum_{i=1}^kd_i$.
Since each component in the cobordism level has at least one positive end, we also must have $l \leq \sum_{i=1}^{k'}d_i'$.

Let $\gamma_{\vec{x}_1}^{A_1},\dots,\gamma_{\vec{x}_l}^{A_l}$ be the positive ends of $u$.
Then for $i = 1,\dots,l$, the $i$th formal component in the cobordism level has negative end $\gamma_{\vec{x}_i}^{A_i}$, and its positive ends form a nonempty subcollection of $\underbrace{\beta_1,\dots,\beta_1}_{d_1},\dots,\underbrace{\beta_k,\dots,\beta_k}_{d_k}$. Each of these positive ends corresponds to one of the unit basis vectors $e_1,\dots,e_k$. Let $\vec{y}_i \in \Z_{\geq 0}^{k'} \setminus \{\vec{0}\}$ denote the sum of these unit basis vectors.
Note that by construction we have $\sum_{i=1}^l \vec{y}_i = \vec{d'}$.

We now consider the map $\iota_*: H_1(X) \rightarrow H_1(X')$ induced by the Liouville embedding $\iota: X \hookrightarrow X'$.
Under the identification from \S\ref{subsec:hypersurf_compl}, this is naturally viewed as a group homomorphism $\Phi: \Z^k/(\vec{d}) \rightarrow \Z^{k'}/(\vec{d'})$, and as such it sends $\vec{x}_i \text{ mod } (\vec{d})$ to $\vec{y}_i\text{ mod }(\vec{d'})$ for $i = 1,\dots,l$.
Since $u$ provides a nulhomology of $\sum_{i=1}^l[\gamma_{\vec{x}_i}^{A_i}] \in H_1(X)$, we must have $\sum_{i=1}^l\vec{x}_i = q\vec{d}$ for some $q \in \Z_{\geq 1}$.

Finally, nonnegativity of the index of $u$ translates into
\begin{align*}
0 &\leq (n-3)(2-l) - (4n-4) + \sum_{i=1}^l \cz_{\triv}(\gamma_{\vec{x}_i}^{A_i}) + 2c_1^{\triv}(u)\\
&= (n-3)(2-l) - (4n-4) + \sum_{i=1}^l \left( \delta(\gamma_{\vec{x}_i}^{A_i}) - 2\vec{x}_i \cdot \vec{1} \right) + 2q(n+1)\\
&\leq (n-3)(2-l) - (4n-4) + l(n-1) - 2q\sum_{i=1}^k d_i + 2q(n+1)\\
&= -2n + 2l - 2 - 2q(\sum_{i=1}^k d_i - n -1),
\end{align*}
i.e. $q(\sum_{i=1}^k d_i - n - 1) \leq l - n - 1.$
This completes the proof of Theorem~\ref{thm:main_combinatorial}.

\begin{remark}
In the above neck stretching argument, it is not too difficult to show that any configuration in $\pi^{-1}(t)$ with $t \in [0,1)$ consists entirely of simple components, and hence can be assumed to be regular, meaning that $\pi^{-1}([0,1))$ is a one-dimensional topological manifold with boundary. 
However, we did not show (or require) that the configurations in $\pi^{-1}(1)$ are transversely cut out, and multiply covered components in the cobordism $X' \setminus \iota(X)$ could be in principle appear. 
Transversality for the whole compactification $\ovl{\calM}^{\{J_t'\}}_{X'}\lll \T^{n-1}p\rrr(\beta_1^{\times d_1'},\dots,\beta_k^{\times d_k'})$ should follow by adapting the Cieliebak--Mohnke framework \cite{CM1,CM2} or a more general virtual perturbation frameworks (c.f. Remark~\ref{rmk:virtual}), leading to a slight strengthening of Theorem~\ref{thm:main_combinatorial}.
\end{remark}

\section{Constructions}\label{sec:constructions}

\subsection{Weinstein cobordisms from degenerations}

In this subsection we prove Theorem~\ref{thm:embeddings} based on the idea that in a degenerating family of divisors there is a Weinstein cobordism from the complement of the special fiber to the complement of the general fiber.
We then complete the proof of Theorem~\ref{thm:main_liouville}.

\begin{proof}[Proof of Theorem~\ref{thm:embeddings}]
Since Weinstein cobordisms can be concatenated, it suffices to consider the case that $\vec{d'}$ is obtained from $\vec{d}$ by either a combination move or a duplication move.
In the former case, put $\vec{d} = (d_1,\dots,d_k)$ and $\vec{d'} = (d_1,\dots,d_{k-2},d_{k-1}+d_{k})$ without loss of generality, and let $D_t$, $t \in [0,1]$, be a smooth family of simple normal crossing divisors in $\CP^n$ such that
\begin{itemize}
\item for $t > 0$, $D_t$ has $k-1$ irreducible components of degrees $d_1,\dots,d_{k-2},d_{k-1}+d_k$ respectively
\item $D_0$ has $k$ irreducible components of degrees $d_1,\dots,d_{k-1},d_k$ respectively.
\end{itemize}
Namely, the last component, which is a smooth hypersurface of degree $d_{k-1}+d_k$, degenerates into a union of two hypersurfaces of degrees $d_{k-1}$ and $d_k$.
Put $\calL := \mathcal{O}(\sum_{i=1}^k d_i)$. Correspondingly, we can find a smooth family of holomorphic sections $\sigma_t \in H^0(\CP^n;\calL)$, i.e. degree $\sum_{i=1}^kd_i$ homogeneous polynomials in $\C[X_0,\dots,X_n]$,
such that $D_t = \sigma_t^{-1}(0)$ for $t \in [0,1]$.
Then the existence of a Weinstein embedding of $X_{\vec{d}}^{2n}$ into $X_{\vec{d'}}^{2n}$ follows from Proposition~\ref{prop:cob_gen_proj_var} below.

Similarly, suppose now that $\vec{d'}$ differs from $\vec{d}$ by a duplication move, and put $\vec{d} = (d_1,\dots,d_k)$ and $\vec{d'} = (d_1,\dots,d_k,d_k)$ without loss of generality.
Let $D_t$, $t \in [0,1]$, be a smooth family of simple normal crossing divisors in $\CP^n$ such that
\begin{itemize}
\item for $t > 0$, $D_t$ has $k+1$ irreducible components of degrees $d_1,\dots,d_k,d_k$ respectively
\item $D_0$ has $k$ irreducible components of degrees $d_1,\dots,d_k$ respectively.
\end{itemize}
Note here that, by \cite[Lem. 4.4]{Seidel_biased_view}, considering a divisor component with multiplicity does not change the Weinstein structure on the complement up to Weinstein homotopy.
Namely, the last two components, which are smooth hypersurfaces of degree $d_k$, degenerate into a single hypersurface of degree $d_k$ (but with multiplicity two).
Put $\calL := \mathcal{O}(\sum_{i=1}^k d_i + d_k)$,
and let  $\sigma_t \in H^0(\CP^n;\calL)$ be smooth family of holomorphic sections such that $D_t = \sigma_t^{-1}(0)$ for $t \in [0,1]$.
Then again the existence of a Weinstein embedding of $X_{\vec{d}}^{2n}$ into $X_{\vec{d'}}^{2n}$ follows from Proposition~\ref{prop:cob_gen_proj_var} below.

\end{proof}

Recall from \S\ref{subsec:div_compl} that if $M$ is a smooth complex projective variety and $D \subset M$ is an ample simple normal crossing divisor, then $M \setminus \Op(D)$ is canonically a Weinstein domain up to Weinstein deformation equivalence.

\begin{prop}\label{prop:cob_gen_proj_var}
Let $M$ be a smooth complex projective variety, let $\calL \rightarrow M$ be an ample line bundle,
and let $\sigma_t \in H^0(M;\calL)$, $t \in [0,1]$ be a smooth family of holomorphic sections such that $D_t := \sigma_t^{-1}(0)$ is a simple normal crossing divisor for each $t \in [0,1]$.
Then for $\delta > 0$ sufficiently small, there is a Weinstein embedding of $M \setminus \Op(D_0)$ into $M \setminus \Op(D_\delta)$.
\end{prop}

\begin{proof}
Pick a Hermitian metric $\langle -,-\rangle$ on $\calL$, with associated norm $||-||$.
For $t \in [0,1]$, put $\phi_t := -\log ||\sigma_t||$.
As in \S\ref{subsec:div_compl}, the function $\phi_t: M \setminus D_t \rightarrow \R$ is exhausting and strictly plurisubharmonic for all $t \in [0,1]$, with critical points contained in a compact subset.
After a small perturbation, we can further assume that $\phi_t$ is generalized Morse for all $t \in [0,1]$. 
Put $N := \phi_0^{-1}((R,\infty)) \cup D_0$ for $R$ sufficiently large,
 so that we have $D_0 \subset N$ and all of the critical points of $\phi_0$ lie in $M \setminus \ovl{N}$.
 Pick $\delta > 0$ sufficiently small such that $R$ is a regular value of $\phi_t$ for all $t \in [0,\delta]$. 
Put $U := \phi_\delta^{-1}((S,\infty)) \cup D_\delta$ for $S$ sufficiently large, so that we have $D_\delta \subset U \subset N$ and none of the critical points of $\phi_\delta$ lie in $\ovl{U}$.
Put $W_t := \phi_t^{-1}((-\infty,R])$ for $t \in [0,\delta]$ and $W' := \phi_\delta^{-1}((-\infty,S])$.
Then $(W_0,-d^{\C}\phi_0|_{W_0},\phi_0|_{W_0})$ and $(W',-d^{\C}\phi_\delta|_{W'},\phi_{\delta}|_{W'})$ are Weinstein domains which represent the natural Weinstein structures on $M \setminus \Op(D_0)$ and $M \setminus \Op(D_\delta)$ respectively up to Weinstein deformation equivalence.
Moreover, the former is Weinstein deformation equivalent to the Weinstein subdomain of the latter corresponding to $\{\phi_\delta \leq R\}$. Indeed, the family of Weinstein domains $(W_t,-d^\C\phi_t|_{W_t},\phi_t|_{W_t})$ for $t \in [0,\delta]$ induces the desired Weinstein deformation equivalence.

\end{proof}

\sss

\begin{proof}[Proof of Theorem~\ref{thm:main_liouville}]
The ``if'' part is immediate from Theorem~\ref{thm:embeddings}, so it suffices to prove the ``only if'' statement,
i.e. given a Liouville embedding $X_{\vec{d}}^{2n} \es X_{\vec{d'}}^{2n}$ we must have $\vec{d} \leqq \vec{d'}$.
By Theorem~\ref{thm:main_combinatorial} we have 
\begin{align*}
q \leq \dfrac{l-n-1}{\sum_{i=1}^k d_i - n -1} \leq \dfrac{\sum_{i=1}^{k'}d_i' - n - 1}{\sum_{i=1}^k d_i - n - 1} < \dfrac{2 \sum_{i=1}^k d_i - 2n - 2}{\sum_{i=1}^k d_i - n - 1} = 2,
\end{align*}
and hence $q = 1$.
We therefore have
\begin{align*}
\vec{d} \cdot \vec{1} \leq l \leq \sum_{i=1}^l \vec{x}_i \cdot \vec{1} = \vec{d} \cdot \vec{1},
\end{align*}
which forces $l = \vec{d} \cdot \vec{1}$ and $\vec{x}_i \cdot \vec{1} = 1$ for $i = 1,\dots,l$.
Therefore, after possibly reordering, we can assume that we have
\begin{align*}
\vec{x}_1,\dots,\vec{x}_l = \underbrace{\vec{e}_1,\dots,\vec{e}_1}_{d_1},\dots,\underbrace{\vec{e}_k,\dots,\vec{e}_k}_{d_k}.
\end{align*}
Note that for two equal tuples $\vec{x}_i,\vec{x}_j$, the corresponding elements $\vec{y}_i,\vec{y}_j$ must have equal residues in $\Z^{k'}/(\vec{d'})$, and hence must simply be equal.
We therefore have
$$ d_1\vec{z}_1 + \dots + d_k\vec{z}_k = \vec{d'}$$
for some tuples $\vec{z}_1,\dots,\vec{z}_k \in \Z_{\geq 0}^{k'} \setminus \{\vec{0}\}$.
To see that $\vec{d} \leqq \vec{d'}$, observe that 
we have
\begin{align*}
(d_1,\dots,d_k) \leqq (\underbrace{d_1,\dots,d_1}_{\vec{z}_1 \cdot \vec{1}},\dots,\underbrace{d_k,\dots,d_k}_{\vec{z}_k\cdot\vec{1}}) \leqq (d_1',\dots,d_{k'}'),
\end{align*}
where the first inequality comes from iteratively applying duplication moves and the second inequality comes from iteratively applying combination moves.
\end{proof}

\subsection{Flexible constructions}

As already pointed out, the main obstructions in this paper are of an exact symplectic nature, i.e. they obstruct exact symplectic embeddings in situations where symplectic embeddings (and in particular formal symplectic embeddings) do exist.
Still, since the divisor complements $X_{\vec{d}}^{2n}$ have fairly nontrivial smooth topology, it is natural to wonder what role (if any) the topology plays.
It turns out that there is a fair amount of freedom to modify the diffeomorphism type without invalidating the obstructions, although first homology groups appear to play an essential role.

We begin with a definition:
\begin{definition}
Two Liouville domains $X,X'$ of the same dimension are {\em Liouville (resp. symplectic) embedding equivalent} if there is a Liouville (resp. symplectic) embedding of $X$ into $X'$ and of $X'$ into $X$.
\end{definition}
\NI For instance, if $X$ and $X'$ are Liouville domains which are Liouville embedding equivalent, and if $X''$ is a third Liouville domain, then we have $X \es X''$ if and only if $X' \es X''$.
In an attempt to remove all smooth topology from the discussion, it is natural to ask whether any Liouville domain $X$ is Liouville embedding equivalent to another domain which is diffeomorphic to the ball. This is easily seen to be false. 
Concretely, suppose by contradiction that $X_{n+1}^{2n} \cong D^*\mathbb{T}^n$ were Liouville embedding equivalent to a Liouville domain $Q^{2n}$ which is diffeomorphic to the ball.
Since $X_{n+1}^{2n}$ symplectically embeds into the standard Liouville $\C^n$, we also have $Q \shookrightarrow \C^n$, and hence $Q \es \C^n$ (since $H^1(Q;\R) = 0$), whence $X_{n+1}^{2n} \es \C^n$, which is a contradiction.

Now suppose $X$ is a Weinstein domain, and $X'$ is obtained from $X$ by Weinstein handle attachments.
Then we have $X \ws X'$, so by monotonicity we have $\G\lll \T^m p \rrr(X) \leq \G\lll \T^m p\rrr(X')$ and more generally $\I^{\leq l}(X) \supset \I^{\leq l}(X')$ for any $m \in \Z_{\geq 0}$ and $l \in \Z_{\geq 1}$.
In the special case that $X'$ differs from $X$ by subcritical or flexible Weinstein handle attachments, a well-known metaprinciple states that ``all'' pseudoholomorphic curve invariants of $X$ and $X'$ should coincide (see e.g. \cite{handleattachinginSH,fauck2016handle,bourgeois2012effect,murphy2018subflexible} for the case of symplectic cohomology).
Accordingly, we conjecture that $\I^{\leq l}(X)$ and in particular $\G\lll\T^m p \rrr$ are invariant under subcritical and flexible Weinstein handle attachments. 
In particular, this would give a large amount of freedom to change the smooth topology of $X$ without affecting these Liouville embedding obstructions (although one cannot kill the homology in the critical dimension $\tfrac{1}{2}\dim(X)$ by Weinstein handles, since these have index at most $\tfrac{1}{2}\dim(X)$).
Note that the above discussion shows that the Liouville embedding type is {\em not} generally invariant under subcritical handle attachments, since it is possible to kill the fundamental group by adding Weinstein two-handles.

Using somewhat more geometric considerations, we have:
\begin{prop}\label{cor:flex_constr}
Fix $n \geq 3$ and $N \in \Z_{\geq 1}$. For each $k \in \Z_{\geq 1}$ and $\vec{d} \in \Z_{\geq 1}^k$ with $\sum_{i=1}^k d_i \leq N$ and such that $1$ is an entry of $\vec{d}$, 
there is a Weinstein domain $\wt{X}_{\vec{d}}^{2n}$ such that
\begin{itemize}
\item
there is a Weinstein embedding $X_{\vec{d}}^{2n} \ws \wt{X}_{\vec{d}}^{2n}$ and a Liouville embedding $\wt{X}_{\vec{d}}^{2n} \es X_{\vec{d}}^{2n}$
\item
$\wt{X}_{\vec{d}}^{2n}$ is almost symplectomorphic to $\wt{X}_{\vec{d'}}^{2n}$ for any such $\vec{d},\vec{d'}$.
\end{itemize}
In particular, Theorem~\ref{thm:main_liouville} still holds if we replace each $X_{\vec{d}}^{2n}$ with the corresponding $\wt{X}_{\vec{d}}^{2n}$.
\end{prop}
\NI The point here is that all of our Liouville embedding obstructions apply equally if we replace $X^{2n}_{\vec{d}}$ with $\wt{X}^{2n}_{\vec{d}}$ (since these are Liouville embedding equivalent), and yet the almost symplectomorphism type of $\wt{X}^{2n}_{\vec{d}}$ is independent of $\vec{d}$, meaning that these obstructions must be purely symplectic in nature (i.e. there are no smooth embedding obstructions).
\begin{proof}
Observe that we can find some (very large) unordered tuple $\vec{f} \in \calS$ (recall \S\ref{subsubsec:constructions}) such that $\vec{d} \leqq \vec{f}$ for each $\vec{d} \in \calS$ with $\sum_{i=1}^k d_i \leq N$ having $1$ as an entry.
More specifically, we assume that $\vec{f}$ contains each $\vec{d}$ with $\vec{d} \cdot \vec{1} \leq N$ which contains $1$ as an entry as a subtuple (up to reordering).
For each such $\vec{d}$, we consider the Weinstein embedding $X_{\vec{d}}^{2n} \ws X_{\vec{f}}^{2n}$ provided by Theorem~\ref{thm:embeddings}.
Note that this presents $X_{\vec{f}}^{2n}$ as the result after concatenating $X_{\vec{d}}^{2n}$ with a Weinstein cobordism $W$.
We define $\wt{X}_{\vec{d}}^{2n}$ to be the Weinstein domain obtained by concatenating $X_{\vec{d}}^{2n}$ with the flexibilization $\flex(W)$ of the Weinstein cobordism $W$ (see \cite[\S11.8]{cieliebak2012stein}).
Then $\wt{X}_{\vec{d}}^{2n}$ is almost symplectomorphic to $X_{\vec{f}}^{2n}$, and evidently there is a Weinstein embedding of
$X_{\vec{d}}^{2n}$ into $\wt{X}_{\vec{d}}^{2n}$.

To see that there is a Liouville embedding $\wt{X}_{\vec{d}}^{2n} \es X_{\vec{d}}^{2n}$, note that by our assumption on $\vec{f}$ there is a symplectic embedding $X_{\vec{f}}^{2n} \shookrightarrow X_{\vec{d}}^{2n}$ (up to Weinstein homotopy), given by adding back in some of the divisor components.
Now consider the result after concatenating the Weinstein embedding $X_{\vec{d}}^{2n} \ws X_{\vec{f}}^{2n}$ with the above symplectic embedding $X_{\vec{f}}^{2n} \shookrightarrow X_{\vec{d}}^{2n}$. 
This can be identified with the result after attaching to $X_{\vec{d}}^{2n}$ a finite piece of the symplectization of $\bdy X_{\vec{d}}^{2n}$.
We see that there are no formal obstructions to extending this to a Liouville embedding $X_{\vec{f}}^{2n} \es X_{\vec{d}}^{2n}$, so there are also no formal obstructions to extending it to a Liouville embedding $\wt{X}_{\vec{d}}^{2n} \es X_{\vec{d}}^{2n}$.
Since $\wt{X}_{\vec{d}}^{2n}$ is obtained from $X_{\vec{d}}^{2n}$ by attaching subcritical and flexible handles, 
we can apply the Lagrangian caps h principle \cite{eliashberg2013lagrangian} to produce a genuine Liouville embedding $\wt{X}_{\vec{d}}^{2n} \es X_{\vec{d}}^{2n}$.
\end{proof}

\begin{remark}\label{rmk:Liouville_vs_Weinstein_flex}
We do not know of any example of a Liouville domain which is diffeomorphic to a Weinstein domain but not symplectomorphic to any Weinstein domain. Still, Liouville embeddings between Weinstein domains are often more flexible than Weinstein embeddings.
For example, any $n \geq 3$, $\flex(D^*\mathbb{T}^n)$ admits a Liouville embedding into $\C^n$ (see \cite[Cor. 6.3]{eliashberg2013lagrangian}), but it does not admit a Weinstein embedding for (ordinary) homological reasons (i.e. Weinstein handle attachments cannot shrink the rank of the $n$th homology group).
More generally, if $X$ and $X'$ are Weinstein domains with $X \ws X'$, then the boundary connect sum of $X$ with sufficiently many copies of $\flex(D^*\mathbb{T}^n)$ still Liouville embeds into $X'$ (again using \cite{eliashberg2013lagrangian}), whereas a Weinstein embedding is impossible for homological reasons.
\end{remark}

\begin{remark}
If $X$ is a Weinstein domain, it is interesting to ask whether the invariant $\I^{\leq l}(X)$ can be computed from the wrapped Fukaya category $\wfuk(X)$ of $X$. 
By \cite{ganatra2019cyclic} together with generation results from \cite{chantraine2017geometric,ganatra2018structural}, the ``cyclic open-closed map'' gives an isomorphism from the cyclic homology of $\wfuk(X)$ to $\sh_{S^1}(X)$. After taking into account additional naturality properties it should be possible to recover the invariant $\F(X)$ from \S\ref{subsec:obs_cyls}.
However, it seems unlikely that we can recover the full invariant $\I^{\leq l}(X)$ solely from the $\mathcal{A}_\infty$ category $\wfuk(X)$ without any additional structure.
Rather, the $\Li$ structure on $\sc_{S^1}$ should be equivalent to an $\Li$ structure on cyclic chains of $\wfuk(X)$ which depends on a smooth Calabi--Yau structure on $\wfuk(X)$ (see e.g. \cite{chen2020gravity} for a somewhat analogous setup), the existence of which is guaranteed by  \cite{ganatra2019cyclic}. Under the cyclic open-closed map, we expect that this data is sufficient to recover the $\Li$ homomorphism ${P_0 \circ \delta_{S^1}: \sc^*_{S^1,+}(X) \rightarrow \K[u^{-1}]}$ (c.f. \S\ref{subsec:obs_higher}), which in turn determines an invariant which is closely analogous (and conjecturally equivalent up to an isomorphism of $\ovl{S}\K[u^{-1}]$) to $\I^{\leq l}(X)$.
Note that subcritical handle attachment does not change the wrapped Fukaya category (see \cite[\S1.7]{ganatra2018structural}), so this approach could potentially be used to determine the effect on $\I^{\leq l}(X)$ of subcritical handle attachment.
\end{remark}

\section{Possible extensions}\label{sec:concl}

In this brief conclusion we sample a few interesting directions for further research.

\subsubsection*{Allowing self-crossings}

One fairly mild generalization of Problem~\ref{prob:hyp} is to let $D$ be a normal crossing divisor which is not necessarily simple, i.e. the irreducible components are allowed to have self intersections. For example, the log Calabi--Yau surfaces studied in \cite{pascaleff2019symplectic} can all be viewed as complements of degree three curves in $\CP^2$ which are not necessarily smooth or irreducible. 
Although these examples do not quite fit into the framework of Theorem~\ref{thm:div_compl1}, it seems natural to expect the theorem to generalize to this case in light of \cite{tehrani2018normal}. 

More ambitiously, one could consider e.g. complements of (nongeneric) hyperplane arrangements in $\C^n$.
In this case it is interesting to ask to what extent the invariant $\I^{\leq l}$ is sensitive to the intersection poset of the arrangement.

\subsubsection*{More general hypersurface singularities}

It is also natural to consider divisors with more general hypersurface singularities. 
For instance, complements in $\C^n$ of vanishing loci of Brieskorn polynomials $P(z_1,\dots,z_n) = z_1^{a_1} + \dots + z_n^{a_n} \in \C[z_1,\dots,z_n]$ for $a_1,\dots a_n \in \Z_{\geq 2}$
play an essential role in McLean's constructions of symplectically exotic affine spaces \cite{mcleanlefschetz} (see also \cite[\S4b]{abouzaid2010altering}).
Concretely, if we take $P = z_1^2 + \dots + z_{n-1}^2 + z_n^3$ for $n \geq 4$ even or $P = z_1^2 + \dots + z_{n-2}^2 + z_{n-1}^3 + z_n^5$ for $n \geq 3$ odd,
then, after attaching a (subcritical) Weinstein two-handle, $\C^n \setminus P^{-1}(0)$ becomes diffeomorphic but not symplectomorphic to $\C^n$.
Furthermore, by counting idempotents in the symplectic cohomology algebra over $\Z/2$, McLean shows that $V_1,V_2,V_3,\dots$ are pairwise nonsymplectomorphic, where $V_k$ denotes the boundary connect sum of $k$ copies of $V$.
Evidently we have Weinstein embeddings $V_k \ws V_{k'}$ for $k \leq k'$, and the above suggests that $V_{k'}$ might be ``larger'' than $V_k$:
\begin{question}
Is there a Liouville embedding $V_k \es V_{k'}$ for $k > k'$?
\end{question}
\NI A starting point would be to compute $\G\lll \T^mp\rrr(\C^n \setminus P^{-1}(0))$ for $m \in \Z_{\geq 0}$, where $P(z_1,\dots,z_n)$ is a Brieskorn polynomial as above.
More broadly, what values can the invariant $\G\lll \T^m p \rrr$ assume on symplectically exotic affine spaces in a given dimension?

\subsubsection*{Other ambient spaces}
We are also interested in divisor complements in more general smooth projective varieties.
In fact, recall that every smooth complex affine algebraic variety is the complement of an ample simple normal crossing divisor in a smooth projective variety as a consequence of Hironaka's resolution of singularities theorem.
The neck stretching strategy in \S\ref{sec:computationsI},\S\ref{sec:computationsII} could plausibly extend to compute the invariant $\G\lll\T^mp\rrr$ for divisor complements in other smooth projective varieties $M$, after replacing $\gw_{\CP^n,[L]}\lll \T^{n-1}p\rrr$ with an appropriate invariant $\gw_{M,A}\lll \T^m p \rrr$ which is nonzero for some $A \in H_2(M)$ and $m \in \Z_{\geq 0}$.
At least in the case $n = 2$, the counts $\gw_{M,A}\lll \T^m p\rrr$ can be computed by the algorithm in \cite{McDuffSiegel_counting}. For example, for $M = \CP^1 \times \CP^1$ we have $\gw_{\CP^1 \times \CP^1,A}\lll \T^{2d}p\rrr = 1 \neq 0$ when $A = d[L_1] + [L_2]$ (i.e. we are counting curves of bidegree $(d,1)$) for $d \in \Z_{\geq 0}$.

\subsubsection*{Higher genus analogues}

On the other hand, there are also examples having no apparent rational curves whatsoever.
For instance, suppose that $M$ is a product of $n$ closed Riemann surfaces, each having genus at least one.
In this case we have $\pi_2(M) = 0$, and hence $\gw_{M,A}\lll\T^{m}p\rrr = 0$ for all $m \in \Z_{\geq 0}$ and all $A \in H_2(M)$.
In fact, if $D \subset M$ is an ample simple normal crossing divisor and $X = M \setminus \Op(D)$ denotes the corresponding divisor complement (endowed with a contact form as in Theorem~\ref{thm:div_compl1}), then there cannot be any nonconstant asymptotically cylindrical pseudoholomorphic curves of genus zero in $X$.
Note that this holds by purely topological considerations, i.e. such a curve could be capped off to give a homologically essential map $S^2 \rightarrow M$, which is a contradiction.
In particular, the obstructions defined in this paper are vacuous in this case. It is possible that analogous invariants defined using higher genus curves could provide more refined obstructions.

As a special case, we consider the Liouville domains given by products of Riemann surfaces with boundary. Note that there are obvious symplectic embeddings given by filling in some of the boundary components.
In dimension $2n=2$, Example~\ref{ex:surfaces} provides complete Liouville embedding obstructions by purely elementary considerations, but this does not extend to higher dimensions.

\subsubsection*{Cotangent bundles}
Another interesting direction is to compute $\I^{\leq l}$ for cotangent bundles $T^*Q$ of closed smooth manifolds $Q$.
This case has implications for (exact) Lagrangian embeddings.
As a starting point, recall that if $Q$ admits a Riemannian metric of negative sectional curvature, then the unit sphere bundle $S^*Q$ admits a corresponding contact form all of whose Reeb orbits have Conley--Zehnder index zero.
For $\dim(Q) \geq 3$, it follows that any formal curve in $D^*Q$ without constraints has nonpositive index (c.f. \cite[Cor. 1.7.4]{EGH2000}), and hence has negative index after imposing a point constraint, so we have $\G\lll\T^m p\rrr(D^*Q) = \infty$ for any $m \in \Z_{\geq 0}$, and in fact:
\begin{prop}
Let $X^{2n \geq 6}$ be a Liouville domain such that $\G\lll \T^mp\rrr(X) < \infty$ for some $m \in \Z_{\geq 0}$.
Then $X$ does not have any embedded Lagrangian submanifold which admits a Riemannian metric of negative sectional curvature.
\end{prop}
\NI In light of the discussion in \S\ref{sec:invariants} and the relationship between symplectic cohomology and string topology (see e.g. \cite{abouzaid2013symplectic,cohencalabi}), the invariants defined in this paper for cotangent bundles should be computable using techniques from rational homotopy theory.

\bibliographystyle{math}

\bibliography{liouvillecomplexity}

\end{document}